\documentclass[11pt]{amsart}
\usepackage{lmodern}
\usepackage{amsmath, amsthm, amssymb, amsfonts}
\usepackage[normalem]{ulem}
\usepackage{hyperref}

\usepackage{verbatim} 
\usepackage{longtable}

\usepackage{mathtools}

\usepackage{tikz}
\usetikzlibrary{decorations.pathmorphing}
\tikzset{snake it/.style={decorate, decoration=snake}}

\usepackage{caption}

\usepackage{tikz-cd}
\usetikzlibrary{arrows}

\theoremstyle{plain}
\newtheorem{thm}{Theorem}[section]
\newtheorem{cor}[thm]{Corollary}
\newtheorem{lem}[thm]{Lemma}
\newtheorem{prop}[thm]{Proposition}
\newtheorem{conj}[thm]{Conjecture}
\newtheorem{question}[thm]{Question}

\theoremstyle{definition}
\newtheorem{defn}[thm]{Definition}

\theoremstyle{remark}
\newtheorem{rmk}[thm]{Remark}

\newcommand{\BA}{{\mathbb{A}}}

\newcommand{\BC}{{\mathbb{C}}}

\newcommand{\BG}{{\mathbb{G}}}

\newcommand{\BL}{{\mathbb{L}}}

\newcommand{\BN}{{\mathbb{N}}}

\newcommand{\BQ}{{\mathbb{Q}}}

\newcommand{\BZ}{{\mathbb{Z}}}

\newcommand{\CA}{{\mathcal A}}

\newcommand{\CC}{{\mathcal C}}

\newcommand{\CE}{{\mathcal E}}
\newcommand{\CF}{{\mathcal F}}

\newcommand{\CH}{{\mathcal H}}

\newcommand{\CK}{{\mathcal K}}
\newcommand{\CL}{{\mathcal L}}
\newcommand{\CM}{{\mathcal M}}
\newcommand{\CN}{{\mathcal N}}
\newcommand{\CO}{{\mathcal O}}

\newcommand{\BBC}{{\underline{\BC}}}

\newcommand{\heart}{\ensuremath\heartsuit}

\newcommand{\sslash}{\mathbin{/\mkern-6mu/}}

\DeclareFontFamily{OT1}{rsfs}{}
\DeclareFontShape{OT1}{rsfs}{n}{it}{<-> rsfs10}{}
\DeclareMathAlphabet{\curly}{OT1}{rsfs}{n}{it}

\usepackage{tikz}
\usepackage{lmodern}
\usetikzlibrary{decorations.pathmorphing}

\begin{document}
\title[Endoscopic decompositions and the Hausel--Thaddeus conjecture]{Endoscopic decompositions and the Hausel--Thaddeus conjecture}
\date{\today}

\author[D. Maulik]{Davesh Maulik}
\address{Massachusetts Institute of Technology}
\email{maulik@mit.edu}

\author[J. Shen]{Junliang Shen}
\address{Massachusetts Institute of Technology}
\email{jlshen@mit.edu}
\address{Yale University}
\email{junliang.shen@yale.edu}

\begin{abstract}
We construct natural operators connecting the cohomology of the moduli spaces of stable Higgs bundles with different ranks and genera which, after numerical specialization, recover the topological mirror symmetry conjecture of Hausel--Thaddeus concerning $\mathrm{SL}_n$- and $\mathrm{PGL}_n$-Higgs bundles. This provides a complete description of the cohomology of the moduli space of stable $\mathrm{SL}_n$-Higgs bundles in terms of the tautological classes, and gives a new proof of the Hausel--Thaddeus conjecture, proven recently by Gr\"ochenig--Wyss--Ziegler via p-adic integration.

Our method is to relate the decomposition theorem for the Hitchin fibration, using vanishing cycle functors, to the decomposition theorem for the twisted Hitchin fibration whose supports are simpler.

\end{abstract}

\maketitle

\setcounter{tocdepth}{1} 

\tableofcontents
\setcounter{section}{-1}

\section{Introduction}

\subsection{Overview}\label{Sec0.1}
 Throughout, we work over the complex numbers $\BC$. Let $C$ be a nonsingular projective curve of genus $g \geq 2$. Let $n,d$ be integers with $n>0$ and $\mathrm{gcd}(n,d) = 1$. 
 
 The cohomology of the moduli space $\widetilde{\CN}_{n,d}$ of rank $n$ degree $d$ stable vector bundles on $C$ has been studied intensively for decades. By \cite{AB,Beau}, the cohomology $H^*(\widetilde{\CN}_{n,d}, \BC)$ is generated by the \emph{tautological classes} --- the K\"unneth factors of the Chern characters of a universal family. Relations between the tautological classes were explored in \cite{Kir, EK}. 
 
 A natural moduli space closely related to $\widetilde{\CN}_{n,d}$ is the moduli of stable $\mathrm{SL}_n$-bundles
\[
\CN_{n,L} \subset \widetilde{\CN}_{n,d}
\]
which parameterizes rank $n$ stable vector bundles with fixed determinant $L\in \mathrm{Pic}^d(C)$. The finite abelian group
\begin{equation*}
\Gamma = \mathrm{Pic}^0(C)[n]
\end{equation*}
acts on $\CN_{n,L}$ via tensor product, which induces a $\Gamma$-action on the cohomology $H^*(\CN_{n,L}, \BC)$. The $\Gamma$-invariant part $H^*(\CN_{n,L}, \BC)^\Gamma$ recovers the cohomology of the quotient $\CN_{n,L}/\Gamma$, which can be viewed as the moduli space of stable $\mathrm{PGL}_n$-bundles. The tautological classes associated with a universal family generate the $\Gamma$-invariant cohomology $H^*(\CN_{n,L}, \BC)^\Gamma$.

The following theorem by Harder--Narasimhan \cite{HN} shows that every class in  $H^*(\CN_{n,L}, \BC)$ is $\Gamma$-invariant.
 
 \begin{thm}[Harder--Narasimhan \cite{HN}]\label{HNThm}
 The $\Gamma$-action on $H^*(\CN_{n,L}, \BC)$ is trivial.
 \end{thm}
 
 As a consequence of Theorem \ref{HNThm}, we obtain immediately that the tautological classes generate the total cohomology $H^*(\CN_{n,L}, \BC)$.
 
The purpose of this paper is to study the $\Gamma$-action on the cohomology of the moduli space of stable $\mathrm{SL}_n$-Higgs bundles from the viewpoint of the Hausel-Thaddeus conjecture \cite{HT}.  We denote by $\CM_{n,L}$ the (coarse) moduli space parameterizing stable Higgs bundles 
 \[
 (\CE ,\theta: \CE \rightarrow \CE \otimes \Omega_C): \quad\mathrm{det}(\CE) \simeq L , \quad \mathrm{trace}(\theta) = 0
\]
on the curve $C$. It is a nonsingular quasi-projective variety admitting a natural hyper-K\"ahler structure \cite{Hit1,Nit}. As in the case of vector bundles, the group $\Gamma$ acts on $\CM_{n,L}$ via tensor product 
\[
\CL \cdot (\CE, \theta)  = (\CL \otimes \CE, \theta), \quad  \CL \in \Gamma.
\]
The induced $\Gamma$-action on $H^*(\CM_{n,L}, \BC)$ yields the following canonical decomposition
\begin{equation}\label{maindecomp}
H^*(\CM_{n,L}, \BC) = H^*(\CM_{n,L}/\Gamma, \BC) \oplus \bigoplus_{\kappa \neq 0} H^*(\CM_{n,L}, \BC)_\kappa
\end{equation}
where $\kappa \in \hat{\Gamma}=  \mathrm{Hom}(\Gamma, \BC^*)$ runs through all nontrivial characters of $\Gamma$, and $H^*(\CM_{n,L}, \BC)_\kappa$ denotes the $\kappa$-isotypic component. By \cite{Markman} (and the isomorphism \cite[(70)]{dCMS}), the tautological classes associated with a universal family of $\CM_{n,L}$ generate the $\Gamma$-invariant cohomology
\[
H^*(\CM_{n,L}/\Gamma, \BC)  = H^*(\CM_{n,L}, \BC)^\Gamma.
\]
However, contrary to Theorem \ref{HNThm}, the $\Gamma$-variant part of (\ref{maindecomp}) is nontrivial, and carries a rich structure, predicted by topological mirror symmetry \cite{HT}. 



In this paper, we focus on the structure of $H^*(\CM_{n,L}, \BC)_\kappa$ for $\kappa \neq 0$. We introduce natural operators which determine $H^*(\CM_{n,L}, \BC)_\kappa$ in terms of the cohomology of the moduli space of stable $\mathrm{GL}_r$-Higgs bundles on a certain curve for some $r\leq n$. These operators respect the perverse and the Hodge filtrations, and, upon specialization to Hodge polynomials, they recover the Hausel--Thaddeus conjecture \cite{HT}. In particular, this gives a new proof using perverse sheaves of the topological mirror symmetry conjecture of Hausel-Thaddeus, which was recently proven by \cite{GWZ} using p-adic integration.

\subsection{Hitchin fibrations}\label{Section0.1} 
The moduli space $\CM_{n,L}$ carries a Lagrangian fibration
\begin{equation}\label{Hitchin_map}
h: \CM_{n,L} \rightarrow \CA = \bigoplus_{i\geq 2} H^0(C, {\Omega_{C}^{\otimes i}})
\end{equation}
given by Hitchin's integrable system, which is now referred to as the \emph{Hitchin fibration}. The $\Gamma$-action on $\CM_{n,L}$ is fiberwise with respect to (\ref{Hitchin_map}). There are two types of moduli spaces closely related to the cohomological study of $\CM_{n,L}$ from the perspective of mirror symmetry \cite{HT} and representation theory \cite{Ngo0, Ngo}.

The moduli spaces of the first type are the fixed loci of an element $\gamma \in \Gamma$. For any $\gamma \in \Gamma$, we denote by $\CM_\gamma \subset \CM_{n,L}$ the $\gamma$-fixed subvariety, which maps to the Hitchin base via
\[
\CM_{\gamma}   \xrightarrow{h_\gamma} \CA_\gamma = \mathrm{Im}(h|_{\CM_\gamma}) \xhookrightarrow{i_\gamma} \CA.
\]
The $\Gamma$-action on $\CM_{n,L}$ induces a $\Gamma$-action on $\CM_{\gamma}$.

The second type of moduli spaces are associated with a cyclic Galois cover $\pi: C' \to C$ of the original curve given by $\gamma\in \Gamma$. We assume $\mathrm{deg}(\pi) = \mathrm{ord}(\gamma) = m$ and $n=mr$. Let $\CM_{r,L}(\pi)$ be the moduli space parameterizing rank $r$ stable Higgs bundles $(\CE, \theta)$ on $C'$ such that
\[
\mathrm{det}(\pi_*\CE) \simeq L, \quad \mathrm{trace}(\pi_*\theta) = 0.
\]
It admits a Hitchin fibration
\[
h_\pi: \CM_{r,L}(\pi) \rightarrow \CA(\pi)
\]
with a fiberwise $\Gamma$-action; see Section \ref{Sect1} for more details on these moduli spaces. From the viewpoint of representation theory, the moduli spaces $\CM_{r,L}(\pi)$ are related to the study of the corresponding endoscopic groups for $\mathrm{SL}_n$ over $C$ \cite{Ngo}. They are nonsingular but disconnected. The Galois group
\[
G_{\pi} = \mathrm{Aut}(\pi)\simeq \BZ/m\BZ
\]
acts on both the source $\CM_{r,L}(\pi)$ and the target $\CA(\pi)$, whose quotients recover $\CM_\gamma$ and $\CA_\gamma$ respectively. We denote by
\begin{equation*}
q_\CA : \CA(\pi) \rightarrow \CA_{\gamma} 
\end{equation*}
the quotient map of the base $\CA(\pi)$. We also consider the largest open subset $\CA(\pi)^* \subset \CA(\pi)$ upon which the $G_\pi$-action is free, and set $A_\gamma^* := \CA(\pi)^*/G_\pi \subset A_\gamma$.

\subsection{Endoscopic decompositions}
In order to understand the decomposition (\ref{maindecomp}) sheaf-theoretically, we consider the canonical decompositions of the direct image complexes 
\[
\mathrm{Rh}_* \BBC \in D_c^b(\CA), \quad \mathrm{Rh_\gamma}_* \BBC \in D_c^b(\CA_\gamma), \quad \mathrm{Rh_\pi}_* \BBC \in D_c^b(\CA(\pi))
\]
into eigen-subcomplexes with respect to the $\Gamma$-actions. We first clarify some notation before stating the main theorems. Throughout, we use $D^b_c(-)$ to denote the bounded derived category of constructible sheaves. We say that $\CK \xrightarrow{\simeq} \CK'$ is an \emph{isomorphism} for two objects in a derived category if it is a quasi-isomorphism between the complexes $\CK$ and $\CK'$. Given a complex with a $\Gamma$-action and a character $\kappa \in \hat{\Gamma}$, we denote by $(-)_\kappa$ the $\kappa$-isotypic component. We call $(-)_\mathrm{st} = (-)_{0\in \hat{\Gamma}}$ its \emph{stable part}, which is the subcomplex fixed by the $\Gamma$-action. The Weil pairing identifies canonically the group $\Gamma$ and its dual,
\begin{equation}\label{Weil} 
\hat{\Gamma} = \Gamma;
\end{equation}
see Section \ref{Weil_pair}. 

Our first result is the following theorem, which relates the stable part of the endoscopic cohomology with the pullback of the $\kappa$-isotypic contribution for $\mathrm{SL}_n$. This extends the endoscopic decomposition of \cite{Yun3} in the case of $\mathrm{SL}_n$ from the elliptic locus to a much larger open subset on the Hitchin base.

\begin{thm}[Theorem \ref{main0'}]\label{main0}
Let $\kappa \in \hat{\Gamma}$ and $\gamma \in \Gamma$ be identified by (\ref{Weil}), let $\pi:C'\to C$ be the cyclic Galois cover associated with $\gamma$, and let $d_\gamma = \mathrm{codim}_{\CA}(\CA_\gamma)$. There are isomorphisms in $D^b_c(\CA(\pi)^*)$ which are canonical up to scaling (see Definition \ref{definition}):
\begin{equation}\label{eqn_end}
 q_\CA^* \left( \mathrm{Rh_*} \BBC \right)_{\kappa}\Big{|}_{\CA(\pi)^*} \xrightarrow{\simeq} \left( \mathrm{Rh_\pi}_* \BBC \right)_{\kappa}\Big{|}_{\CA(\pi)^*}[-2d_\gamma]  \xrightarrow{\simeq} \left(\mathrm{Rh_{\pi}}_*\BBC\right)_{\mathrm{st}}\Big{|}_{\CA(\pi)^*}[-2d_\gamma]  
\end{equation}
with the first isomorphism $G_\pi$-equivariant. 
\end{thm}

In the equation (\ref{eqn_end}), the $G_\pi$-equivariant structure for the first term is given by the pullback map along the $G_\pi$-quotient $q_\CA: \CA(\pi)^* \to \CA_\gamma^*$. The $G_\pi$-equivariant structure for the second term is induced by the $G_\pi$-action on $\CM_{r,L}(\pi)$. 

The following theorem is a further extension of Theorem \ref{main0}, which provides a complete description of the $\kappa$-isotypic component of $\mathrm{Rh}_* \BBC$ in terms of the $\gamma$-fixed subvariety $\CM_\gamma \subset \CM_{n,L}$. 



\begin{thm}[Theorem \ref{thm3.2}]\label{main}
Let $\kappa \in \hat{\Gamma}$ and $\gamma \in \Gamma$ be identified by (\ref{Weil}), and let $d_\gamma = \mathrm{codim}_{\CA}(\CA_\gamma)$. We have an isomorphism
\begin{equation}\label{eqn_ck}
c_\kappa: \left(\mathrm{Rh}_* \BBC\right)_\kappa \xrightarrow{\simeq} {i_\gamma}_*\left(\mathrm{Rh_\gamma}_* \BBC\right)_\kappa [-2d_\gamma] \in D_c^b(\CA)
\end{equation}
which is canonical up to scaling.
\end{thm}

The construction of the operator
\[
c_\kappa: \left(\mathrm{Rh}_* \BBC\right)_\kappa \xrightarrow{\simeq} {i_\gamma}_*\left(\mathrm{Rh_\gamma}_* \BBC\right)_\kappa [-2d_\gamma] \in D_c^b(\CA)
\]
realizing the isomorphism of Theorem \ref{main} is the main theme of this paper. It is of a geometric nature, given by a combination of algebraic correspondences and vanishing cycle functors. Since it induces a correspondence between the $\kappa$-part of the cohomology of an $\mathrm{SL}_n$-Hitchin fiber and the $\kappa$-part of the cohomology of the corresponding endoscopic Hitchin fiber, we call Theorems \ref{main0} and \ref{main} the \emph{endoscopic decomposition} associated with $\mathrm{SL}_n$ and the character $\kappa$. A major difference between Theorem \ref{main} and the work of Ng\^o \cite{Ngo} and Yun \cite{Yun3} is that they mainly work with $D$-twisted Hitchin fibrations with $\mathrm{deg}(D)>2g-2$ or with just the elliptic locus of the $K_C$-twisted Hitchin fibration, while we are interested in entire space in the latter setting. The structure of the supports of the direct image complexes is much more complicated in the $K_C$-case over the total Hitchin base; see \cite{Supp}.

In the following, we give some applications of Theorems \ref{main}.

\subsection{Structure of the cohomology of $\CM_{n,L}$}

Let $\kappa \in \hat{\Gamma}$ and $\gamma \in \Gamma$ be identified by (\ref{Weil}). Let $\pi: C' \to C$ be the degree $m$ cyclic Galois cover associated with $\gamma$. Assume $n=mr$. We denote by $\widetilde{\CM}'_{r,d}$ the moduli space of stable ($\mathrm{GL}_r$-)Higgs bundles  
\[
(\CE, \theta): \quad \mathrm{rank}(\CE) = r, \quad  \mathrm{deg}(\CE) =d
\]
on the curve $C'$.

Recall the decomposition (\ref{maindecomp}). The following theorem is a structural result for $H^*(\CM_{n,L}, \BC)_\kappa$. 

\begin{thm}[Theorem \ref{thm5.4}] \label{thm0.4}
The operator (\ref{eqn_ck}) induces a surjective morphism
\[
\mathfrak{p}_\kappa: H^{i}(\widetilde{\CM}'_{r,d}, \BC) \twoheadrightarrow H^{i+2d_\gamma}(\CM_{n,L}, \BC)_\kappa.
\]
Moreover, let $P_kH^*(\widetilde{\CM}'_{r,d}, \BC)$ and $P_kH^{*}(\CM_{n,L}, \BC)_\kappa$ be the perverse filtrations defined via the Hitchin fibrations, then
\[
\mathfrak{p}_\kappa \left( P_k H^i(\widetilde{\CM}'_{r,d}, \BC)\right) = P_{k+d_\gamma}H^{i+2d_\gamma}( \CM_{n,L}, \BC)_{\kappa}.
\]
\end{thm}

We refer to \cite{dCM0, dCHM1} for perverse filtrations; see also Section \ref{Sec5.1} for a brief review. 

If $\kappa =0$, we have $\pi = \mathrm{id}: C \xrightarrow{\simeq} C$. The operator $\mathfrak{p}_{\kappa=0}$ in this special case recovers the restriction map
\[
j^*: H^{i}(\widetilde{\CM}_{n,d}, \BC) \twoheadrightarrow H^{i}(\CM_{n,L}, \BC)^\Gamma
\]
associated with the embedding $j: \CM_{n,L} \hookrightarrow \widetilde{\CM}_{n,d}=\widetilde{\CM}'_{r,d}$. 

By Markman's theorem \cite{Markman}, the cohomology $H^*(\widetilde{\CM}'_{r,d}, \BC)$ is generated by the tautological classes associated with a universal family on $\widetilde{\CM}'_{r,d}$. Hence Theorem \ref{thm0.4} shows that each isotypic component $H^{*}( \CM_{n,L}, \BC)_{\kappa}$ for $\kappa \neq 0$ is governed by the tautological classes of a \emph{different} moduli space of Higgs bundles through the operator $\mathfrak{p}_\kappa$. More discussions concerning Theorem \ref{thm0.4} and the P=W conjecture \cite{dCHM1} are given in Section \ref{Section5}.

\subsection{The Hausel--Thaddeus conjecture}\label{Sec0.5} 

In \cite{HT}, Hausel and Thaddeus showed that the moduli spaces of stable $\mathrm{SL}_n$- and $\mathrm{PGL}_n$-Higgs bundles are mirror partners in the sense of the Strominger--Yau--Zaslow mirror symmetry. As a consequence, these two moduli spaces should have identical Hodge numbers. 

As explained in \cite{HT}, the moduli space of degree $d$ stable $\mathrm{PGL}_n$-Higgs bundles can be realized as the quotient $\CM_{n,L}/\Gamma$ which is naturally a Deligne--Mumford stack. Therefore, Hausel--Thaddeus conjectured that, for any two line bundles $L,L'$ with
\[
\mathrm{deg}(L)=d, ~~~\mathrm{deg}(L') =d', \quad \mathrm{gcd}(d,n) = \mathrm{gcd}(d',n) =1,
\]
the Hodge numbers of $\CM_{n,L}$ are the same as the stringy Hodge numbers of the stack $[\CM_{n,L'}/\Gamma]$ (twisted by a certain gerbe $\alpha$):
\begin{equation}\label{HTC}
h^{i,j}(\CM_{n,L}) = h_\mathrm{st}^{i,j}([\CM_{n,L'}/\Gamma], \alpha);
\end{equation}
see \cite[Section 4]{HT} for precise definitions of the gerbe $\alpha$ and the stringy Hodge numbers. Later, Hausel further conjectured a refinement of (\ref{HTC}): the Hodge numbers of $H^*(\CM_{n,L}, \BC)_\kappa$ coincide with the  Hodge numbers of the gerby sector $[(\CM_{n,L'})_\gamma/\Gamma]$. Here $(\CM_{n,L'})_\gamma \subset \CM_{n,L'}$ is the $\gamma$-fixed subvariety, and $\kappa$ and $\gamma$ are matched via (\ref{Weil}). We refer to \cite[Conjecture 4.5]{Survey} concerning Hodge numbers and \cite[Conjecture 5.9]{Survey} for a further refinement involving perverse filtrations.

The following theorem is a direct consequence of Theorem \ref{main} which proves the Hausel--Thaddeus conjecture and a refinement of it; see the formulation (1.2.2) and the appendix in \cite{LW}  for an explanation of how the right-hand side is equivalent to the gerby description above.

\begin{thm}\label{thm0.5}
Assume that $e$ is a multiplicative inverse of $d'$ modulo $n$. Let $\gamma$ and $\kappa$ be matched via (\ref{Weil}). The following identity holds in the Grothendieck group of complex Hodge structures $K_0(\mathrm{HS})$:
\begin{equation}\label{eqnthm0.5}
[P_kH^i(\CM_{n,L}, \BC)_\kappa] = [P_{k-d_\gamma}H^{i-2d_\gamma}((\CM_{n,L'})_\gamma, \BC)_{de\kappa}(-d_\gamma)].
\end{equation}
Here $(\bullet)$ stands for the Tate twist\footnote{Recall that for $V\in K_0(\mathrm{HS})$, we have $h^{i,j}(V(\bullet))= h^{i+\bullet,j+\bullet}(V)$.}, and each piece of the perverse filtrations admits a natural Hodge structure by the theory of mixed Hodge modules. In particular, for any $q\in \BZ$ coprime to $n$, we have
\begin{equation}\label{eqnthm_new}
[P_kH^i(\CM_{n,L}, \BC)] = \sum_{\gamma\in \Gamma} [P_{k-d_\gamma}H^{i-2d_\gamma}((\CM_{n,L'})_\gamma, \BC)_{q\kappa}(-d_\gamma)] \in K_0(\mathrm{HS}).
\end{equation}
\end{thm}

\begin{rmk}
By the last paragraph of \cite[Proof of Proposition 8.2]{HT}, the shift 
\[
d_\gamma =  \mathrm{codim}_{\CA}(\CA_\gamma) = \frac{1}{2} \mathrm{codim}_{\CM_{n,L}}(\CM_{\gamma})
\]
in (\ref{eqnthm0.5}) coincides with the ``fermionic shift" $F(\gamma)$ in the formulation of Hausel and Thaddeus.
\end{rmk}

A refined version of the Hausel--Thaddeus conjecture was previously proven by Gr\"ochenig--Wyss--Ziegler \cite[Theorem 7.24]{GWZ} by p-adic integration, and was generalized by Loeser--Wyss \cite[Remark 5.3.4]{LW} by motivic integration. Note that our refined version (\ref{eqnthm0.5}) is slightly different from the versions of \cite{GWZ, LW}, since the right-hand side of (\ref{eqnthm0.5}) depends on the degree of $L$ while the corresponding term in \cite{GWZ, LW} is independent of this degree.  Instead, our refined version is closer to the conjecture formulated by Hausel in \cite[Conjectures 4.5 and 5.9]{Survey}. Motivated by the Hausel--Thaddeus conjecture, connections between the moduli of Higgs bundles and the $\gamma$-fixed locus with $\gamma \in \Gamma$ were discussed in \cite{Unramified} via Fourier--Mukai transform.


\subsection{Idea of the proof}
Our approach proceeds in two steps.
We first show analogs of Theorems \ref{main0} and \ref{main} for the moduli space $\CM^D_{n,L}$ of $D$-twisted $\mathrm{SL}_n$-Higgs bundles with $\mathrm{deg}(D)$ even and greater than $2g-2$ (see Section \ref{Sect1} for precise definitions).  As mentioned earlier, one expects this case to be simpler than the original setting, due to work of Chaudouard--Laumon \cite{CL} and de Cataldo \cite{dC_SL}, which determines the supports appearing in the decomposition theorem for the twisted Hitchin map.
After proving the corresponding support theorem for endoscopic moduli, we study the endoscopic decomposition of Ng\^{o} \cite{Ngo} and Yun \cite{Yun3} over the elliptic locus, and extend it over the full twisted Hitchin base.

Unfortunately, this approach is not sufficient when $D = K_C$, since the supports of the Hitchin map remain mysterious \cite{Supp}.  Moreover, although $\CM_{n,L}$ embeds inside $\CM^D_{n,L}$ for certain effective divisor $D$ with $\mathrm{deg}(D) >2g-2$, we cannot simply pull back equation \eqref{eqn_ck}.  

Instead, we realize $\CM_{n,L}$ as the critical locus of a regular function 
\begin{equation}\label{eqn_mu}
\mu_{\pi,\CM}: \CM^D_{n,L} \to \BA^1;
\end{equation}
see Theorem \ref{thm4.4}. This allows us to express the cohomology of $\CM_{n,L}$ as the vanishing cohomology of (\ref{eqn_mu}). In addition, since the function $\mu_{\pi,\CM}$ factors through the Hitchin base, we can use the vanishing cycles functor to relate the decomposition theorem for $\CM_{n,L}$ in terms of that for $\CM^D_{n,L}$.  By applying this technique to the twisted version of equation \eqref{eqn_ck}, we obtain the full result.

\subsection{Relation to other work}
As discussed in Section \ref{Sec0.5}, the Hausel--Thaddeus conjecture as well as its refinements has been proven by Gr\"ochenig--Wyss--Ziegler \cite{GWZ} via p-adic integration. Using similar approach, they also gave a new proof of Ng\^o's geometric stabilization theorem \cite{GWZ2} which plays a crucial role in Ng\^o's proof of the fundamental lemma of the Langlands program \cite{Ngo}. Our approach goes along the inverse direction --- we prove the Hausel--Thaddeus conjecture by extending Ng\^o's method \cite{Ngo,Yun3} in the proof the geometric stabilization theorem via perverse sheaves and the support theorems. This carries out the proposal of Hausel in \cite[Section 5.4]{Survey}. A benefit of the sheaf-theoretic approach is that it allows us to construct concrete geometric operators which realize the Hausel--Thaddeus conjecture (Theorems \ref{main0} and \ref{main}) and which provide better understanding of $H^*(\CM_{n,L}, \BC)_\kappa$ for each non-trivial $\kappa$. 


\subsection{Acknowledgement}
We are grateful to Mark Andrea de Cataldo, Michael Gr\"ochenig, and Zhiwei Yun for helpful discussions; we are grateful to Elsa Maneval and Dimitri Wyss for helpful correspondences that led us to discover a mistake in the previous version. We also thank the anonymous referee for careful reading and numerous useful suggestions. D.M. would like to thank Chris and Martin Kratt for providing logistical support.  J.S. was supported by NSF DMS-2000726 and DMS-2134315.

\section{Hitchin-type moduli spaces}\label{Sect1}
Throughout, we work over the complex numbers $\BC$. In this section, we fix the curve $C$ of genus $g \geq 2$, the rank $n$, and the line bundle $L \in \mathrm{Pic}^d(C)$ which serves as the determinant of the Higgs bundles as in Section \ref{Sec0.1}. We study several Hitchin-type moduli spaces relevant to Theorems \ref{main0} and \ref{main}.

\subsection{$D$-Higgs bundles}
For our purpose, it is important to consider generalized Higgs bundles $(\CE, \theta)$ with the Higgs field $\theta$ twisted by a divisor $D$ not necessarily the canonical divisor $K_C$. Such flexibility also plays a crucial role in the proof of the fundamental lemma \cite{Ngo0, Ngo}.

Let $D$ either be an effective divisor of degree $\mathrm{deg}(D) > 2g-2$ or $D= K_C$. A $D$-Higgs bundle is a pair $(\CE, \theta)$ where $\CE$ is a vector bundle and $\theta$ is a $D$-twisted Higgs field 
\[
\theta: \CE \rightarrow \CE \otimes \CO_C(D).
\]
We denote by $\mathrm{char}(\theta)$ the tuple of the coefficients for the  characteristic polynomial associated with $(\CE, \theta)$,
\[
\mathrm{char}(\theta) = (a_1, \dots, a_n), \quad a_i =\mathrm{trace}(\wedge^i \theta) \in H^0(C, \CO_C(iD)).
\]

Parallel to the case of $K_C$-Higgs bundles, the stability condition for $D$-Higgs bundles is with respect to the slope $\mu(\CE) = \mathrm{deg}(\CE)/\mathrm{rank}(\CE)$. By \cite{Nit}, there is a nonsingular quasi-projective moduli space $\widetilde{\CM}_{n,d}^D$ parameterizing stable $D$-Higgs bundles of rank $n$, degree $d$, with the Hitchin map
\begin{equation}\label{GL_Hit}
\widetilde{h}^D : \widetilde{\CM}_{n,d}^D \rightarrow \widetilde{\CA}^D = \oplus_{i=1}^n H^0(C, \CO_C(iD)), \quad (\CE, \theta) \mapsto \mathrm{char}(\theta),
\end{equation}
which is proper and surjective. 

The moduli space of stable $\mathrm{SL}_n$ $D$-Higgs bundles is defined to be the subvariety
\begin{equation}\label{SLn_space}
\CM^D_{n,L} = \{(\CE, \theta) \in \widetilde{\CM}_{n,d}^D; \quad \mathrm{det}(\CE) \simeq L, \quad \mathrm{trace}(\theta) = 0  \} \subset \widetilde{\CM}_{n,d}^D.
\end{equation}
It is nonsingular and irreducible by \cite[Section 2.1]{dC_SL}, which has a Hitchin map induced from the Hitchin map of the ambient space (\ref{GL_Hit}),
\begin{equation}\label{SLn_Hit}
h^D: \CM^D_{n,L} \rightarrow \CA^D = \oplus_{i=2}^n H^0(C, \CO_C(iD)).
\end{equation}
It is clear that the variety $\CM^D_{n,L}$ is the fiber over the closed point $(L,0)$ of the smooth map
\begin{equation}\label{eqn9}
q: \widetilde{\CM}_{r,d}^{D} \rightarrow  \widetilde{\CM}_{1,d}^{D} =\mathrm{Pic}^d(C) \times H^0(C, \CO_C(D)), \quad (\CE, \theta) \mapsto (\mathrm{det}(\CE), \mathrm{trace}(\theta)).
\end{equation}

A major difference between the $D = K_C$ case and the case of $\mathrm{deg}(D)>2g-2$ is that, the Hitchin fibration for either $\mathrm{GL}_n$ or $\mathrm{SL}_n$ is \emph{Lagrangian} with respect to the canonical hyper-K\"ahler structure for $D = K_C$, while the dimension of the Hitchin base is always larger than the dimension of a fiber in the case of $\mathrm{deg(D)}>2g-2$. In view of the support theorems (Section \ref{Section2.2}), such a difference will influence substantially the study of the topology of Hitchin fibrations.

From now on, all $D$-Higgs bundles will be uniformly called \emph{Higgs bundles} for convenience.

\subsection{Relative Hitchin moduli spaces}\label{Sec1.2}
In this section, we study the relative Hitchin moduli space associated with a Galois cover $\pi: C' \to C$. This parameterizes \emph{stable} Higgs bundles with respect to the endoscopic group of $\mathrm{SL}_n$ over $C$ attached to a character $\kappa \in \hat{\Gamma}$; see \cite{Ngo0, Ngo}. 

Let $C'$ be a nonsingular curve with a cyclic Galois cover 
\[
\pi: C' \to C
\]
whose Galois group is $G_\pi \cong \mathbb{Z}/m\BZ$. We denote the divisor $\pi^*D$ by $D'$. For an element $\sigma \in H^0(C', \CO_{C'}(D'))$, the pushforward along $\pi$ gives an element
\[
\pi_* \sigma \in H^0(C, \pi_* \CO_{C'}(D')).
\]
The trace of $\pi_*\sigma$ recovers its projection to the direct summand component 
\[
\mathrm{trace}(\pi_*\sigma) \in H^0(C, \CO_C(D)) \subset H^0(C, \pi_* \CO_{C'}(D')).
\]
The moduli space $\widetilde{\CM}_{r,d}^{D'}(C')$ of rank $r$ degree $d$ stable Higgs bundles on $C'$ admits a map
\begin{equation} \label{eqn8}
q_\pi: \widetilde{\CM}_{r,d}^{D'}(C') \rightarrow  \widetilde{\CM}_{1,d}^{D}(C)  
\end{equation}
which is the composition of (\ref{eqn9}) for the curve $C'$ and the pushforward
\begin{equation}\label{eqn10}
\pi_*: \widetilde{\CM}_{1,d}^{D'}(C') \rightarrow  \widetilde{\CM}_{1,d}^D(C), \quad (\CL, \sigma) \mapsto (\mathrm{det}(\pi_*\CL), \mathrm{trace}(\pi_\ast \sigma)).\footnote{When we wish to be specific regarding the dependence of the moduli spaces on the underlying curve $C$ or $C'$, we add $(C)$ or $(C')$ after the corresponding moduli spaces.}
\end{equation}
Since both (\ref{eqn9}) and (\ref{eqn10}) are smooth, the composition $q_\pi$ is also smooth.

We define the relative Hitchin moduli space of rank $r$ and degree $d$ associated with $\pi$ as the subvariety of $\widetilde{\CM}_{r,d}^{D'}(C')$ given as a fiber of (\ref{eqn8}):
\[
\CM_{r,L}^D(\pi) = q_\pi ^{-1}(L,0) \subset \widetilde{\CM}_{r,d}^{D'}(C').
\]
The variety $\CM_{r,L}^D(\pi)$, which recovers the $\mathrm{SL}_n$-Hitchin moduli space (\ref{SLn_space}) when $\pi = \mathrm{id}$, is nonsingular due to the smoothness of $q_\pi$.

Next we describe the Hitchin fibration associated with $\CM_{r,L}^D(\pi)$ which generalizes (\ref{SLn_Hit}). Recall from (\ref{GL_Hit}) the $\mathrm{GL}_r$-Hitchin fibration $h^{D'}: \widetilde{\CM}_{r,d}^{D'}(C') \to \widetilde{\CA}^{D'}(C')$ associated with the curve $C'$. The restriction of $h^{D'}$ to $\CM_{r,L}^D(\pi)$ induces the Hitchin map 
\begin{equation}\label{eqn_h_pi}
h_\pi^D: \CM_{r,L}^D(\pi) \to \CA^D(\pi)
\end{equation}
which fits into the commutative diagram
\begin{equation}\label{diag12}
\begin{tikzcd}
\CM_{r,L}^D(\pi) \arrow[r, hook] \arrow[d, "h_\pi^D"]
& \widetilde{\CM}_{r,d}^{D'}(C') \arrow[d, "h^{D'}"] \\
\CA^D(\pi) \arrow[r, hook]
& \widetilde{\CA}^{D'}(C').
\end{tikzcd}
\end{equation}
The Hitchin base $\CA^D(\pi)$ can be concretely described as 
\begin{equation}\label{base1}
\CA^D(\pi) = H^0(C', \CO_{C'}(D'))_{\mathrm{var}} \oplus \left(\oplus_{i=2}^r H^0(C', \CO_{C'}(iD'))\right)
\end{equation}
where $H^0(C', \CO_{C'}(D'))_{\mathrm{var}}$ is the variant part with respect to the natural Galois group $G_{\pi}$-action induced by the $G_\pi$-action on $C'$; see \cite[Section 5]{HP}. Since the line bundles $\CO_{C'}(iD')$ have canonical $G_{\pi}$-linearizations, there is a natural $G_\pi$-action on the Hitchin base (\ref{base1}).

\begin{prop}\label{prop1.1} We have the following properties.
\begin{enumerate}
    \item[(a)] The moduli space $\CM_{r,L}^D(\pi)$ is a disjoint union of $m$ nonsingular isomorphic components
    \begin{equation}\label{M_i_decomp}
    \CM_{r,L}^D(\pi) = \bigsqcup_{i=1}^m M_i.
    \end{equation}
    \item[(b)] The restrictions of $h_\pi^D$ to all components $h_i: M_i \to \CA^D(\pi)$ are $\CA^D(\pi)$-isomorphic. More precisely, for each pair $1 \leq i,j \leq m$, there exists an isomorphism $\phi_{ij}: M_i \xrightarrow{\simeq} M_j$ induced by tensoring with a line bundle $\CL_{ij} \in \Gamma$ satisfying the commutative diagram 
    \begin{equation}\label{eq13}
    \begin{tikzcd}[column sep=small]
    M_i \arrow[dr, "h_i"] \arrow[rr, "\phi_{ij}"] & & M_j \arrow[dl, "h_j"] \\
       & \CA^D(\pi).  & 
\end{tikzcd}
\end{equation}
\end{enumerate}
\end{prop}

\begin{proof}

Recall that $\CM_{r,L}^D(\pi)$ is the fiber of 
\[
q_\pi = (\pi_*) \circ q
\]
(see (\ref{eqn8})) over the point $(L,0) \in \widetilde{\CM}_{1,d}^{D}(C)$. The map $q$ is surjective and smooth, whose fibers are isomorphic to the moduli of stable $\mathrm{SL}_r$-Higgs bundles of degree $d$ on the curve $C'$. In particular, each fiber of $q$ is nonsingular and irreducible. 


The morphism $\pi_*$ given in (\ref{eqn10}) respects the product structures
\begin{equation}\label{prod_str}
\pi_*: \mathrm{Pic}^d(C') \times H^0(C', \CO_{C'}(D'))\rightarrow  \mathrm{Pic}^d(C) \times H^0(C, \CO_C(D)),
\end{equation}
where the morphism between the second factors form a trivial affine bundle. For the first factors of (\ref{prod_str}), a fiber of $\pi_*: \mathrm{Pic}^d(C') \rightarrow  \mathrm{Pic}^d(C)$ is isomorphic to the degree $d$ Prym variety associated with the Galois cover $\pi: C' \to C$, which is the disjoint union of $m$ isomorphic abelian varieties; see \cite[Section 7]{HT}. Hence the moduli space $\CM_{r,L}^D(\pi)$ has $m$ nonsingular connected components. 

Assume $n = mr$. Tensoring with a line bundle $\CL \in \Gamma = \mathrm{Pic}^0(C)[n]$ induces an $\CA^D(\pi)$-automorphism 
\begin{equation}\label{phi_L}
\phi_\CL: \CM_{r,L}^D(\pi) \xrightarrow{\simeq} \CM_{r,L}^D(\pi),  \quad \phi_\CL (\CE, \theta) = (\CL\otimes \CE , \theta).
\end{equation}
Moreover, for a general point $a \in \CA^D(\pi)$ corresponding to a degree $n=rm$ spectral cover
\[
g_a: C'_a \to C' \xrightarrow{\pi} C,
\]
the fiber ${(h_\pi^D)}^{-1}(a)$ is identical to a fiber of the morphism
\[
{g_a}_*: \mathrm{Pic}^d(C'_a) \rightarrow \mathrm{Pic}^d(C), \quad \CL \mapsto \mathrm{det}({g_a}_* \CL),
\]
where $\Gamma$ acts transitively on the set of its connected components (\emph{c.f.} \cite[Lemmas 2.1 and 2,2]{HP}). This ensures that $\Gamma$ acts transitively on $\{M_i\}_{1\leq i\leq m}$. In particular, for any pair $1\leq i,j \leq m$, there exits a line bundle $\CL_{ij}\in \Gamma$ such that the isomorphism $\phi_{ij} = \phi_{\CL_{ij}}$ given in (\ref{phi_L}) satisfies the commutative diagram (\ref{eq13}). (a) and (b) are proved.
\end{proof}

\subsection{Weil pairing and cyclic covers}\label{Weil_pair}
Recall from Section \ref{Sec0.1} that the group $\Gamma = \mathrm{Pic}^0(C)[n]$ acts on the $\mathrm{SL}_n$-moduli space $\CM^D_{n,L}$ via tensor product. For $\gamma \in \Gamma$, the $\gamma$-fixed subvariety $\CM^D_\gamma \subset \CM^D_{n,L}$ carries an induced Hitchin map
\[
\CM^D_{\gamma}   \xrightarrow{h^D_\gamma} \CA^D_\gamma = \mathrm{Im}(h^D|_{\CM_\gamma}) \xhookrightarrow{i^D_\gamma} \CA^D.
\]
As indicated by Theorem \ref{main}, the cohomology of $\CM^D_{\gamma}$ is related to a $\kappa$-isotypic component of the cohomology of $\CM^D_{n,L}$ with respect to the $\Gamma$-action. 

In order to describe this connection, we need a  correspondence (\ref{Weil}) between an element $\gamma \in \Gamma$ and a character $\kappa \in \hat{\Gamma}$, which we review in the following. 

Let $\mu_n \subset \BC^*$ denote the finite group of the $n$-th roots of unity. We have the \emph{Weil pairing} on the group of $n$-torsion points of $\mathrm{Pic}^0(C)$, 
\[
\langle~~,~~\rangle_\Gamma: \Gamma \times \Gamma \rightarrow \mu_n.
\]
Under the identification 
\[
\mathrm{Pic}^0(C)[n] = H_1(C, \BZ/n\BZ),
\]
the Weil pairing coincide with the intersection pairing on $H_1(C, \BZ/n\BZ)$. In particular, it is non-degenerate, which induces a character
\[
\kappa = \langle \gamma, ~~ \rangle_\Gamma: \Gamma \rightarrow \mu_n \subset \BC^*
\]
for each $\gamma \in \Gamma$. This gives the identification (\ref{Weil}). 

We also note that an element $\gamma \in \Gamma$ naturally corresponds to a cyclic Galois cover of $C$ whose degree is the order of $\gamma \in \Gamma$. In fact, for fixed $\gamma \in \Gamma$, we denote by $L_\gamma$ the $n$-torsion line bundle associated with $\gamma$, and let $m$ be its order which divides $n$. Taking the $m$-th roots of unity fiberwise inside the total space of $L_{\gamma}$ with respect to the zero section $C$, we obtain a cyclic Galois cover 
\[
\pi: C' \rightarrow C
\]
with the Galois group $G_{\pi} \simeq \mathbb{Z}/m\BZ$. Conversely, every degree $m$ \'etale cyclic Galois cover arises this way.

\subsection{Characters}
In this section we give a concrete description for the character $\kappa$ in terms of the Prym variety associated with the corresponding Galois cover $\pi: C' \to C$.

As before, we assume that $\pi:C' \to C$ is a degree $m$ cyclic Galois cover associated with $\kappa \in \hat{\Gamma}$ as in Section \ref{Weil_pair}, and we assume that $n = mr$. The character $\kappa \in \hat{\Gamma}$ matches with $\gamma \in \Gamma$ via (\ref{Weil}). We consider the Prym variety
\begin{equation}\label{Prym1}
\mathrm{Prym}(C'/C) = \mathrm{Ker}\left( \mathrm{Nm}: \mathrm{Pic}^0(C') \to \mathrm{Pic}^0(C) \right)
\end{equation}
with $\mathrm{Nm}$ the norm, and the component group 
\[
\pi_0(\mathrm{Prym}(C'/C))\simeq \BZ/m\BZ.
\]
For an $n$-torsion line bundle $\CL \in \Gamma$, we have
\[
\mathrm{Nm}(\pi^*\CL) = \CL^{\otimes m} \in \mathrm{Pic}^0(C)[r].
\]
In particular, the line bundle $\pi^*\CL^{\otimes r}\in \mathrm{Pic}^0(C')$ represents a point in $\mathrm{Prym}(C'/C))$, which yields a natural group homomorphism
\begin{equation}\label{eqnn16}
  \Gamma  \rightarrow \pi_0(\mathrm{Prym}(C'/C)), \quad \CL  \mapsto [\pi^*\CL^{\otimes r}] \in \pi_0(\mathrm{Prym}(C'/C)).
\end{equation}

The morphism (\ref{eqnn16}) admits a factorization
\[
\Gamma =\mathrm{Pic}^0(C)[n] \xrightarrow{[r]} \mathrm{Pic}^0(C)[m] \rightarrow \pi_0(\mathrm{Prym}(C'/C)).
\]
The first map above is the multiplication by $r$ which is clearly surjective. The second map sends $\CL \in \mathrm{Pic}^0(C)[m]$ to the equivalent class of the line bundle $\pi^*\CL \in \mathrm{Prym}(C'/C)$, and its surjectivity is given by the proof of  \cite[Theorem 1.1 (2)]{HP}. Hence (\ref{eqnn16}) is surjective.

Recall that $\gamma \in \Gamma$ is of order $m$, so
\[
\gamma \in \mathrm{Pic}^0(C)[m] \subset \Gamma.
\]
The following lemma is obtained by viewing the Weil pairing on $\mathrm{Pic}^0(C)[i]$ for any $i \in \BN$, via Poincar\'{e} duality, as the intersection pairing 
\[
H_1(C, \BZ/i\BZ) \times H_1(C, \BZ/i\BZ) \rightarrow \BZ/i\BZ
 \simeq \mu_i.\]

\begin{lem}\label{lem1.2}
Assume $\gamma' \in \Gamma$. Let $\langle ~~, ~~\rangle_{\mathrm{Pic}^0(C)[m]}$ denotes the Weil pairing on $\mathrm{Pic}^0(C)[m]$, and we view $r\gamma'$ naturally as an element in $\mathrm{Pic}^0(C)[m]$. We have 
\[
\langle \gamma, \gamma'\rangle_\Gamma = \langle \gamma, r\gamma'\rangle_{\mathrm{Pic}^0(C)[m]}.
\]
\end{lem}

\begin{prop}\label{prop1.2}
The character $\kappa \in \hat{\Gamma}$ (corresponding to $\pi:C' \to C$) factors through a character of $\pi_0(\mathrm{Prym}(C'/C))$ of order $m$ via the morphism (\ref{eqnn16}), 
\[
\kappa: \Gamma \rightarrow \pi_0(\mathrm{Prym}(C'/C))\left( \simeq \BZ/m\BZ  \right) \to \BC^*.
\]
\end{prop}

\begin{proof}
Recall from Section \ref{Weil_pair} that the character $\kappa$ is given by the Weil pairing $\langle \gamma, ~~~\rangle_\Gamma$ where $\gamma \in \Gamma$ corresponds to $\pi$. Since $\gamma$ is of order $m$, \emph{i.e.,}
\[
\gamma \in \mathrm{Pic}^0(C)[m] \subset  \Gamma,
\]
Lemma \ref{lem1.2} implies for any element $\gamma' \in \Gamma$ that
\begin{equation}\label{eeqn17}
\kappa(\gamma') = \langle \gamma, \gamma'\rangle_\Gamma = \langle \gamma, r\gamma'\rangle_{\mathrm{Pic}^0(C)[m]}.
\end{equation}

We consider the subgroup
\[
K = \mathrm{Ker}\left\{ \pi^*: \mathrm{Pic}^0(C) \to \mathrm{Pic}^0(C') \right\} \subset \mathrm{Pic}^0(C).
\]
It is the cyclic finite subgroup $\langle \gamma \rangle$ of $\mathrm{Pic}^0(C)$ generated by $\gamma$. Since $\gamma$ is of order $m$, we have $K \subset \mathrm{Pic}^0(C)[m]$. By the proof of \cite[Theorem 1.1 (1)]{HP}, there is a canonical isomorphism
\begin{equation}\label{KKK}
\pi_0(\mathrm{Prym}(C'/C)) = \hat{K},
\end{equation}
which, for any $\CL\in \mathrm{Pic}^0(C)[m]$, identifies $[\pi^*\CL]\in\pi_0(\mathrm{Prym}(C'/C))$ with the character of $K$ sending the generator $\gamma \in K$ to 
\[
\langle \gamma, \CL\rangle_{\mathrm{Pic}^0(C)[m]} \in \mu_m \subset \BC^*.
\]
Comparing with (\ref{eeqn17}), this implies that $\kappa: \Gamma \to \BC^*$ is the composition of (\ref{eqnn16}) and the order $m$ character of $\pi_0(\mathrm{Prym}(C'/C))$ given by
\[
\gamma \in K = \mathrm{Hom}({\pi_0(\mathrm{Prym}(C'/C))}, \BC^*).
\]
Here we dualize (\ref{KKK}) in the last identity. This completes the proof.
\end{proof}

Since ${\pi_0(\mathrm{Prym}(C'/C))}\cong \BZ/m\BZ$, its character group is also cyclic. We obtain the following corollary of Proposition \ref{prop1.2}. 

\begin{cor}\label{cor1.3}
A character $\rho: \Gamma \to \BC^*$ lies in the cyclic subgroup $\langle \kappa \rangle \subset \hat{\Gamma}$ if and only if $\rho$ is induced by a character of $\pi_0(\mathrm{Prym}(C'/C))$ factoring through (\ref{eqnn16}).
\end{cor}


Now we consider the kernel of the morphism (\ref{eqnn16}), which we denote by
\begin{equation}\label{Omega}
\Omega \subset \Gamma.
\end{equation}
The subgroup $\Omega$ acts on the moduli space $\CM^D_{r,L}(\pi)$ via tensor product.

\begin{lem}\label{lem1.5}
The $\Omega$-action on $\CM^D_{r,L}(\pi)$ preserves each connected component $M_i$ of (\ref{M_i_decomp}).
\end{lem}

\begin{proof} 
Recall that $\CM_{r,L}^D(\pi)$ is a fiber of 
\[
q_\pi = (\pi_*) \circ q: \widetilde{\CM}_{r,d}^{D'}(C') \xrightarrow{q} \widetilde{\CM}_{1,d}^{D'}(C') \xrightarrow{\pi_*} \widetilde{\CM}_{1,d}^{D}(C).
\]
The fibers of the first map are connected. The second map respects the product structure (\ref{prod_str}). Both the maps $q$ and $\pi_*$ are $\Gamma$-equivariant. Hence the $m$ connected components of (\ref{M_i_decomp}) matches the connected components of the degree $d$ Prym variety
\[
\mathrm{Prym}^d(C'/C)= (\pi_*)^{-1}(L),\quad  \pi_*: \mathrm{Pic}^d(C') \to \mathrm{Pic}^d(C).
\]
Since $\mathrm{Prym}^d(C'/C)$ is a torsor of $\mathrm{Prym}(C'/C)$ (defined in (\ref{Prym1})), and by definition the group $\Omega$ acts trivially on the component group $\pi_0(\mathrm{Prym}(C'/C))$, we obtain that the action of $\Omega$ preserves each connected component of $\mathrm{Prym}^d(C'/C)$. The proposition then follows from the fact that the restriction of $q$ to $\CM^D_{r,d}(\pi)$
\[q|_{(\pi_*)^{-1}(L,0)}: \CM^D_{r,L}(\pi) \to (\pi_*)^{-1}(L,0) = \mathrm{Prym}^d(C'/C)\times H^0(C', \CO_{C'}(D'))_{\mathrm{var}}
\]
is $\Omega$-equivariant.
\end{proof}

\subsection{Endoscopic moduli spaces and $\gamma$-fixed loci}\label{Section1.3}
In this section, we connect the $\gamma$-fixed subvariety $\CM^D_\gamma$ to the relative Hitchin moduli spaces introduced in Section \ref{Sec1.2}. 


We fix $\gamma \in \Gamma$ of order $m$. Let $\pi: C' \rightarrow C$ be the cyclic Galois cover with the Galois group $G_{\pi} \simeq \mathbb{Z}/m\BZ$ corresponding to $\gamma$ as in Section \ref{Weil_pair}. Assume $n=mr$. We consider the relative Hitchin moduli space $\CM^D_{r,L}(\pi)$ with the Hitchin fibration
\begin{equation}\label{h^D_H}
h_{\pi}^D:\CM^D_{r,L}(\pi) \rightarrow  \CA^D(\pi).
\end{equation}


By \cite[Proposition 7.1]{HT}, the Galois group $G_{\pi}$ acts freely on the relative moduli space $\CM^D_{r,L}(\pi)$ whose quotient recovers the $\gamma$-fixed subvariety $\CM^D_\gamma$. The group $G_{\pi}$ also acts on the base $\CA^D(\pi)$ with the Hitchin map (\ref{h^D_H}) $G_{\pi}$-equivariant. In summary, we have the following commutative diagrams
\begin{equation}\label{diagram111}
\begin{tikzcd}
\CM^D_{r,L}(\pi) \arrow[r, "q_\CM"] \arrow[d, "h^D_{\pi}"]
& \CM_\gamma^D \arrow[r, hook] \arrow[d, "h^D_\gamma"] & \CM^D_{n,L} \arrow[d, "h^D"] \\
\CA^D{(\pi)} \arrow[r, "q_\CA"]
& \CA^D_\gamma \arrow[r, hook, "i_\gamma"] & \CA^D
\end{tikzcd}
\end{equation}
where $q_\CM$ and $q_\CA$ are the quotient maps with respect to the natural actions of the Galois group $G_{\pi}$. Let $\CA^D(\pi)^*$ be the largest open subset of $\CA^D(\pi)$ upon which the $G_\pi$-action is free. Then the left diagram is Cartesian after restricting to 
\[
 \CA^D(\pi)^* \xrightarrow{~~~q_\CA~~~}  {\CA_\gamma^{D}}^*:= \CA^D(\pi)^*/G_\pi.
\]

\begin{rmk}
We see from the description (\ref{base1}) that $\CA^D(\pi)^*$ is non-empty. In fact, it suffices to find vectors in the affine space $\CA^D(\pi)$ that are not fixed by any nontrivial element of the cyclic group $G_\pi$. We consider the pushforward of the structure sheaf $\CO_{C'}$ along the Galois cover $\pi: C' \to C$, which admits a splitting $\pi_* \CO_{C'} = \oplus_\chi L_\chi$ where $\chi$ runs through all characters of $G_\pi$ and $L_\chi$ is a degree 0 line bundle corresponding to $\chi$. The projection formula yields 
\[
\pi_* \CO_{C'}(iD') = \pi_* \pi^* \CO_C(iD) = \bigoplus_{\chi} \CO_C(iD)\otimes L_\chi,\quad \forall i \geq 1.
\]
In particular, for \emph{any} character $\chi$ of $G_\pi$, we have
\begin{equation}\label{claim}
H^0(C',\CO_{C'}(iD'))_\chi=H^0(C, \CO_{C}(iD)\otimes L_\chi) \neq 0
\end{equation}
where we used the Riemann--Roch formula. The non-emptiness of $\CA^D(\pi)^*$ follows from (\ref{claim}) and (\ref{base1}).
\end{rmk}

\begin{lem}\label{lemma1.6}
The direct image complex 
\[
\mathrm{Rh^D_{\pi}}_* \BBC \in D^b_c(\CA^D(\pi))
\]
is $G_{\pi}$-equivariant, and we have a canonical isomorphism 
\begin{equation*}\label{eqn12}
\left( {q_\CA}_* {\mathrm{Rh}^D_{\pi}}_* \BBC \right)^{G_\pi} \xrightarrow{\simeq} {\mathrm{Rh}^D_\gamma}_* \BBC \in D^b_c(\CA^D_\gamma).
\end{equation*}
\end{lem}

\begin{proof}
Since the trivial local system on $\CM^D_{r,L}(\pi)$ is $G_{\pi}$-equivariant, the pushforward ${q_\CM}_* \BBC$ along the quotient map $q_\CM$ admits a natural $G_\pi$-action with a canonical isomorphism
\begin{equation}\label{eqn13}
\left( {q_\CM}_* \BBC \right)^{G_{\pi}} \xrightarrow{\simeq} \BBC \in D^b_c(\CM^D_\gamma).
\end{equation}
The map $h^D_\pi$ is $G_{\pi}$-equivalent, therefore we obtain in $D^b_c(\CA^D_\gamma)$ that
\[
\left({q_\CA}_* \mathrm{Rh^D_\pi}_* \BBC \right)^{G_{\pi}}   =\left(\mathrm{Rh^D_\gamma}_* {q_\CM}_* \BBC \right)^{G_{\pi}}
 = \mathrm{Rh^D_\gamma}_*\left( {q_\CM}_* \BBC \right)^{G_{\pi}} \xrightarrow{\simeq} \mathrm{Rh^D_\gamma}_*\BBC
\]
where the last isomorphism is induced by (\ref{eqn13}). 
\end{proof}

\section{Support theorems for Hitchin moduli spaces}
\subsection{Supports}

Let $f: X\to Y$ be a proper morphism between nonsingular quasi-projective varieties. The decomposition theorem by Bernstein, Beilinson, Deligne, and Gabber \cite{BBD} implies that the direct image complex $\mathrm{Rf}_\ast \BBC$ is (non-canonically) isomorphic to a direct sum of shifted simple perverse sheaves,
\begin{equation}\label{bbd}
\mathrm{Rf}_* \BBC \simeq \bigoplus_{Z_i} \mathrm{IC}_{Z_i}(L_i)[d_i] \in D^b_c(Y),
\end{equation}
where $d_i \in \BZ$, $Z_i \subset Y$ is an irreducible subvariety, and $L_i$ is a local system on an open subset $U_i$ of $Z_i$. Every $Z_i$ above is called a \emph{support} of $f: X\to Y$. We say that a direct summand $\CF$ of the object
\[
\mathrm{Rf}_* \BBC = \CF \oplus \CF'
\]
has \emph{full support} if each perverse constituent $ \mathrm{IC}_{Z_i}(L_i)[d_i]$ of $\CF$ has support $Z_i = Y$.

In this section, we analyze the supports of various Hitchin fibrations introduced in Section 1 when $\mathrm{deg}(D)>2g-2$.



\subsection{Support theorems}\label{Section2.2}
For the Hitchin fibration $h^D$ (resp. $h^D_\pi$), we define the elliptic locus of the Hitchin bases $\CA^D$ (resp. $\CA^D(\pi)$), denoted by $\CA^{D,\mathrm{ell}}$ (resp. $\CA^{D,\mathrm{ell}}(\pi))$, to be the open subset consisting of integral spectral curves.

 Following the methods of Ng\^{o} \cite{Ngo} and Chaudouard--Laumon \cite{CL}, de Cataldo showed in \cite{dC_SL} that all the supports for the $\mathrm{SL}_n$-Hitchin fibration (\ref{SLn_Hit}) are governed by the elliptic locus $\CA^{D,\mathrm{ell}} \subset \CA^D$ when $\mathrm{deg}(D)>2g-2$. 

\begin{thm} [{\cite[Theorem 1.0.2]{dC_SL}}] \label{SuppThm1}
Let $D$ be an effective divisor on $C$ of degree $\mathrm{deg}(D) >2g-2$. Then the generic points of the supports of 
\[
h^D: \CM^D_{n,L} \to \CA^D
\]
are contained in $\CA^{D,\mathrm{ell}} \subset \CA^D$.
\end{thm}

Now we consider the $\Gamma$-action on the moduli space $\CM^D_{n,L}$. This action is fiberwise with respect to the Hitchin map $h^D$, which induces a canonical decomposition of the direct image complex 
\begin{equation}\label{kappa_decomp}
\mathrm{Rh^D_*} \BBC = \bigoplus_{\kappa \in \hat{\Gamma}} \left(\mathrm{Rh^D_*} \BBC \right)_\kappa \in D^b_c(\CA^D);
\end{equation}
see \cite[Lemma 3.2.5]{LauNgo}. We define the \emph{stable part} $\left(\mathrm{Rh^D_*} \BBC \right)_{\mathrm{st}}$ as the component in (\ref{kappa_decomp}) corresponding to the trivial character $\kappa = 0 \in \hat{\Gamma}$.

Combining Theorem \ref{SuppThm1} with Ng\^o's support theorems over the elliptic loci \cite[Theorems 7.8.3 and 7.8.5]{Ngo}, we obtain the following complete description of the supports for every $\kappa$-part of (\ref{kappa_decomp}).

\begin{cor}\label{cor2.2}
The only support of $\left(\mathrm{Rh^D_*} \BBC \right)_\kappa$ is $\CA^D_\gamma$ where $\gamma$ corresponds to $\kappa$ via (\ref{Weil}).
\end{cor}

The trivial character $0 \in \hat{\Gamma}$ corresponds to $0 \in \Gamma$, and $\CA^D_{0} = \CA^D$. Hence as a special case of Corollary \ref{cor2.2}, the stable part $\left(\mathrm{Rh^D_*} \BC \right)_{\mathrm{st}}$ has full support $\CA^D$. 

We fix a character $\kappa\in \hat{\Gamma}$ of order $m$. Let $\pi: C' \to C$ be the cyclic Galois cover associated with $\kappa$ as in Section \ref{Weil_pair}. Assume $n = mr$. By the projection formula, we have 
\[
\mathrm{det}\left(\pi_*(\pi^*\CL \otimes \CE)\right) = \mathrm{det}\left( \CL \otimes \pi_*\CE \right) =  \CL^{\otimes mr} \otimes \mathrm{det}(\pi_*\CE) =\mathrm{det}(\pi_*\CE)
\]
for $\CL \in \Gamma =\mathrm{Pic}^0(C)[n]$ and $\CE$ a rank $r$ vector bundle on $C'$.
Therefore, the group $\Gamma$ acts on the moduli space $\CM^D_{r,L}(\pi)$ via tensor product
\[
\CL \cdot (\CE , \theta) = (\pi^*\CL \otimes \CE, \theta),\quad \CL \in \Gamma.
\]
We have similarly a $\kappa$-decomposition as (\ref{kappa_decomp}) for the Hitchin fibration (\ref{eqn_h_pi}) associated with $\CM^D_{r,L}(\pi)$,
\[
\mathrm{Rh_\pi^D}_* \BBC = \bigoplus_{\kappa \in \hat{\Gamma}} \left(\mathrm{Rh_\pi^D}_* \BBC \right)_\kappa \in D^b_c(\CA^D(\pi)).
\]

The main result of this section is to prove a support theorem for the Hitchin map (\ref{eqn_h_pi}) associated with $\pi: C' \to C$. 


\begin{thm}\label{SuppThm2}
Let $D$ be an effective divisor on $C$ of degree $\mathrm{deg}(D) >2g-2$. Assume that the degree $m$ Galois cover $\pi:C' \to C$ is associated with $\kappa \in \hat{\Gamma}$ via (\ref{Weil}). Assume $n = mr$. Then we have the following concerning the supports of the Hitchin map $h_\pi^D: \CM_{r,L}^D(\pi) \to \CA^D(\pi)$.
\begin{enumerate}
    \item[(a)] The generic points of the supports of ${\mathrm{Rh^D_\pi}}_* \BBC$ are contained in the elliptic locus $\CA^{D,\mathrm{ell}}(\pi)$.
    \item[(b)] The stable part $\left( \mathrm{Rh_\pi^D}_* \BC \right)_\mathrm{st}$ has full support $\CA^D(\pi)$.
\end{enumerate}
\end{thm}

The validity of support theorems (Theorems \ref{SuppThm1} and \ref{SuppThm2}) are the main ingredients in the proof of the endoscopic decomposition in the twisted case of $\mathrm{deg}(D) >2g-2$ when $\mathrm{deg}(D)$ is even; see the proof of Theorem \ref{main1}.

\subsection{Weak abelian fibrations} \label{Sec2.3}We recall the notion of weak abelian fibration introduced in \cite{Ngo}, which is modelled on the properties of Hitchin fibrations. 

We follow closely the exposition of \cite[Section 2.6]{dC_SL}. Let $A$ be an irreducible nonsingular variety. Assume that 
\[
h: M \to A, \quad g: P \to A
\]
are morphisms of the same relative dimensions $\mathrm{dim}(h) = \mathrm{dim}(g)$ satisfying the following properties:
\begin{enumerate}
    \item[(a)] The map $g: P\to A$ is a smooth commutative group scheme with connected fibers;
    \item[(b)] The map $h$ is proper, and $M$ is nonsingular;
    \item[(c)] The group scheme $P$ acts on $M$ fiberwise with affine stabilizers for every geometric point of $M$.
\end{enumerate}

We call a triple $(M,A,P)$ as above a \emph{weak abelian fibration}, if the Tate module 
\[
T_{\overline{\BQ}_l}(P) = \mathrm{R^{2d_g-1}g}_!\overline{\BQ}_l(d_g)
\]
as an $l$-adic sheaf is polarizable; see \cite[Section 2.6]{dC_SL} the paragraph following Lemma 2.6.2. 

Over a closed point $a \in A$, we consider the Chevalley decomposition for the restricted group scheme $P_a$,
\[
0 \to P_a^{\mathrm{aff}} \to P_a \to  P_a^{\mathrm{ab}} \to 0,
\]
where $P_a^{\mathrm{aff}}$ is the maximal connected affine linear subgroup of $P_a$ and $P_a^{\mathrm{ab}}$ is an abelian variety. 


We recall in the following the support inequality of Ng\^o \cite{Ngo}. 

\begin{thm}[{{\cite[Theorem 7.2.2]{Ngo}}}] \label{supp_ineq}
Let (M,A,P) be a weak abelian fibration. Assume the irreducible subvariety $Z \subset A$ is a support of $h: M \to A$, then 
\begin{equation}\label{supp_in}
\mathrm{dim}(h) -\mathrm{dim}(A) +\mathrm{dim}(Z) \geq d_Z^{\mathrm{ab}}(P).
\end{equation}
Here $d_Z^{\mathrm{ab}}(P)=\mathrm{dim} (P_a^{\mathrm{ab}})$ with $a \in Z$ a general point. 
\end{thm}

\subsection{Hitchin fibrations} Following \cite{CL, dC_SL}, we show that the Hitchin fibration 
\begin{equation}\label{eqn(18)}
h_\pi^D: \CM_{r,L}^D(\pi) \to \CA^D(\pi)
\end{equation}
associated with $\pi: C' \to C$ admits the structure of a weak abelian fibration. 

Recall the commutative diagram (\ref{diag12}). The $\mathrm{GL}_r$-Hitchin base $\widetilde{\CA}^{D'}(C')$ parameterizes spectral curves in the total space $V(D')$ of the line bundle $\CO_{C'}(D')$. We assume that
\begin{equation}\label{CC}
\CC \to \widetilde{\CA}^{D'}(C')
\end{equation}
is the universal spectral curve. Let $g_\CC: \mathrm{Pic}^0_\CC \to \widetilde{\CA}^{D'}(C')$ be the relative degree 0 Picard scheme associated with (\ref{CC}), which acts on $\widetilde{\CM}^{D'}_{r,d}(C')$ via tensor product. The following result is obtained in \cite{CL}; see also \cite[Section 3]{dC_SL} for a detailed review.

\begin{prop}[\cite{CL}]
The triple 
\begin{equation} \label{GL_weak}
\left(\widetilde{\CM}^{D'}_{r,d}(C'), \widetilde{\CA}^{D'}(C'), \mathrm{Pic}^0_\CC \right), \quad \widetilde{h}^{D'}:\widetilde{\CM}^{D'}_{r,d}(C') \rightarrow \widetilde{\CA}^{D'}(C'), \quad g_\CC: \mathrm{Pic}^0_\CC \to \widetilde{\CA}^{D'}(C')
\end{equation}
is a weak abelian fibration.
\end{prop}

In order to study (\ref{eqn(18)}), we ``fix the determinant" after pushing forward along $\pi: C' \to C$. Since the relative Hitchin moduli space $\CM_{r,L}^D(\pi)$ is a closed fiber of (\ref{eqn8}), we consider the morphism of $\CA^D(\pi)$-group schemes
\[
N_\pi: \mathrm{Pic}^0_\CC \times_{\widetilde{\CA}^{D'}(C')}\CA^D(\pi) \to \mathrm{Pic}^0(C)\times \CA^D(\pi)
\]
given by composition of the $\CA^D(\pi)$-morphisms
\[
\mathrm{Pic}^0_\CC \times_{\widetilde{\CA}^{D'}(C')}A^D(\pi) \to \mathrm{Pic}^0(C') \times \CA^D(\pi) \rightarrow \mathrm{Pic}^0(C) \times \CA^D(\pi).
\]
Here the first map is the restriction of the norm map $N_p$ (\cite[(42)]{dC_SL}) to $A^D(\pi)$, and the second map is 
\[
(\CL , a) \mapsto (\mathrm{det}(\pi_*\CL), a).
\]
By the discussion of \cite[Section 4.1]{dC_SL}, the map $N_\pi$ is smooth. We let 
\begin{equation}\label{Prym}
g_\pi: P \rightarrow \CA^D(\pi) 
\end{equation}
be the kernel of $N_\pi$, and we let
the group scheme $g^0_\pi: P^0 \to \CA^D(\pi)$ be the identity component of $P$. The nonsingular group schemes $P$ and $P^0$ act on $\CM^D_{r,L}(\pi)$ induced by the $\mathrm{Pic}^0_\CC$-action on $\widetilde{\CM}^{D'}_{r,d}(C')$.

\begin{prop}\label{prop2.6}
The triple
\begin{equation}\label{End_weak}
    \left( \CM^D_{r,L}(\pi), A^D(\pi), P^0 \right), \quad h_\pi^D:  \CM^D_{r,L}(\pi)\to A^D(\pi), \quad g^0_\pi: P^0 \to A^D(\pi)
\end{equation}
is a weak abelian fibration.
\end{prop}

\begin{proof}
The weak abelian fibration structure for (\ref{End_weak}) is essentially inherited from that for $(\ref{GL_weak})$. The proof is parallel to \cite[Section 4]{dC_SL}. Here we summarize some necessary minor modifications.

It is clear that (a) and (c) in Section \ref{Sec2.3} follow from the construction. We need to verify (b), and show that the Tate module associated with $g_\pi: P \to \CA^D(\pi)$ is polarizable.
\subsubsection*{(i) Affine stabilizers.} For a closed point in $\CM^D_{r,L}(\pi) \subset \widetilde{\CM}^{D'}_{r,d}(C')$, its stabilizer with respect to the $P^0$-action on $\CM^D_{r,L}(\pi)$ is a subgroup of the corresponding stabilizer with respect the $\mathrm{Pic}^0_\CC$-action on $\widetilde{\CM}^{D'}_{r,d}(C')$, whose affineness follows from the fact that (\ref{GL_weak}) is a weak abelian fibration.
\subsubsection*{(ii) Polarizability of the Tate module.}
This follows from the proof of \cite[Theorem 4.7.2]{dC_SL}. In fact, for a closed point $a\in \CA^D(\pi)$, the Tate module of the abelian part $P^{\mathrm{ab}}_a$ is an orthogonal direct summand component of $T_{\overline{\BQ}_l}(\mathrm{Pic}^{0,\mathrm{ab}}_{\CC,a})$ with respect to the non-degenerate Tate-Weil pairing on $T_{\overline{\BQ}_l}(\mathrm{Pic}^{0,\mathrm{ab}}_{\CC,a})$. Hence the restriction of the Tate-Weil pairing on $T_{\overline{\BQ}_l}(\mathrm{Pic}^{0,\mathrm{ab}}_{\CC,a})$ to $T_{\overline{\BQ}_l}(P^\mathrm{ab}_a)$ is non-degenerate.
\end{proof}

By \cite[Section 9]{CL} (see also \cite[Section 5.2]{dC_SL}), the $\mathrm{GL}_r$-Hitchin base admits a stratification
\[
\widetilde{\CA}^{D'}(C') = \bigsqcup_{\underline{m}, \underline{n}}\widetilde{\CA}_{\underline{m}, \underline{n}}
\]
with $\underline{m} = (m_1, m_2, \dots, m_s)$, $\underline{n}=(n_1, n_2, \dots, n_s)$ satisfying 
\begin{enumerate}
    \item[(a)] $n_i \geq n_{i+1}$ for any $i$;
    \item[(b)] $m_i \geq m_{i+1}$ whenever $n_i = n_{i+1}$;
    \item[(c)] $\sum_{i=1}^s m_in_i = r$.
\end{enumerate}
Each $\widetilde{\CA}_{\underline{m}, \underline{n}}$ is a locally closed subset formed by spectral curves of the topological type $(\underline{m}, \underline{n})$:
\begin{equation}\label{strat}
\widetilde{\CA}_{\underline{m}, \underline{n}} = \left\{ E \subset V(D'): ~~~ E = \sum_i m_iE_i, ~~~ E_i \subset V(D')\right\}
\end{equation}
where $V(D')$ is the total space of $\CO_{C'}(D')$ and $E_i$ is an integral spectral curve of degree $n_i$ over $C'$. The stratification (\ref{strat}) induces a stratification on $\CA^D(\pi) \subset \CA^{D'}(C')$,
\[
\CA^D(\pi) = \bigsqcup_{\underline{m}, \underline{n}}{\CA}(\pi)_{\underline{m}, \underline{n}}, \quad {\CA}(\pi)_{\underline{m}, \underline{n}} = \CA^D(\pi) \cap \widetilde{\CA}_{\underline{m}, \underline{n}}.
\]

We have the following multi-variable inequality.

\begin{prop}[\emph{c.f.} {\cite[Corollary 5.4.4]{dC_SL}}]\label{prop2.7}
Let $Z \subset \CA^D(\pi)$ be an irreducible subvariety whose general points are of the type $(\underline{m}, \underline{n})$. Then we have
\begin{equation}\label{delta_ineq}
d^{\mathrm{ab}}_Z\left(\mathrm{Pic}^0_\CC\right) \geq \sum_i \left( {d_{\widetilde{h}^{D'}_{n_i}(C')}}- d_{\widetilde{\CA}^{D'}_{n_i}(C')} \right) + \mathrm{dim}(Z)+ \left(\mathrm{deg}(D)-g+1\right).
\end{equation}
Here ${d_{\widetilde{h}^{D'}_{n_i}(C')}}$ and $d_{\widetilde{\CA}^{D'}_{n_i}(C')}$ are the dimensions of a fiber and the base respectively of the $\mathrm{GL}_{n_i}$-Hitchin fibration (\ref{GL_Hit}) associated with the curve $C'$ and the divisor $D'=\pi^*D$, and \[
d^{\mathrm{ab}}_Z\left(\mathrm{Pic}^0_\CC\right) = \mathrm{dim}\left((\mathrm{Pic}^{0}_{\CC,a})^\mathrm{ab}\right)
\]
with $a \in Z$ a general point.
\end{prop}

\begin{proof}
When $\pi =\mathrm{id}$, the subspace
\begin{equation}\label{extra1}
\CA^D(\mathrm{id}) = \{\mathrm{char}(\theta) \in \widetilde{\CA}^D: ~~~\mathrm{trace}(\theta)=0\} \subset \widetilde{\CA}^D
\end{equation}
coincides with the $\mathrm{SL}_n$-Hitchin base, and the inequality (\ref{delta_ineq}) is the $\delta$-inequality for $\mathrm{SL}_n$ proven in \cite[Corollary 5.4.4 (76)]{dC_SL}. As explained in the last paragraph of \cite[Proof of Corollary 5.4.4]{dC_SL}, the main ingredient in the proof of enhancing the $\delta$-inequality for $\mathrm{GL}_n$ to that for $\mathrm{SL}_n$ is \cite[Theorem 5.4.2]{dC_SL}, which asserts that the restriction of the $\delta$-regular weak abelian fibrations to their elliptic loci in the sub-Hitchin base (\ref{extra1}) remain $\delta$-regular. This follows from the product structure \cite[(74)]{dC_SL} for the group schemes associated with the spectral curves, which says that the variations of the group schemes associated with the spectral curves are trivial along the $H^0(C, \CO_C(D))$-direction.

Now for a general $\pi: C' \to C$ with $D' =\pi^*D'$ as we consider here, we have the canonical decomposition for the Hitchin base
\begin{equation}\label{extra2}
\CA^D(\pi) \oplus H = \widetilde{\CA}^{D'}
\end{equation}
with $H = H^0(C, \CO_C(D))$ a direct summand component of $H^0(C', \CO_{C'}(D'))$. Applying the product structure \cite[(74)]{dC_SL} to the curve $C'$ and the divisor $D'$, we obtain an analogous product structure for (\ref{extra2}), that the variations of the group schemes associated with the spectral curves are trivial along the $H$-direction. Hence (\ref{delta_ineq}) holds in the relative case $\pi: C' \to C$ by the same reason as for $\mathrm{SL}_n$.
\end{proof}

\subsection{Proof of Theorem \ref{SuppThm2} (a)}
The proof of Theorem \ref{SuppThm2} (a) is parallel to the proofs of the main theorems in \cite{CL, dC_SL}, which we provide in the following for the reader's convenience. The crucial ingredient is to combine Ng\^{o}'s support inequality (\ref{supp_in}) and the multi-variable inequality (\ref{delta_ineq}).  

We assume $Z \subset \CA^D(\pi)$ is an irreducible support of (\ref{eqn(18)}) whose general points have types $(\underline{m}, \underline{n})$. By Theorem \ref{supp_ineq}, Proposition \ref{prop2.6}, and Proposition \ref{prop2.7}, we have
\begin{equation*}
\begin{split}
\mathrm{dim}({h^D_\pi}) - &  \mathrm{dim}({A^D(\pi)})+ \mathrm{dim}(Z)  \geq d^{\mathrm{ab}}_Z\left(\mathrm{Pic}^0_\CC\right) - g \\
& \geq \sum_i \left( {d_{\widetilde{h}^{D'}_{n_i}(C')}}- d_{\widetilde{\CA}^{D'}_{n_i}(C')} \right) + \mathrm{dim}(Z) +(\mathrm{deg}(D)-g+1)-g
\end{split}
\end{equation*}
where we use $d^{\mathrm{ab}}(P^0_a) = d^{\mathrm{ab}}(\mathrm{Pic}^0_{\CC,a})-g$. Hence 
\begin{equation}\label{eqn25}
\mathrm{dim}({h^D_\pi}) - \mathrm{dim}({\CA^D(\pi)}) \geq \sum_{i=1}^s \left( {d_{\widetilde{h}^{D'}_{n_i}(C')}}- d_{\widetilde{\CA}^{D'}_{n_i}(C')} \right)+(\mathrm{deg}(D)-2g+1).
\end{equation}
Here we recall that $s$ is the number of  irreducible components of the spectral curve. We now apply the dimension formulas for $\mathrm{GL}_*$-Hitchin fibrations \cite[Section 6.1]{dC_SL} to compute both sides of (\ref{eqn25}). The left-hand side is equal to:
\begin{align*}
 \mathrm{LHS} & =  \left({d_{\widetilde{h}^{D'}_{r}(C')}}- d_{\widetilde{\CA}^{D'}_{r}(C')} \right)  + \left( \mathrm{dim}H^0(C, \CO_C(D))-g\right)\\
 & = \left(-r\mathrm{deg}(D') + 2r(g'-1)+1\right) +\left(\mathrm{deg}(D)-2g+1\right)
\end{align*}
where $g'$ is the genus of $C'$ and we used the last formula of \cite[(77)]{dC_SL}. Similarly for the right-hand side we have:

\begin{align*}
 \mathrm{RHS} & = \sum_{i=1}^s\left(-n_i\mathrm{deg}(D') + 2n_i(g'-1)+1\right) +\left(\mathrm{deg}(D)-2g+1\right) \\
 & = \left( (-\sum_{i=1}^sn_i)\left((\mathrm{deg}(D')-2(g'-1)\right)+s \right) +\left(\mathrm{deg}(D)-2g+1\right).
 \end{align*}
In particular, (\ref{eqn25}) implies that 
\begin{equation}\label{eqn26}
1-s \geq  \left(r-\sum_in_i \right)\left(\mathrm{deg}(D')-(2g'-2)\right).
\end{equation}
Since
\[
\mathrm{deg}(D')-(2g'-2) = \mathrm{deg}(\pi) (\mathrm{deg}(D)-(2g-2))>0
\]
by the assumption on $D$, the inequality (\ref{eqn26}) forces that $s=1$ and $m_1=1$. This implies that the generic point of $Z$ lies in $\CA^{D,\mathrm{ell}}(\pi)$, which completes the proof of Theorem \ref{SuppThm2} (a). \qed

\subsection{Proof of Theorem \ref{SuppThm2} (b)} Due to Theorem \ref{SuppThm2} (a), it suffices to prove (b) over the elliptic locus with respect to the restricted Hitchin map
\begin{equation}\label{ell_h}
h^{\mathrm{ell}}_\pi = h_\pi^D|_{\CA^{D,\mathrm{ell}}(\pi)}: \CM^{D,\mathrm{ell}}_{r,L}(\pi) \to \CA^{D,\mathrm{ell}}(\pi).
\end{equation}

Recall the group scheme (\ref{Prym}) of the relative Prym variety. By the support theorem \cite[Propositions 7.2.2 and 7.2.3]{Ngo}, we only need to show that the constructible sheaf of the top degree cohomology 
\begin{equation}\label{st_sh}
    \left(\mathrm{R^{2d_{h_\pi^D}}h^{\mathrm{ell}}_\pi}_* \BBC\right)_\mathrm{st},\quad  d_{h_\pi^D} = \mathrm{dim}(h_\pi^D)
\end{equation}
is the trivial local system on the elliptic locus $\CA^{D,\mathrm{ell}}(\pi)$.

It is clear that (\ref{st_sh}) contains the trivial local system given by the sum of point classes for all irreducible components of the fibers of (\ref{ell_h}). Hence it suffices to prove that the stable part of the degree $2d_{h_\pi^D}$ cohomology group is 1-dimensional for each fiber of (\ref{ell_h}).

Assume $a \in \CA^{D,\mathrm{ell}}(\pi)$. Let $C'_a$ be the corresponding integral spectral curve with the spectral cover
\begin{equation}\label{eqn32}
C'_a \xrightarrow{p_a} C' \xrightarrow{\pi} C.
\end{equation}
The Hitchin fiber $\CM^{D}_{r,d}(\pi)_a$ contains a Zariski dense open subset
\[
\CM^{\mathrm{reg}}_a \subset \CM^{D}_{r,L}(\pi)_a 
\]
parameterizing line bundles on the spectral curve $C'_a$, which is a torsor of the group scheme $P_a$. Hence  $\CM^{D}_{r,L}(\pi)_a$ has $|\pi_0(P_a)|$ irreducible components.

We fix a base point in $x\in \CM_a^{\mathrm{reg}}$. Since $\CM_a^{\mathrm{reg}}$ is a torsor of the group scheme $P_a$, the base point $x$ trivializes the torsor and therefore the $\pi_0(P_a)$-action on $x$ yields an isomorphism
\[
[x]: \pi_0(\CM^{\mathrm{reg}}_a)  \xrightarrow{\simeq} \pi_0(P_a),
\]
which further induces
\begin{equation}\label{eqn29}
H^{2d_{h_\pi^D}}\left(\CM^{D}_{r,L}(\pi)_a, \BC \right) = \bigoplus_{v\in \pi_0(P_a)} \BC v.
\end{equation}
The action of $P_a$ on the cohomology $H^*\left(\CM^{D}_{r,L}(\pi)_a, \BC \right)$ factors through the group $\pi_0(P_a)$ of connected components, which acts naturally on the right-hand side of (\ref{eqn29}). In particular, the action of the discrete subgroup $\Gamma \subset P_a$ on (\ref{eqn29}) factors through the natural action of $\pi_0(P_a)$. By the proof of \cite[Theorem 1.1 (2)]{HP}, the morphism
\[
\Gamma = \mathrm{Pic}^0(C)[n]  \twoheadrightarrow  \pi_0(P_a)
\]
induced by the pullback $p_a^*\circ\pi^*$ along (\ref{eqn32}) is a surjection. Therefore we obtain that
\begin{equation*}
H^{2d_{h_\pi^D}}\left(\CM^{D}_{r,L}(\pi)_a, \BC \right)_\mathrm{st} = H^{2d_{h_\pi^D}}\left(\CM^{D}_{r,L}(\pi)_a, \BC \right)^\Gamma \subseteq H^{2d_{h_\pi^D}}\left(\CM^{D}_{r,L}(\pi)_a, \BC \right)^{\pi_0(P_a)} = \BC,
\end{equation*}
where the last equality is given by the $\pi_0(P_a)$-equivariant isomorphism (\ref{eqn29}). This implies
\begin{equation}\label{EMG}
\left(\mathrm{R^{2d_{h_\pi^D}}h^{\mathrm{ell}}_\pi}_* \BBC\right)_\mathrm{st} = \BBC
\end{equation}
and completes the proof of (b). \qed

\begin{rmk}
The vector space
\[
H^{2d_{h_\pi^D}}\left(\CM^{D}_{r,L}(\pi)_a, \BC \right)_\mathrm{st} = H^{2d_{h_\pi^D}}\left(\CM^{D}_{r,L}(\pi)_a, \BC \right)^\Gamma 
\]
may fail to be 1-dimensional when $a \in \CA^D(\pi) \setminus \CA^{D,\mathrm{ell}}(\pi)$. In particular, the constructible sheaf $\left(\mathrm{R^{2d_{h_\pi^D}}h^D_\pi}_* \BBC\right)_\mathrm{st}$ 
is not a rank 1 local system over the total Hitchin base $\CA^D(\pi)$. Hence the proof of Theorem \ref{SuppThm2} (b) relies heavily on the support theorem --- Theorem \ref{SuppThm2} (a).
\end{rmk}

\subsection{Transfer from the $\kappa$-part to the stable part}\label{section2.7}
In Section \ref{section2.7}, we assume that $D$ is an effective divisor of degree $\mathrm{deg}(D) >2g-2$ or $D= K_C$. Our main purpose is to prove Proposition \ref{prop2.8} which allows us to transfer naturally from the $\kappa$-part to the stable part of the complex $\mathrm{Rh_\pi^D}_* \BBC$. This extends \cite[Proposition 2.3.2]{Yun3} to the total Hitchin base for certain endoscopic Hitchin moduli spaces associated with $\mathrm{SL}_n$. We note that this transfer does not 
rely on the support theorem.

Recall the decomposition (\ref{M_i_decomp}) of Proposition \ref{prop1.1} (a). By Lemma \ref{lem1.5}, the group $\Omega$ (introduced in (\ref{Omega})) acts on each direct image complex ${\mathrm{Rh}_i}_* \BBC$, and we consider its $\Omega$-invariant part
\[
({\mathrm{Rh}_i}_* \BBC)^\Omega \in D_c^b(\CA^D(\pi)).
\]
For any pair $1 \leq i,j \leq m$, the isomorphism of Proposition \ref{prop1.1} (b):
\[
\phi_{ij} = \phi_{\CL_{ij}}: M_i \xrightarrow{\simeq} M_j, \quad (\CE, \phi) \mapsto (\CE \otimes \CL_{ij}, \theta)
\]
induced by a line bundle $\CL_{ij} \in \Gamma$ yields an isomorphism
\begin{equation*}
\phi_{ij}^*: {\mathrm{Rh}_j}_* \BBC \xrightarrow{\simeq} {\mathrm{Rh}_i}_* \BBC.
\end{equation*}
It preserves the $\Omega$-invariant parts, 
\begin{equation}\label{add2}
\phi_{ij}^*:  ({\mathrm{Rh}_j}_* \BBC)^\Omega  \xrightarrow{\simeq} ({\mathrm{Rh}_i}_* \BBC)^\Omega.
\end{equation}
We note that the isomorphism (\ref{add2}) only depends on the class of the line bundle $\CL_{ij} \in \Gamma$ in the quotient group 
\[
\pi_0(\mathrm{Prym}(C'/C)) = \Gamma / \Omega.
\]
Hence the $\Gamma$-action on $\bigoplus_{j=1}^m ({\mathrm{Rh}_j}_* \BBC)^\Omega$ passes through $\pi_0(\mathrm{Prym}(C'/C))$. Since the group $\Omega$ preserves each component $M_i$, it follows from Proposition \ref{prop1.1} (b) that the elements of the cyclic group $\pi_0(\mathrm{Prym}(C'/C)) = \Gamma/\Omega$ acts transitively on the set $\{M_i\}_{i=1}^m$. We may view $\pi_0(\mathrm{Prym}(C'/C))$ as the group of connected components of $\CM^D_{r,L}(\pi)$. For any fixed $1\leq i_0 \leq m$, the isomorphisms (\ref{add2}) yields a canonical $\pi_0(\mathrm{Prym}(C'/C))$-equivariant isomorphism
\begin{equation}\label{eqn41}
\bigoplus_{j=1}^m ({\mathrm{Rh}_j}_* \BBC)^\Omega = (\mathrm{Rh_{i_0}}_*\BBC)^\Omega \otimes \left(\bigoplus_{v\in \pi_0(\mathrm{Prym}(C'/C))} \BC v \right)
\end{equation}
where the action on the right-hand side is the natural one.

Before stating Proposition \ref{prop2.8}, we introduce the following definition for convenience.

\begin{defn}\label{definition}
Let $X$ be an algebraic variety, and let $\CF_1, \CF_2 \in D^b_c(X)$ be two objects. We say that two morphisms
\[
f_1: \CF_1 \to \CF_2, \quad f_2: \CF_1 \to \CF_2
\]
are \emph{equivalent up to scaling}, if there exists $\lambda \in \BC^*$ such that $f_1 = \lambda f_2$. We say that there is an isomorphism between two objects $\CF_1$ and $\CF_2$:
\[
f: \CF_1 \to \CF_2
\]
which is \emph{canonical up to scaling}, if our construction induces a set of isomorphisms $f_i: \CF_1 \xrightarrow{\simeq} \CF_2$ which are all equivalent up to scaling.
\end{defn}

\begin{prop}\label{prop2.8}
Let $D$ be an effective divisor on $C$ of degree $\mathrm{deg}(D) >2g-2$ or $D= K_C$. Assume that $\pi:C' \to C$ is the degree $m$ Galois cover associated with $\gamma \in \Gamma$, which corresponds to the character $\kappa \in \hat{\Gamma}$ via (\ref{Weil}). Assume $n = mr$. Then for any two elements $\kappa_1, \kappa_2$ in the cyclic group $\langle \kappa \rangle \subset \hat{\Gamma}$ generated by $\kappa$, there is an isomorphism for the corresponding isotypic components:
    \begin{equation}\label{pp2.8}
    \left( \mathrm{Rh_\pi^D}_* \BBC \right)_{\kappa_1} = \left( \mathrm{Rh_\pi^D}_* \BBC \right)_{\kappa_2}
    \end{equation}
which is canonical up to scaling. In particular, (\ref{pp2.8}) induces an isomorphism which is canonical up to scaling:
\[
\left( \mathrm{Rh_\pi^D}_* \BBC \right)_{\kappa} = \left( \mathrm{Rh_\pi^D}_* \BBC \right)_{\mathrm{st}}.
\]
\end{prop}

\begin{proof}
We consider the $\Omega$-invariant part 
\begin{equation}\label{eqnn42}
\left( \mathrm{Rh_\pi^D}_* \BBC \right)^\Omega \in D_c^b(\CA^D(\pi))
\end{equation}
of the direct image complex $\mathrm{Rh_\pi^D}_* \BC $. On one hand, the group $\Omega$ acts on each complex ${\mathrm{Rh}_i}_* \BC$ and we have
\begin{equation}\label{eqnn43}
\left( \mathrm{Rh_\pi^D}_* \BBC \right)^\Omega = \bigoplus_{j=1}^m ({\mathrm{Rh}_j}_* \BBC)^\Omega.
\end{equation}
On the other hand, by Corollary \ref{cor1.3}, an isotypic component $\left( \mathrm{Rh_\pi^D}_* \BBC \right)_{\kappa'}$ contributes to (\ref{eqnn42}) if and only if $\kappa'$ lies in $\langle \kappa \rangle$. Hence
\begin{equation}\label{eqnn44}
\left( \mathrm{Rh_\pi^D}_* \BBC \right)^\Omega = \bigoplus_{\kappa' \in \langle \kappa \rangle} \left( \mathrm{Rh_\pi^D}_* \BBC \right)_{\kappa'}.
\end{equation}
Combining (\ref{eqn41}), (\ref{eqnn43}), and (\ref{eqnn44}), we obtain a natural $\pi_0(\mathrm{Prym}(C'/C))$-equivariant isomorphism
\[
\bigoplus_{\kappa' \in \langle \kappa \rangle} \left( \mathrm{Rh_\pi^D}_* \BBC \right)_{\kappa'}  =  (\mathrm{Rh_{i_0}}_*\BBC)^\Omega \otimes \left(\bigoplus_{v\in \pi_0(\mathrm{Prym}(C'/C))} \BC v \right).
\]
In particular, if we take the $\kappa'$-parts on both sides, 
since the $\kappa'$-part of the regular representation is one-dimensional, this yields a natural isomorphism up to scaling
\begin{equation}\label{eqn45}
f_{i_0,\kappa'}: \left( \mathrm{Rh_\pi^D}_* \BBC \right)_{\kappa'} \xrightarrow{\simeq}  (\mathrm{Rh_{i_0}}_*\BBC)^\Omega. 
\end{equation}
This gives an isomorphism (\ref{pp2.8}) up to scaling, which, \emph{a priori}, still depends on the choice of $1\leq i_0 \leq m$. 

Different choices of $i_0$ influence the isomorphism (\ref{pp2.8}) via the action of an element 
\[
g\in \Gamma/\Omega = \pi_0(\mathrm{Prym}(C'/C))
\] 
on both objects of (\ref{eqn45}). After isolating the $\kappa'$-isotypic component, we conclude this only changes the isomorphism (\ref{pp2.8}) by a possible scalar ambiguity.
\end{proof}

\subsection{Changing the degree}
Assume $\mathrm{deg}(D) >2g-2$. As another application of Ng\^o's support theorem, we analyze the $G_\pi$-equivariant complex 
\begin{equation}\label{321}
\left(\mathrm{Rh_\pi^D}_* \BBC \right)_{\kappa} \in D_c^b(\CA^D(\pi))
\end{equation}
when the degree of the line bundle $L \in \mathrm{Pic}^d(C)$ is changed.

Due to Theorem \ref{SuppThm2} (b) and Proposition \ref{prop2.8}, the object (\ref{321}) has full support $\CA^D(\pi)$. Hence (\ref{321}) is completely determined by its restriction to the open subset $U^{\mathrm{sm}} \subset \CA^D(\pi)$ where the spectral curves are nonsingular.

Ng\^o's analysis of supports for direct image complexes \cite[Section 7]{Ngo}  works for each $\kappa$-part (see \cite[Proposition 7.2.3]{Ngo}). In particular, as a corollary of the ``freeness" \cite[Proposition 7.4.10]{Ngo}, the isomorphism class of the restriction of (\ref{321}) to $U^\mathrm{sm}$ is determined by the group scheme $P^0|_{U^\mathrm{sm}}$ of (\ref{End_weak}) and the constructible sheaf
\begin{equation}\label{kappash}
\left(\mathrm{R^{2d_{h_\pi^D}}h^D_\pi}_* \BBC\right)_{\kappa}\Big{|}_{U^{\mathrm{sm}}}  \in \mathrm{Sh}_c(U^{\mathrm{sm}})
\end{equation}
which are both equipped with the $G_\pi$-actions. See \cite[Appendix]{dCRS} for a precise form expressing (\ref{321}) in terms of the direct image complex associated with 
\[
g^0_{\pi}\Big{|}_{U^{\mathrm{sm}}}: P^0|_{U^\mathrm{sm}} \rightarrow U^{\mathrm{sm}}
\]
and (\ref{kappash}). 

The following proposition will only be used in Section \ref{Section5.5}.


\begin{prop}\label{Prop2.11}
Assume $\mathrm{deg}(D)>2g-2$. Let $q$ be an integer coprime to $n$. We have an isomorphism of the $G_\pi$-equivariant objects 
\begin{equation}\label{new1}
\left(\mathrm{Rh_{\pi,L}^D}_* \BBC \right)_{q\kappa} \simeq \left(\mathrm{Rh_{\pi,L^{\otimes q}}^D}_* \BBC \right)_{\kappa} \in D_c^b(\CA^D(\pi)).
\end{equation}
Here $h_{\pi,L^{\otimes q}}^D: \CM^D_{r,L^{\otimes q}}(\pi) \to \CA^D(\pi) $ stands for the Hitchin fibration associated with the line bundle $L^{\otimes q}$.
\end{prop}

\begin{proof}
For notational convenience, we use $e$ to denote $2d_{h_\pi^D}$. After restricting to $U^\mathrm{sm}$ we have 
\begin{equation}\label{neww}
\left(\mathrm{R^eh_{\pi,L^{\otimes q}}^D}_* \BBC \right)_{\kappa'} \simeq  \BBC, \quad \quad \forall \kappa' \in \langle \kappa \rangle
\end{equation}
by Proposition \ref{prop2.8} and (\ref{EMG}). 
We need to analyze the $G_\pi$-equivariant structure on the rank 1 trivial local systems (\ref{neww}).

Now we consider the constructible sheaf
\begin{equation}\label{eqnn51}
\mathrm{R^eh_{\pi,L^{\otimes q}}^D}_* \BBC \Big{|}_{U^\mathrm{sm}} \in \mathrm{Sh}_c(U^{\mathrm{sm}})
\end{equation}
with the $G_\pi$-equivariant structure. 

By Proposition \ref{prop1.1} (b), the sheaf (\ref{eqnn51}) is a trivial local system of rank $m$ (corresponding to the $m$ connected components of $\CM^D_{r,L^{\otimes q}}(\pi)$). We may view (\ref{eqnn51}) as an $m$-dimensional vector space $V_q \simeq \BC^m$ on which the cyclic groups $G_\pi$ and $\Gamma$ act. Therefore, to prove (\ref{new1}) we only need to show that the isotypic component $(V_q)_{\kappa}$ is $G_\pi$-equivariantly isomorphic to $(V_1)_{q\kappa}$.

Recall the degree $dq$ Prym variety $\mathrm{Prym}^{dq}(C'/C)$ associated with the line bundle $L^{\otimes q}$, whose $m$ connected components are identified with the $m$ connected components of $\CM^D_{r,L^{\otimes q}}(\pi)$. To connect $V_q$ and $V_1$, we consider the the ``multiplication by $q$" map:
\begin{equation}\label{mult_q}
\mathrm{mult}_q: \mathrm{Prym}^{d}(C'/C) \to \mathrm{Prym}^{dq}(C'/C), \quad \CL \mapsto \CL^{\otimes q}
\end{equation}
which is clearly $G_\pi$-equivariant.

We note that $\mathrm{mult}_q$ induces an identification of the $m$ connected components for the Prym varieties on both sides of (\ref{mult_q}). In fact, choosing base points $x \in \mathrm{Prym}^{q}(C'/C)$ and $qx \in \mathrm{Prym}^{dq}(C'/C)$ trivializes both $\mathrm{Prym}(C'/C)$-torsors, and the map $\mathrm{mult}_q$ induces a ``multiplication by $q$" map on the cyclic group $\pi_0(\mathrm{Prym}(C'/C))\simeq \BZ/m\BZ$. The claim follows from the fact that $\mathrm{gcd}(m,q)=1$. 

In conclusion, (\ref{mult_q}) induces a $G_\pi$-equivariant isomorphism
\[
[\mathrm{mult}_q]: V_1 \xrightarrow{\simeq} V_q
\]
whose $\Gamma$-action on the right-hand side is given by the $\Gamma$-action on the left-hand side composed with the ``multiplication by $q$" $\Gamma \xrightarrow{\cdot q} \Gamma$. In particular we have a $G_\pi$-equivariant isomorphism between $(V_q)_{\kappa}$ and $(V_1)_{q\kappa}$. This completes the proof of the proposition. \qedhere

\end{proof}



The constraint $\mathrm{deg}(D) > 2g-2$ will be removed by Remark \ref{rmk4.8}, despite the fact that we no longer have full supports in that case.

\section{Endoscopic decompositions}\label{Sec3}

\subsection{Overview: main results}\label{Sec3.1}

In Sections \ref{Sec3} and \ref{sec4}, we establish a generalized version of Theorem \ref{main} for any effective divisor $D$ satisfying either $\mathrm{deg}(D)$ is even and greater than $2g-2$, or $D = K_C$.\footnote{We refer to Remark \ref{odd_deg} for the case when $\mathrm{deg}(D)$ is odd and greater than $2g-2$.}

Let $D$ be as above, and let $\pi: C' \to C$ be a degree $m$ cyclic Galois cover with $n=mr$. Recall the Hitchin fibrations
\[
h^D: \CM^D_{n,L} \to \CA^D, \quad h_\pi^D: \CM^D_{r,L}(\pi) \rightarrow \CA^D(\pi),
\]
the fiberwise $\Gamma$-actions, and the corresponding $\kappa$-decompositions. The Galois group $G_\pi$ acts naturally on $\CM^D_{r,d}(\pi)$ and $\CA^D(\pi)$ such that the Hitchin map $h_\pi^D$ is $G_\pi$-equivariant; see Section \ref{Section1.3}. By Lemma \ref{lemma1.6}, the direct image complex $\mathrm{Rh_{\pi}^D}_*\BC$ is $G_\pi$-equivariant, so is each $\kappa$-isotypic part
\[
\left(\mathrm{Rh_{\pi}^D}_*\BBC\right)_{\kappa} \in D^b_c(\CA^D(\pi))
\]
due to the commutativity of the $\Gamma$- and the $G_\pi$-actions. We also note that 
\[
q_\CA^* \left( \mathrm{Rh^D_*} \BBC \right)_{\kappa}\in D^b_c(\CA^D(\pi))
\]
is naturally $G_\pi$-equivariant induced by the pullback map from the $G_\pi$-quotient
\begin{equation*}
q_\CA: \CA^D(\pi) \rightarrow {\CA^D_\gamma}.
\end{equation*}
Recall the open subsets  $\CA^D(\pi)^*$ and ${\CA^D_\gamma}^*$ for the Hitchin bases and the free $G_\pi$-quotient map between them from Section \ref{Section1.3}.

The following theorem is a generalization of Theorem \ref{main0}.

\begin{thm}\label{main0'}
Let $\kappa \in \hat{\Gamma}$ and $\gamma \in \Gamma$ be identified by (\ref{Weil}), let $\pi:C'\to C$ be the Galois cover associated with $\gamma$, and let $d^D_\gamma = \mathrm{codim}_{\CA^D}(\CA^D_\gamma)$. There are isomorphisms in $D^b_c(\CA^D(\pi)^*)$ which are canonical up to scaling:
\begin{equation}\label{main00}
   q_\CA^* \left( \mathrm{Rh^D_*} \BBC \right)_{\kappa}\Big{|}_{\CA^D(\pi)^*} \xrightarrow{\simeq} \left( \mathrm{Rh_\pi^D}_* \BBC \right)_{\kappa}\Big{|}_{\CA^D(\pi)^*}[-2d^D_\gamma]  \xrightarrow{\simeq} \left(\mathrm{Rh_{\pi}^D}_*\BBC\right)_{\mathrm{st}}\Big{|}_{\CA^D(\pi)^*} [-2d^D_\gamma]   
 \end{equation}
with the first isomorphism $G_\pi$-equivariant.
\end{thm}

The second isomorphism of (\ref{main00}) is obtained immediately from Proposition \ref{prop2.8}, which actually holds over the total space $\CA^D(\pi)$:
\[
\left( \mathrm{Rh_\pi^D}_* \BBC \right)_{\kappa}[-2d^D_\gamma]  \xrightarrow{\simeq} \left(\mathrm{Rh_{\pi}^D}_*\BBC\right)_{\mathrm{st}}[-2d^D_\gamma].   
\]
The following theorem is a sheaf-theoretic enhancement of the Hausel--Thaddeus conjecture. 

\begin{thm}\label{thm3.2}
Let $i^D_\gamma: \CA^D_\gamma \hookrightarrow \CA^D$ be the closed embedding. With the same notation as in Theorem \ref{main0'}, there is an isomorphism which is canonical up to scaling:
\begin{equation}\label{ck^D}
c^D_\kappa: \left(\mathrm{Rh^D_*} \BBC\right)_\kappa \xrightarrow{\simeq} {i^D_\gamma}_*\left(\mathrm{Rh^D_\gamma}_* \BBC\right)_\kappa [-2d^D_\gamma] \in D_c^b(\CA^D).
\end{equation}
\end{thm}

We first observe that (\ref{ck^D}) induces the first isomorphism of (\ref{main00}). So Theorem \ref{main0'} is recovered by Theorem \ref{thm3.2}. In fact, we restrict $c^D_\kappa$ to the open subset ${\CA_\gamma^D}^* \subset \CA_\gamma^D$ and pull it back along
the free $G_\pi$-quotient map
\begin{equation}\label{extra3}
q_\CA: \CA^D(\pi)^* \to {\CA^D_\gamma}^*.
\end{equation}
Since the left diagram of (\ref{diagram111}) is Cartesian restricting to (\ref{extra3}), we recover the first map of (\ref{main00}) via proper base change.

Theorems \ref{main0'} and \ref{thm3.2} recover Theorems \ref{main0} and \ref{main} when $D = K_C$. When $\mathrm{deg}(D)$ is even and greater than $2g-2$, Theorem \ref{thm3.2} provide a concrete description of the contribution of each support $\CA^D_\gamma$ to 
\[
 \mathrm{Rh^D_*} \BBC \in D_c^b(\CA^D).
\]
This enhances the main theorem of de Cataldo \cite{dC_SL} when $\mathrm{deg}(D)$ is even.

As discussed above, for proving Theorems \ref{main0'} and \ref{thm3.2}, we only need to construct $G_\pi$-equivariant isomorphisms
\[
c^D_\kappa: \left(\mathrm{Rh^D_*} \BBC\right)_\kappa \xrightarrow{\simeq} {i^D_\gamma}_*\left(\mathrm{Rh^D_\gamma}_* \BBC\right)_\kappa [-2d^D_\gamma] \in D_c^b(\CA^D),
\]
which we treat in this section for the following special cases. 

\begin{thm}\label{main1}
Theorems \ref{main0'} and \ref{thm3.2} hold when $\mathrm{deg}(D)$ is greater than $2g-2$.
\end{thm}

In Section \ref{sec4}, we reduce the possibly most interesting case $D=K_C$ to Theorem \ref{main1}.

\subsection{Spectral curves and line bundles}
Recall the universal spectral curve (\ref{CC}) for $\mathrm{GL}_r$-Higgs bundles over the curve $C'$. We denote its restriction to the subspace $\CA^D(\pi) \subset \widetilde{\CA}^{D'}(C')$ by
\begin{equation}\label{eqn48}
\CC_{\pi} \rightarrow \CA^D(\pi).
\end{equation}
The morphism (\ref{eqn48}) is $G_\pi$-equivariant with respect to the natural Galois group $G_\pi$ actions on both the base $\CA^{D}(\pi)$ and the universal curve $\CC_\pi$.

We consider the largest Zariski open subset
\begin{equation*}
    \CA^{\heart}(\pi) \subset A^D(\pi)
\end{equation*}
such that:  \footnote{We note that our notation $\CA^\heart(\pi)$ has different meaning as the similar notation used in \cite{Ngo}.}
\begin{enumerate}
    \item[(a)] The action of $G_\pi$ is free on $\CA^{\heart}(\pi)$; 
    \item[(b)] The restriction of the spectral curves (\ref{eqn48})
    \begin{equation}\label{eqn49}
    \CC^\heart_\pi \rightarrow \CA^{\heart}(\pi)
    \end{equation}
    is smooth.
\end{enumerate}

By taking the $G_\pi$-quotients, the family (\ref{eqn49}) descends to a family of nonsingular curves
\begin{equation}\label{eqn50}
    \CC^\heart_\gamma \rightarrow \CA^\heart_\gamma 
\end{equation}
where $\CA^\heart_\gamma = \CA^\heart(\pi)/G_\pi$ is an open dense subset of $\CA^D_\gamma = \CA^D(\pi)/G_\pi$. We denote by
\begin{equation}\label{eqn51}
\CC^{\circ}_\gamma \rightarrow \CA^\heart_\gamma
\end{equation}
the restriction of the universal $\mathrm{SL}_n$-spectral curves over $\CA^D$ to $\CA^\heart_\gamma$. The families (\ref{eqn50}) and (\ref{eqn51}) are connected by the following lemma.

\begin{lem}\label{lem3.3}
There is a natural $\CA^\heart_\gamma$-morphism
 \begin{equation}\label{eqn52}
    \begin{tikzcd}[column sep=small]
    \CC^\heart_\gamma \arrow[dr] \arrow[rr,"u_\CC"] & & \CC^\circ_\gamma \arrow[dl] \\
       & \CA_\gamma^\heart  & 
\end{tikzcd}
\end{equation}
whose restriction to each closed fiber
\[
u_{a}: \CC^\heart_{\gamma,a} \rightarrow \CC^\circ_{\gamma,a}, \quad a\in \CA^\heart_\gamma 
\]
is a normalization of curves.
\end{lem}

\begin{proof}
We first recall the construction of \cite[Section 5.1]{HP} that, for a given degree $r$ spectral curve $C'_\alpha \rightarrow C'$ lying in $V(D')$, there is a natural birational morphism 
\[
C'_\alpha \to C_\alpha
\]
with $C_\alpha$ a degree $n$ spectral curve over $C$ lying in the total space $V(D)$. In fact, given $C'_\alpha \to C'$, let 
\[
g^*C'_\alpha \rightarrow C'
\]
be another degree $r$ spectral cover over $C'$ obtained as the pullback of $C'_\alpha \rightarrow C'$ along the Galois automorphism
\[
g: C' \xrightarrow{\simeq} C', \quad g\in G_\pi.
\]
The $G_\pi$-invariant curve 
\[
\widetilde{C}'_\alpha = \bigcup_{g\in G_\pi} g^* C'_\alpha
\]
is a degree $n(=mr)$ spectral cover over $C'$, which descends to a degree $n$ spectral cover $C_\alpha \to C$ via taking the $G_\pi$-quotient. Moreover, we see from the construction of $\Phi_\Gamma$ in \cite[Section 5.1]{HP} that the point $[C'_\alpha] \in \CA^\heart(\pi)$ maps to $[C_\alpha] \in \CA^\heart_\gamma$ via the natural quotient map $q^D_\CA: \CA^\heart(\pi) \to \CA^\heart_\gamma$. The composition
\[
C'_\alpha \hookrightarrow \widetilde{C}'_\alpha \rightarrow C_\alpha
\]
is birational, hence it is a normalization by the smoothness of $C'_\alpha$.

The construction above works for families of spectral curves over the Hitchin bases. Hence we obtain a commutative diagram
\begin{equation*}
\begin{tikzcd}
\CC^\heart_\pi \arrow[r] \arrow[d]
& \CC^\circ_\gamma \arrow[d] \\
\CA^\heart(\pi) \arrow[r, "q_\CA"]
& \CA^\heart_\gamma
\end{tikzcd}
\end{equation*}
where the left vertical morphism is $G_\pi$-equivariant, and the horizontal morphisms are $G_\pi$-quotient maps. The lemma follows from descending the left vertical arrow.
\end{proof}

For a closed point $a \in \CA^\heart_\gamma$, we denote the corresponding spectral curves over $C'$ and $C$ by $C'_a$ and $C_a$ respectively with the morphism
\[
u_a:  C'_a \rightarrow C_a
\]
given by Lemma \ref{lem3.3}. We consider the commutative diagram
\begin{equation}\label{eqn53}
  \begin{tikzcd}[column sep=small]
    C'_a \arrow[dr, "s'_a"] \arrow[rr,"u_a"] & & C_a \arrow[dl, "s_a"] \\
       & C.  &
    \end{tikzcd}
\end{equation}
Here $s_a: C_a \to C$ is the spectral cover over $C$, and $s'_a: C'_a \to C$ is the composition of the spectral cover $C'_a \to C'$ and the cyclic Galois cover $\pi:C' \to C$. Both $s'_a$ and $s_a$ are finite of degree $n$.

We also consider the following line bundles
\begin{equation*}
{\omega_{\pi,a}} = \mathrm{det}({s'_a}_* \CO_{C'_a}) \in \mathrm{Pic}(C), \quad \omega_a = \mathrm{det}({s_a}_* \CO_{C_a}) \in \mathrm{Pic}(C).
\end{equation*}
The line bundle $\omega_{\pi,a}$ is defined for every spectral curve over $\CA^D(\pi)$, which gives a family of line bundles over the affine space $\CA^D(\pi)$. Hence it is constant over $\CA^D(\pi)$ and does not depend on the choice of the spectral curve. Similarly, the line bundle $\omega_a$ is also independent of the spectral curve over $C$. So we may write
\[
\omega_\pi = \omega_{\pi,a}, \quad \omega = \omega_a.
\]
The following lemma is obtained via a direct calculation.

\begin{lem}\label{lem3.4}
We have 
\[
\mathrm{deg}(\omega_\pi) =n(1-r)\frac{\mathrm{deg}(D)}{2}, \quad  \mathrm{deg}(\omega) = n(1-n)\frac{\mathrm{deg}(D)}{2}.
\]
In particular, if $\mathrm{deg}(D)$ is even, both line bundles $\omega_\pi$ and $\omega$ have degrees divisible by $n$ 
\end{lem}

\begin{proof}
Since the second equality is a special case of the first one, we only prove the degree formula for $\mathrm{deg}(\omega_\pi)$.

By the discussion before Lemma \ref{lem3.4}, the line bundle $\omega_\pi$ does not depend on the choice of the spectral curve. Let 
\[
s'_a: C'_a \to C' \xrightarrow{\pi}  C
\]
be the spectral cover where $C'_a$ is a nonsingular curve lying in $V(D')$ of genus
\[
g(C'_a) = r(r-1) \frac{\mathrm{deg}(D')}{2}+r\left(g(C')-1\right)+1;
\]
see the formula of $d_{h_n}$ in \cite[Section 6.1]{dC_SL}. Then applying the Riemann-Roch formula to $\chi(C'_a, \CO_{C'_a}) = \chi(C, \omega_\pi)$, we obtain that 
\[
1-g(C'_a) = \mathrm{deg}(\omega_\pi)+(1-g)\]
which implies the lemma.
\end{proof}


The following lemma concerns pushing forward a line bundle $\CN \in \mathrm{Pic}(C_a)$ and its pullback 
\[
u_a^*\CN \in \mathrm{Pic}(C_a')
\]
to the curve $C_a'$.

\begin{lem}\label{lem3.5}
With the same notation as in (\ref{eqn53}), we have
\[
\mathrm{det}({s_a}_* \CN) = \mathrm{det}({s'_a}_*u_a^*\CN)\otimes \omega \otimes \omega_\pi^{-1}.
\]
\end{lem}

\begin{proof}
Recall the norm maps 
\[
\mathrm{Nm}: \mathrm{Pic}(C'_a) \rightarrow \mathrm{Pic}(C), \quad \mathrm{Nm}:\mathrm{Pic}(C_a) \rightarrow \mathrm{Pic}(C)
\]
from \cite[Section 3]{HP}. By \cite[Lemma 3.4]{HP}, we have
\[
\mathrm{Nm}(\CN) = \mathrm{Nm}(u_a^*\CN).
\]
Then \cite[Corollary 3.12]{HP} implies that 
\[
\mathrm{det}({s_a}_*\CN) \otimes \omega^{-1} = \mathrm{Nm}(\CN) = \mathrm{Nm}(u_a^*\CN) = \mathrm{det}({s'_a}_*u_a^*\CN)\otimes \omega_\pi^{-1}.\qedhere
\]
\end{proof}

For $n\geq 1$ and $L \in \mathrm{Pic}(C)$, we consider the regular parts
\[
\CM^{D,\mathrm{reg}}_{n,L} \subset \CM^D_{n,L}, \quad \CM^{D, \mathrm{reg}}_{r,L}(\pi) \subset \CM^{D}_{r,L}(\pi)
\]
which are open subvarieties parameterizing Higgs bundles obtained as the pushforward of \emph{line bundles} supported on the spectral curves. We define the line bundle
\begin{equation}\label{L'}
L' = L \otimes \omega \otimes \omega_\pi^{-1} \in \mathrm{Pic}(C). 
\end{equation}

The following is a corollary of Lemma \ref{lem3.5}.

\begin{cor}\label{cor3.6}
The pullback $u^*_\CC$ of (\ref{eqn52}) induces a $G_\pi$-equivariant morphism of the regular parts
\[
g_u: \CM^{D,\mathrm{reg}}_{n,L}\times_{\CA^D} \CA^\heart(\pi) \rightarrow \CM^{D, \mathrm{reg}}_{r,L'}(\pi)\times_{\CA^D(\pi)}\CA^\heart(\pi)
\]
where $L'$ is given by (\ref{L'}). The morphism $\CA^\heart(\pi) \rightarrow \CA^D$ used in the base change of the left-hand side is the composition
\[
\CA^\heart(\pi) \xrightarrow{q_\CA} \CA^\heart_\gamma \hookrightarrow \CA^D.
\]
\end{cor}

\begin{rmk}
Since $\mathrm{gcd}(n, \mathrm{deg}(L)) = 1$ and
\[
\mathrm{deg}(L') = \mathrm{deg}(L) + \mathrm{deg}(\omega) - \mathrm{deg}(\omega_\pi),
\]
Lemma \ref{lem3.4} implies that 
\begin{equation*}
\mathrm{gcd}(r, \mathrm{deg}(L')) = \mathrm{gcd}(n, \mathrm{deg}(L')) = 1
\end{equation*}
when $\mathrm{deg}(D)$ is even.
\end{rmk}

Finally, we note that both varieties $\CM^{D,\mathrm{reg}}_{n,L}$ and $\CM^{D, \mathrm{reg}}_{r,L'}(\pi)$ admit $\Gamma$-actions induced by the tensor product with $n$-torsion line bundles. The $\Gamma$-actions are clearly fiberwise with respect to the Hitchin maps 
\begin{equation}\label{Hitchins}
\CM^{D,\mathrm{reg}}_{n,L} \to \CA^D, \quad \CM^{D, \mathrm{reg}}_{r,L'}(\pi) \to \CA^D(\pi).
\end{equation}
Since the pullback $u_a^*: \mathrm{Pic}(C_a) \to \mathrm{Pic}(C_a')$ induced by (\ref{eqn53}) is compatible with the $\Gamma$-action on both sides, we conclude the following lemma.

\begin{lem}\label{lem3.8}
The morphism $g_u$ given in Corollary \ref{cor3.6} is $\Gamma$-equivariant.
\end{lem}

\subsection{Proof of Theorem \ref{main1} (step 1): Cohomological correspondences}\label{Section3.3}
Throughout Section \ref{Section3.3} to Section \ref{Section3.5}, we assume that $D$ is an effective divisor with $\mathrm{deg}(D)$ even and greater than $2g-2$. Our first step in the proof of Theorem \ref{main1} is to construct a correspondence between the direct image complexes associated with the two Hitchin maps attached to (\ref{Hitchins}). Then we show that this correspondence induces an isomorphism for the $\kappa$-parts following Ng\^{o} and Yun (\cite[Appendix A]{Yun3}).

We consider the graph of $g_u$ in Corollary \ref{cor3.6} which gives a subvariety 
\[
\mathrm{Graph}(g_u) \subset \left(\CM^{D, \mathrm{reg}}_{r,L'}(\pi)\times_{\CA^D(\pi)}\CA^\heart(\pi)\right)\times _{\CA^\heart(\pi)} \left(\CM^{D,\mathrm{reg}}_{n,L}\times_{\CA^D} \CA^\heart(\pi)\right).
\]
Taking its Zariski closure, we obtain a closed subvariety
\[
\Sigma = \overline{\mathrm{Graph}(g_u)} \subset \CM^D_{r,L'}(\pi) \times_{\CA^D(\pi)} \left( \CM^D_{n,L} \times_{\CA^D} \CA^D(\pi) \right)
\]
which fits into the commutative diagram
\begin{equation}\label{corr}
\begin{tikzcd}[column sep=tiny]
& \Sigma \ar[dl] \ar[dr] 
&
&[1.5em] \\
\CM^D_{n,L}\times_{\CA^D}\CA^D(\pi)  \ar[dr, "h^D"]
&
& \CM^D_{r,L'}(\pi) \ar[dl, "h^D_{\pi,L'}"] \\
& \CA^D(\pi).
&
&
\end{tikzcd}
\end{equation}
Here we use $h^D_{\pi,L'}$ to denote the Hitchin fibration
\[
h^D_{\pi,L'} := h^D_{\pi}: \CM^D_{r,L'}(\pi) \to \CA^D(\pi)
\]
indicating its dependence on the line bundle $L'$. All the morphisms in (\ref{corr}) are proper.

By Corollary \ref{cor3.6} and Lemma \ref{lem3.8}, the morphism $g_u$ is equivariant under the actions of $G_\pi$ and $\Gamma$. Hence, as the Zariski closure of the graph of $g_u$, the subvariety $\Sigma$ is invariant under the natural actions of $G_\pi$ and $\Gamma$ on the ambient space $\CM^D_{r,L'}(\pi) \times_{\CA^D(\pi)} \left( \CM^D_{n,L} \times_{\CA^D} \CA^D(\pi) \right)$. Since the projections
\[
\begin{split}
\CM^D_{r,L'}(\pi) \times_{\CA^D(\pi)} \left( \CM^D_{n,L} \times_{\CA^D} \CA^D(\pi) \right) & \to \CM^D_{r,L'}, \\
\CM^D_{r,L'}(\pi) \times_{\CA^D(\pi)} \left( \CM^D_{n,L} \times_{\CA^D} \CA^D(\pi) \right) & \to \CM^D_{n,L} \times_{\CA^D} \CA^D(\pi)
\end{split}
\]
to both factors are $G_\pi$- and $\Gamma$-equivariant, we obtain that the projections from the invariant subvariety $\Sigma$ to both factors are also $G_\pi$- and $\Gamma$-equivariant.

\begin{thm}\label{thm3.9}
The correspondence (\ref{corr}) induces a morphism
\begin{equation}\label{Sigma}
[\Sigma]_\#: q_\CA^* \left( \mathrm{Rh^D_*} \BBC\right)  \rightarrow \mathrm{Rh_{\pi,L'}^D}_*\BBC [-2d^D_\gamma]
\end{equation}
which is equivariant under the natural actions of $G_\pi$ and $\Gamma$. Assume that the element $\gamma \in \Gamma$ inducing the Galois cover $\pi: C' \to C$ matches with $\kappa$ via (\ref{Weil}). Then the $G_\pi$-equivariant morphism for the $\kappa$-parts
\begin{equation}\label{eqn57}
[\Sigma]_{\#,\kappa}: q_\CA^* \left( \mathrm{Rh^D_*} \BBC\right)_\kappa  \to \left(\mathrm{Rh_{\pi,L'}^D}_*\BBC\right)_\kappa [-2d^D_\gamma]
\end{equation}
induced by (\ref{Sigma}) is an isomorphism after restricting to $\CA^D(\pi)^*$:
\begin{equation}\label{extra4}
[\Sigma]_{\#,\kappa}\Big{|}_{\CA^D(\pi)^*}: q_\CA^* \left( \mathrm{Rh^D_*} \BBC\right)_\kappa\Big{|}_{\CA^D(\pi)^*}  \xrightarrow{\simeq} \left(\mathrm{Rh_{\pi,L'}^D}_*\BBC\right)_\kappa\Big{|}_{\CA^D(\pi)^*} [-2d^D_\gamma].
\end{equation}
\end{thm}

\begin{proof}
The first part follows from the general theory of cohomological correspondences. We refer to \cite[Appendix A]{Yun12} as a reference; see also the paragraph before \cite[Proposition 3.3.1]{Yun3}. In particular, since the variety $\Sigma$ is invariant under the $G_\pi$- and the $\Gamma$-actions, and both projections from $\Sigma$ to $\CM^D_{n,L}\times_{\CA^D}\CA^D(\pi)$ and $\CM^D_{r,L'}(\pi)$ are $G_\pi$- and $\Gamma$-equivariant, we conclude that the cohomological correspondence (\ref{Sigma}) is also $G_\pi$- and $\Gamma$-equivariant. 

Now we restrict $[\Sigma]_{\#,\kappa}$ to the open subset $\CA^D(\pi)^*$ and show that it is an isomorphism. It suffices to prove that the restriction of $[\Sigma]_{\#,\kappa}$ induces an isomorphism on every perverse cohomology, \emph{i.e.},
\begin{equation}\label{eqn58}
    ^p\CH^i([\Sigma]_{\#,\kappa})\Big{|}_{\CA^D(\pi)^*}: q_\CA^* {^p\CH^i}\left( \mathrm{Rh^D_*} \BBC\right)_\kappa \Big{|}_{\CA^D(\pi)^*} \xrightarrow{\simeq} {^p\CH^{i-2d^D_\gamma}}\left(\mathrm{Rh_{\pi,L'}^D}_*\BBC\right)_\kappa \Big{|}_{\CA^D(\pi)^*}.
\end{equation}
Here we used that $q_\CA$ is \'etale restricting to $\CA^D(\pi)^*$. By Corollary \ref{cor2.2}, the left-hand side of (\ref{eqn58}) has $\CA^D(\pi)^*$ as the only support. Theorem \ref{SuppThm2} (b) and Proposition \ref{prop2.8} yield that the right-hand side of (\ref{eqn58}) also has $\CA^D(\pi)^*$ as the only support. Therefore, both sides of (\ref{eqn58}) are intermediate extensions of certain local systems defined on an open subset of $\CA^D(\pi)$. As a consequence, in order to prove (\ref{eqn58}), we only need to show that
\[
 ^p\CH^i([\Sigma]_{\#,\kappa})\Big{|}_U: q_\CA^* {^p\CH^i}\left( \mathrm{Rh^D_*} \BBC\right)_\kappa \Big{|}_U \xrightarrow{\simeq} {^p\CH^{i-2d^D_\gamma}}\left(\mathrm{Rh_{\pi,L'}^D}_*\BBC\right)_\kappa\Big{|}_U
\]
with $U \subset \CA^D(\pi)^*$ a Zariski open subset. This reduces the proof of (\ref{eqn58}) to showing that, for a general point $a \in \CA^D(\pi)^*$, the correspondence between the Hitchin fibers 
\begin{equation}\label{eqn59}
[\Sigma_a]: H^i\left((\CM^D_{n,L})_a, \BC\right) \rightarrow H^{i-2d^D_\gamma}\left(\CM^D_{r,L'}(\pi)_a, \BC\right)
\end{equation}
induced by the fundamental class of $\Sigma_a$ is an isomorphism between the $\kappa$-parts. Here $\Sigma_a$ is the restriction of $\Sigma$ over $a$. Let $a$ be a general point lying in the open subset $\CA^\heart(\pi)\subset \CA^D(\pi)$. The pullback of the diagram (\ref{eqn52}) along $\{a\} \rightarrow \CA^\heart_\gamma$ is a normalization 
\[
u_a: C'_a \to C_a
\]
of curves where $C_a$ has at worst nodal singularities. In this case, the description of the correspondence (\ref{eqn59}) is concretely given in \cite[Appendix A]{Yun3} and the isomorphism of the $\kappa$-parts follows from a direct calculation \cite[Lemma 3.4.1]{Yun3}.\footnote{For the correspondence (\ref{eqn59}), we only concern Hitchin fibers over a general closed point $a\in \CA^D(\pi)$. Since tensoring with degree 1 line bundles on the spectral curves over $a \in \CA^D(\pi)$ identifies the Hitchin fibers in the case of degree 0 and the Hitchin fibers in the case of coprime rank and degree, the result of \cite[Lemma 3.4.1]{Yun3} concerning degree 0 Higgs bundles applies here.} This completes the proof.
\end{proof}


\begin{cor}\label{cor3.11}
With the same assumption as in Theorem \ref{thm3.9}, the $G_\pi$-equivariant isomorphism (\ref{extra4}) induces an isomorphism:
\[
[\Sigma]^{G_\pi}_{\#,\kappa}: \left(\mathrm{Rh^D_*} \BBC\right)_\kappa \xrightarrow{\simeq} {i^D_\gamma}_*\left(\mathrm{Rh^D_{\gamma,L'}}_* \BBC\right)_\kappa [-2d^D_\gamma] \in D^b_c(\CA^D).
\]
\end{cor}

\begin{proof}
We pushforward the morphism (\ref{eqn57}) along the $G_\pi$-quotient map $q_\CA: \CA^D(\pi) \to \CA^D_\gamma$ and take the $G_\pi$-invariant parts on both sides. This gives a morphism
\[
[\Sigma]^{G_\pi}_{\#,\kappa}: \left(\mathrm{Rh^D_*} \BBC\right)_\kappa \to {i^D_\gamma}_*\left(\mathrm{Rh^D_{\gamma,L'}}_* \BBC\right)_\kappa [-2d^D_\gamma]
\]
where both sides are semi-simple with $\CA^D_\gamma$ the only support by Corollary \ref{cor2.2}, Theorem \ref{SuppThm2} (b), and Proposition \ref{prop2.8}. Hence similar to the proof of Theorem \ref{thm3.9}, it suffices to check that it is an isomorphism restricting to ${\CA^D_\gamma}^*$ which is equivalent to (\ref{extra4}).
\end{proof}

\begin{rmk}
We cannot conclude that (\ref{eqn57}) is an isomorphism over the total space $\CA^D(\pi)$ from the isomorphism (\ref{extra4}) due to the following reason. For a non-smooth finite morphism $f: X \to Y$ and a semi-simple perverse sheaf $\CK$ on $Y$, the pullback $f^*\CK$ may fail to be semi-simple. A typical example is the case where
\[
f: \BA^1 \to \BA^1, \quad z \mapsto z^2
\]
and $\CK$ is the intermediate extension of a 2-torsion rank 1 local system on $\BC^* \subset \BA^1$. The object $f^*\CK$ is not determined by its restriction to the open subset $\BC^*$. In particular, although the restriction of the natural morphism $f^*\CK\to \underline{\BC}$ to $\BC^*$ is an isomorphism, the morphism $f^*\CK \to \underline{\BC}$ itself fails to be an isomorphism on $\BA^1$. The issue here is caused by the fact that $f$ is not \'etale at $0\in \BA^1$.
\end{rmk}

\subsection{Proof of Theorem \ref{main1} (Step 2): changing from $L'$ to $L$}\label{Section3.4}

By Lemma \ref{lem3.4}, we have
\begin{equation}\label{eqnL}
\mathrm{deg}(L') - \mathrm{deg}(L) = 0 ~~~ \mathrm{mod}~~~ n.
\end{equation}
Hence there exists a line bundle
\begin{equation}\label{eqn61}
\CN_0 \in \mathrm{Pic}^k(C), \quad k = \frac{\mathrm{deg}(L') - \mathrm{deg}(L)}{n} \in \BZ
\end{equation}
such that $L' = L \otimes \CN_0^{\otimes n}$. The line bundle $\CN_0$ induces an $\CA^D(\pi)$-isomorphism between the relative Hitchin moduli spaces 
\begin{equation}\label{eqn62}
  \begin{tikzcd}[column sep=small]
    \CM^D_{r,L}(\pi) \arrow[dr, "h^D_{\pi,L}"] \arrow[rr,"\phi_{\CN_0}"] & & \CM^D_{r,L'}(\pi) \arrow[dl, "h^D_{\pi,L'}"] \\
       & \CA^D(\pi)  & 
    \end{tikzcd}
\end{equation}
via tensor product $\phi_{\CN_0}(\CE, \theta) = (\CE \otimes \CN_0, \theta)$. It is clear that $\phi_{\CN_0}$ is $G_\pi$- and $\Gamma$-equivariant.

\begin{prop}\label{prop3.10}
There is a $G_\pi$-equivariant isomorphism 
\begin{equation}\label{extra6}
\left(\mathrm{Rh_{\pi,L'}^D}_*\BBC\right)_\kappa = \left(\mathrm{Rh_{\pi,L}^D}_*\BBC\right)_\kappa \in D^b_c(\CA^D(\pi))
\end{equation}
induced by (\ref{eqn62}). Up to scaling, it is independent of the choice of line bundle (\ref{eqn61}). In particular, (\ref{extra6}) induces
\[
\left(\mathrm{Rh_{\gamma,L'}^D}_*\BBC\right)_\kappa = \left(\mathrm{Rh_{\gamma,L}^D}_*\BBC\right)_\kappa \in D^b_c(\CA_\gamma^D).
\]
\end{prop}

\begin{proof}
The pullback along $\phi_{\CN_0}$ in the diagram (\ref{eqn62}) induces a $G_\pi$- and $\Gamma$- equivariant isomorphism
\begin{equation}\label{eqn63}
\mathrm{Rh_{\pi,L'}^D}_*\BBC = \mathrm{Rh_{\pi,L}^D}_*\BBC 
\end{equation}
which gives our desired isomorphism. For another choice $\CN'_0$ of the line bundle (\ref{eqn61}), we have 
\[
\CN_0^{-1}\otimes \CN'_0 \in \Gamma. 
\]
Hence the difference of the isomorphisms $\phi^*_{\CN_0}$ and $\phi^*_{\CN'_0}$ is induced by an automorphism of $\CM^D_{r,L}$ given by an element $g\in \Gamma$. In particular, our choice of isomorphism is affected only by scaling.


The last claim follows from Lemma \ref{lemma1.6}.
\end{proof}

\begin{rmk}\label{rmk3.12}
In view of Lemma \ref{lem3.4}, the equation (\ref{eqnL}) and the existence of the line bundle (\ref{eqn61}) rely heavily on the assumption that $\mathrm{deg}(D)$ is even.
\end{rmk}

\subsection{Completing the proof of Theorem \ref{main1}}\label{Section3.5}
Theorem \ref{main1} follows from Corollary \ref{cor3.11} and Proposition \ref{prop3.10}.

More precisely, we constructed an isomorphism 
\[
c^D_\kappa: \left(\mathrm{Rh^D_*} \BBC\right)_\kappa \xrightarrow{\simeq} {i^D_\gamma}_*\left(\mathrm{Rh^D_\gamma}_* \BBC\right)_\kappa [-2d^D_\gamma] \in D_c^b(\CA^D)
\]
which is canonical up to scaling as the composition  
\[
  \left( \mathrm{Rh^D_*} \BBC\right)_\kappa  \xrightarrow[\simeq]{~~[\Sigma]^{G_\pi}_{\#,\kappa}~~} {i^D_\gamma}_*\left(\mathrm{Rh_{\gamma,L'}^D}_*\BBC\right)_\kappa [-2d^D_\gamma] \xrightarrow[\simeq]{\mathrm{Prop}.~~3.13.} {i^D_\gamma}_*\left(\mathrm{Rh_{\gamma,L}^D}_*\BBC\right)_\kappa [-2d^D_\gamma].
\]
Here the first isomorphism $[\Sigma]^{G_\pi}_{\#,\kappa}$ is given by Corollary \ref{cor3.11}, and the second isomorphism is given by Proposition \ref{prop3.10}. This completes the proof of Theorem \ref{main1}. \qed

\section{Vanishing cycles and Hitchin moduli spaces}\label{sec4}

\subsection{Overview}
 In this section, we follow the same notation as in Section \ref{Sec3}. We complete the proof of Theorem \ref{thm3.2} (and therefore Theorem \ref{main0'} as explained after Theorem \ref{thm3.2}) by constructing an operator
 \[
c^D_\kappa: \left(\mathrm{Rh^D_*} \BBC\right)_\kappa \xrightarrow{\simeq} {i^D_\gamma}_*\left(\mathrm{Rh^D_\gamma}_* \BBC\right)_\kappa [-2d^D_\gamma] \in D_c^b(\CA^D)
\]
which is constructed up to scaling for $D= K_C$. Our main tool is Theorem \ref{thm4.4} and its variant Corollary \ref{cor_correction}, where we apply a vanishing cycle functor to connect the moduli of $D$-Higgs bundles to the moduli of $(D+p+q)$-Higgs bundles with $p,q \in C$ two distinct closed points. This reduces the hardest case $D=K_C$ to the easier cases already treated in Theorem \ref{main1}.

In this section, it is convenient to work with 
the moduli stacks of stable $\mathrm{SL}_n$-Higgs bundles and relative stable Higgs bundles associated with $\pi: C' \to C$.  As these are nonsingular Deligne-Mumford stacks, gerbes over the coarse moduli spaces, this has no effect on the  direct image complexes (\ref{main00}).  
 Therefore, throughout this section, we still use the notation $\CM^D_{n,L}$ and $\CM^D_{r,L}(\pi)$ to denote the corresponding moduli stacks for stable Higgs bundles and relative stable Higgs bundles.

\subsection{Restrictions of Higgs bundles to a point}
Let $p$ be an abstract reduced point $\mathrm{Spec}(\BC)$. Any rank $n$ vector bundle on $p$ is an $n$-dimensional vector space. Hence the category of $\mathrm{SL}_n$-Higgs bundles on $p$ can be thought of as the category of matrices in $\mathfrak{sl}_n$ up to $\mathrm{SL}_n$-conjugation, whose moduli stack is given by the quotient
\[
\CM_{n,p} = [\mathfrak{sl}_n/\mathrm{SL}_n].
\]
Here $\mathrm{SL}_n$ acts on $\mathfrak{sl}_n$ via conjugation. The Hitchin fibration associated with $\CM_{n,p}$ is
\[
h_p: \CM_{n,p} \to \CA_p
\]
where $\CA_p = \mathfrak{sl}_n\sslash\mathrm{SL}_n$ is the affine $\mathrm{GIT}$-quotient parameterizing all  characteristic polynomials
\[
(a_2, a_3, \dots, a_n) \in \CA_p, \quad a_i = \mathrm{trace}(\wedge^i \theta_p)
\]
of the traceless endomorphism $\theta_p \in \mathrm{End}(\BA^n)_0$ associated with a matrix in $\CM_{n,p}$. We refer to \cite[Section 2.2]{Ngo} for more details concerning the stack $\CM_{n,p}$ and the morphism $h_p$.

Now we consider $p$ as a closed point on the curve $C$. We fix a trivialization 
\begin{equation}\label{trivialization}
\CO_C(D)_p \xrightarrow{\simeq} \BC.
\end{equation}
Then the restriction map with respect to the closed embedding
\[
i_p: \{p\} \hookrightarrow C
\]
induces the following commutative diagram
\begin{equation}\label{eqn66}
\begin{tikzcd}
\CM_{n,L}^D \arrow[r, "\mathrm{ev}_p"] \arrow[d, "h^D"]
& \CM_{n,p} \arrow[d, "h_p"] \\
\CA^D \arrow[r]
& \CA_p
\end{tikzcd}
\end{equation}
where the trivialization (\ref{trivialization}) induces an evaluation (at $p$) map $\mathrm{ev}_p$:
\[
\mathrm{ev}_p(\CE, \theta) = i_p^*(\CE, \theta) \in \CM_{n,p}.
\]

To generalize the diagram (\ref{eqn66}) for the relative moduli space $\CM^D_{r,L}(\pi)$ with $n=mr$, we consider the Lie group
\[
H_\pi = \{(g_1, g_2, \dots, g_m) \in \mathrm{GL}_r^{\times m}: ~~ \prod_{i} \mathrm{det}(g_i) = 1\} \subset \mathrm{GL}_r^{\times m},
\]
which is naturally a subgroup of $\mathrm{SL}_n$ with Lie algebra 
\[
\mathfrak{h}_\pi =  \{(g_1, g_2, \dots, g_m) \in \mathfrak{gl}_r^{\times m}: ~~ \sum_{i} \mathrm{trace}(g_i) = 0\}.
\]
The quotient stack
\[
\CM_{r,p}(\pi) = [\mathfrak{h}_\pi/H_\pi]
\]
is the moduli of $\mathrm{SL}_n$-Higgs bundles on the point $p$ obtained as the pushforward of rank $r$-Higgs bundles on $m$ distinct reduced points $\sqcup_{i=1}^m p_i$ along the projection
\[
\pi_p: \sqcup_{i=1}^m p_i \to p.
\]
Similar to (\ref{eqn66}), we have the following diagram given by the restriction to $p \in C$:
\begin{equation}\label{eqn67}
\begin{tikzcd}
\CM_{r,L}^D(\pi) \arrow[r, "\mathrm{ev}_p"] \arrow[d, "h^D_\pi"]
& \CM_{r,p}(\pi) \arrow[d, "h_{\pi,p}"] \\
\CA^D \arrow[r]
& \CA_p(\pi).
\end{tikzcd}
\end{equation}
Here for a Higgs bundle $(\CE, \theta) \in \CM^D_{r,L}$ on $C'$, again (\ref{trivialization}) induces an evaluation map:
\[
\mathrm{ev}_p(\CE, \theta) = i_p^*\left(\pi_* \CE, \pi_* \theta\right) \in \CM_{r,p}(\pi), 
\]
and the ``Hitchin map over a point" $h_p$ is the natural projection
\[
h_{\pi,p}: [\mathfrak{h}_\pi/H_\pi] \rightarrow \CA_p(\pi):= \mathfrak{h}_\pi\sslash H_\pi.
\]
The diagram (\ref{eqn67}) recovers (\ref{eqn66}) when $C' = C$ and $\pi: C' \to C$ is the identity.

The following proposition proves the smoothness of the evaluation map; for convenience in applications, we state the result in a more general form.

\begin{prop}\label{prop4.1}
Let $p_1, p_2,\cdots, p_s$ are $s$ distinct points on $C$. Assume that $D$ is a divisor on $C$ satisfying
\begin{enumerate}
    \item[(a)] $D-\Sigma_{i=1}^s p_i = K_C$, or
    \item[(b)] $D-\Sigma_{i=1}^s p_i$ is effective and $\mathrm{deg}(D- \Sigma_{i=1}^s p_i)> 2g-2$.
\end{enumerate}
Then the evaluation map 
\[ 
\mathrm{ev}_{p_1,\cdots, p_s}:= \left(\mathrm{ev}_{p_1}, \cdots, \mathrm{ev}_{p_s}\right): 
\CM^D_{r,L}(\pi) \to \CM_{r,p_1}(\pi) \times \cdots \times \CM_{r,p_s}(\pi)
\]
induced by (\ref{eqn67}) is smooth.
\end{prop}

\begin{proof}
For notational convenience, we prove the case when $s=1$ and $p=p_1 = p_s$; the general case works identically.

We first review the deformation theory of $\CM^D_{r,L}(\pi)$ following \cite{BR} and \cite[Section 4.14]{Ngo}. The deformation theory of a rank $r$ Higgs bundle $(\CE, \theta)$ on $C'$ is governed by the tangent complex
\begin{equation}\label{eeq}
\left[\CE{nd}(\CE) \xrightarrow{\mathrm{ad}(\theta)} \CE{nd}(\CE)\otimes \CO_{C'}(D')\right]
\end{equation}
lying in degrees -1 and 0. Since now we only consider rank $r$ Higgs bundles $(\CE, \theta)$ on $C'$ with the fixed determinant and trace after pushing forward to $C$, to govern the deformation theory of $\CM^D_{r,L}(\pi)$ we need to remove the traces of $\CE{nd}(\CE)$ after pushing forward the complex (\ref{eeq}) to $C$. Hence the deformation theory of $\CM^D_{r,L}(\pi)$ is governed by 
\begin{equation}\label{C^bullet}
C^\bullet(\CE,\theta,D) = \left[(\pi_*\CE{nd}(\CE))_0 \xrightarrow{\pi_*\mathrm{ad}(\theta)} (\pi_*\CE{nd}(\CE))_0\otimes \CO_C(D)\right], \quad (\CE,\theta)\in \CM^D_{r,L}(\pi),
\end{equation}
where $(\pi_*\CE{nd}(\CE))_0$ denotes the kernel with respect to the trace
\[
\mathrm{tr}_C: \pi_* \CE{nd}(\CE) \xrightarrow{\pi_*\mathrm{tr}_{C'}} \pi_*\CO_{C'} \to \CO_C
\]
on $C$. The complex (\ref{C^bullet}) is the tangent complex of $\CM^D_{r,L}(\pi)$. The automorphism space, the tangent space, and the obstruction space are thus given by the cohomology groups
\[
H^0(C, C^\bullet(\CE,\theta,D)), \quad H^1(C, C^\bullet(\CE,\theta,D)), \quad \textup{and}~~~~H^2(C, C^\bullet(\CE,\theta,D))
\]
respectively. Since the evaluation map $\mathrm{ev}_p$ is induced by the restriction to the point $p$ via $i_p: \{p\}\hookrightarrow C$, the tangent map of $\mathrm{ev}_p$ is
\begin{equation}\label{tangent}
\mathrm{Tan}_{\mathrm{ev}_p}: H^1(C, C^\bullet(\CE,\theta,D)) \rightarrow H^1(p, i_p^*C^\bullet(\CE,\theta,D))
\end{equation}
induced by the restriction morphism between the tangent complexes. Here $i_p^*C^\bullet(\CE,\theta,D) = [\mathfrak{h}_\pi \xrightarrow{\mathrm{ad}} \mathfrak{h}_{\pi}]$ recovers the tangent complex of $\CM_{r,p}(\pi)$; see \cite[Appendix 8.2]{Supp}. To prove the smoothness of $\mathrm{ev}_p$, we show in the following that (\ref{tangent}) is surjective.

The restriction map between the tangent complexes
\[
C^\bullet(\CE,\theta,D)) \to {i_p}_*i_p^*C^\bullet(\CE,\theta,D)) 
\]
fits into the exact triangle
\begin{equation}\label{idk2}
E^\bullet \to C^\bullet(\CE,\theta,D) \to {i_p}_*i_p^*C^\bullet(\CE,\theta,D) \xrightarrow {+1}
\end{equation}
where $E^\bullet$ is given by
\[
E^\bullet =C^\bullet(\CE,\theta,D)\otimes \CO_C(-p) = \left[(\pi_*\CE{nd}(\CE))_0\otimes \CO_C(-p) \xrightarrow{\pi_*\mathrm{ad}(\theta)} (\pi_*\CE{nd}(\CE))_0\otimes \CO_C(D-p)\right].
\]
The long exact sequence associated with (\ref{idk2}) contains
\[
 H^1(C, C^\bullet(\CE,\theta,D)) \xrightarrow{\mathrm{Tan}_{\mathrm{ev}_p}} H^1(C,{i_p}_* i_p^*C^\bullet(\CE,\theta,D)) \rightarrow H^2(C, E^\bullet).
\]
Hence, in view of the Serre duality, it suffices to show the vanishing of 
\begin{equation}\label{SD}
H^2(C, E^\bullet)^\vee = H^0(C, (E^\bullet)^\vee \otimes \Omega_C)
\end{equation}
where 
\[
(E^\bullet)^\vee \otimes \Omega_C = \left[(\pi_*\CE{nd}(\CE))_0\otimes \CO_C(p-D+K_C) \xrightarrow{-\mathrm{ad}(\pi_*\theta)} (\pi_*\CE{nd}(\CE)_0\otimes \CO_C(p+K_C)\right].
\]

To calculate (\ref{SD}), we consider the hyper-cohomology group  $H^0(C, (\widetilde{E}^\bullet)^\vee \otimes \Omega_C)$ with 
\[
\widetilde{E}^\bullet = \left[(\pi_*\CE{nd}(\CE))\otimes \CO_C(-p) \xrightarrow{\pi_*\mathrm{ad}(\theta)} (\pi_*\CE{nd}(\CE))\otimes \CO_C(D-p)\right].
\]
It is clear that $H^0(C, (\widetilde{E}^\bullet)^\vee \otimes \Omega_C)$ contains (\ref{SD}) as a direct sum component, and the complement is contributed by the trace parts on $C$. The hyper-cohomology group  $H^0(C, (\widetilde{E}^\bullet)^\vee \otimes \Omega_C)$ can be written as
\begin{equation}\label{idk}
H^0\left(C', \left[\CE{nd}(\CE)\otimes \pi^*\CO_C(p-D+K_C) \xrightarrow{-\mathrm{ad}(\theta)} \CE{nd}(\CE)\otimes \pi^*\CO_C(p+K_C)\right] \right)
\end{equation}
via the projection formula associated with $\pi: C' \to C$. By \cite[Theorem 5.1]{GK} (see also the calculation of \cite[Lemma 7.3]{GWZ}), (\ref{idk}) can be interpreted as the group of homomorphisms of stable Higgs bundles
\begin{equation}\label{Hom}
\mathrm{Hom}_{C'}\left( (\CE, \theta), (\CE \otimes \pi^* \CO_C(K_C-D+p), \theta) \right).
\end{equation}

In the case (a), these two stable Higgs bundles in (\ref{Hom}) coincide, so the $\mathrm{Hom}$ space (\ref{Hom}) is 1-dimensional given by the identity map. Hence we have the vanishing of (\ref{SD}) by removing the $1$-dimensional trace parts on $C$ from (\ref{Hom}). In the case (b), we have
\[
\mathrm{deg}(\CE) > \mathrm{deg}\left(\CE \otimes \pi^*\CO_C(K_C-D+p)\right).
\]
Hence (\ref{Hom}) vanishes due to the stability condition, which further implies the vanishing of (\ref{SD}).
\end{proof}

Assume that the divisor $D$ and the distinct points $p_1, \cdots, p_s$ satisfy (a) and (b) of Proposition \ref{prop4.1}. The moduli of stable $(D- \sum_i p)$-Higgs bundles admits a natural closed embedding into the moduli of $D$-Higgs bundles. More precisely, let $(\CE, \theta)$ be a stable $(D-\sum_i p_i)$-Higgs bundle on $C$, then we may view it naturally as a $D$-Higgs bundle $(\CE, \theta')$ by setting the new Higgs field as the composition
\[
\theta': \CE \xrightarrow{\theta} \CE\otimes \CO_C(D-\sum_i p_i) \rightarrow \CE \otimes\CO_C(D)
\]
where the second map is induced by $\CO_C(-\sum_i p_i) \hookrightarrow \CO_C$. By definition, the (slope-)stability conditions of $(\CE, \theta)$ and $(\CE, \theta')$ coincide. Hence we obtain a closed embedding 
\begin{equation}\label{eqn75}
\CM^{D-\sum_i p_i}_{n,L} \hookrightarrow \CM^D_{n,L}, \quad (\CE,\theta) \mapsto (\CE, \theta').
\end{equation}
Similarly, we also have the relative version with respect to $\pi: C' \to C$:
\begin{equation}\label{eqn76}
\CM^{D-\sum_i p_i}_{r,L}(\pi) \hookrightarrow \CM^D_{r,L}(\pi), \quad (\CE,\theta)\mapsto (\CE, \theta')
\end{equation}
where $\theta': \CE \to \CE\otimes \pi^*\CO_{C}(D)$ is given by the composition
\[
\theta': \CE\xrightarrow{\theta} \CE \otimes \pi^*\CO_C(D-\sum_i p_i) \to \CE \otimes \pi^*\CO_C(D).
\]
As before, the relative case (\ref{eqn76}) recovers (\ref{eqn75}) by setting $\pi = \mathrm{id}$. 

A Higgs bundle in $\CM^D_{r,L}(\pi)$ which sits inside $\CM^{D-\sum_i p_i}_{r,L}(\pi)$ can be characterized by the vanishing of the restricted Higgs field over the points $p_1, \cdots, p_s$. Therefore it is natural to use the evaluation map $\mathrm{ev}_{p_1,\cdots, p_s}$ to describe (\ref{eqn75}) and (\ref{eqn76}) as in the following lemma. 

\begin{lem}\label{lem4.2}
Assume the divisor $D$ and the distinct points $p_1, \cdots, p_s$ satisfy (a) and (b) of Proposition \ref{prop4.1}. We denote by 
\[
0_H = [0/H_\pi]\hookrightarrow \CM_{r,p}(\pi) = [\mathfrak{h}_\pi/H_\pi]
\]
the closed substack corresponding to the $0$ matrix.
Then the closed embedding (\ref{eqn76}) is realized as a closed fiber of $\mathrm{ev}_{p_1,\cdots,p_s}$ over 
\[
O_H:= 0_H \times \cdots \times 0_H \subset \prod_{i=1}^s[\mathfrak{h}_\pi/H_\pi] = \prod_{i=1}^s \CM_{r,p_i}(\pi).
\]
\end{lem}

\subsection{Functions and critical loci}
We consider the quotient map
\[
\mathfrak{sl}_n \rightarrow \mathfrak{sl}_n\sslash \mathrm{SL}_n
\]
sending a matrix to the coefficients
\[
(a_2, a_3, \dots, a_n)\in \mathfrak{sl}_n\sslash \mathrm{SL}_n = \mathrm{Spec}\left(\BC[\mathfrak{sl}_n]^{\mathrm{SL}_n}\right), \quad \mathrm{deg}(a_i) = i
\]
of its  characteristic polynomial. The term $a_i$ defines a degree $i$ polynomial function on the Lie algebra $\mathfrak{sl}_n$. We define the quadratic function on the Lie algebra $\mathfrak{sl}_n$ given by $a_2$ as
\begin{equation}\label{eq77}
\mu = a_2: \mathfrak{sl}_n \to  \BA^1,
\end{equation}
which induces a function $\mu_\pi: \mathfrak{h}_\pi \to \BA^1$ via the composition
\begin{equation}\label{eq78}
\mu_\pi : \mathfrak{h}_\pi \hookrightarrow \mathfrak{sl}_n \xrightarrow{\mu} \BA^1.
\end{equation}
By definition, the functions $\mu$ and $\mu_\pi$ are invariant under the conjugation actions by the Lie groups $\mathrm{SL}_n$ and $H_\pi$ respectively. 

Since any matrix $g \in \mathfrak{h}_\pi \subset \mathfrak{sl}_n$ satisfies $\mathrm{trace}(g) = 0$, the quadratic function $\mu_\pi$ (up to scaling) can be written as
\begin{equation}\label{trace}
    g \mapsto \mathrm{trace}(g^2).
\end{equation}
For the semi-simple Lie algebra $\mathfrak{sl}_n$, (\ref{trace}) is the Killing form which is clearly non-degenerate. In the following we prove the non-degeneracy for general $\mathfrak{h}_\pi$.

\begin{lem}\label{lem4.3}
The critical locus of the quadratic function $\mu_\pi$ is the isolated reduced point $0 \in \mathfrak{h}_\pi$, \emph{i.e.},
\[
\mathrm{Crit}(\mu_\pi)\left( = \{d\mu_{\pi}=0\}\right) = \{0\}\subset \mathfrak{h}_\pi.
\]
Consequently, the perverse sheaf of vanishing cycles $\varphi_{\mu_\pi}(\underline{\BC}[\mathrm{dim}({\mathfrak{h}_\pi}])$ (see \cite[Theorem 5.2.21]{Dim}) is the skyscraper sheaf supported on the closed point $0 \in \mathfrak{h}_\pi$. 
\end{lem}

\begin{proof}
To prove the first part of Lemma \ref{lem4.3}, it suffices to show that the quadratic form (\ref{trace}) on $\mathfrak{h}_\pi$ is non-degenerate.

We consider the decomposition 
\begin{equation} \label{decomp1}
\mathfrak{h}_\pi = \mathfrak{h}'_\pi \oplus \mathfrak{c}
\end{equation}
where $\mathfrak{c} \subset \mathfrak{h}_\pi$ is the Cartan subalgebra of trace-free diagonal matrices, and $\mathfrak{h}'_\pi$ consists of the matrices in $\mathfrak{h}_\pi$ such that the entries of the diagonals vanish. For a matrix $g\in \mathfrak{h}_\pi$ with the decomposition
\[
g = g' + c, \quad g' \in \mathfrak{h}'_\pi,~~~c\in \mathfrak{c},
\]
a direct calculation yields $\mathrm{trace}(g'c)=0$. Therefore we have
\[
\mathrm{trace}(g^2) = \mathrm{trace}({g'}^2) + \mathrm{trace}(c^2).
\]
So it suffices to show that the quadratic forms (\ref{trace}) are non-degenerate for both $\mathfrak{h}'_\pi$ and $\mathfrak{c}$. 

We notice that the Cartan subalgebra $\mathfrak{c}$ of $\mathfrak{h}_\pi$ is the same as that of $\mathfrak{sl}_n$. Also, equipped with the quadratic forms (\ref{trace}), $\mathfrak{h}'_\pi$ is a direct summand component of the Lie algebra $\mathfrak{sl}_r^{\times m}$ via the decomposition (\ref{decomp1}) for $\mathfrak{sl}_r^{\times m}$. Since both $\mathfrak{h}_\pi'$ and $\mathfrak{c}$ are direct summand components of semi-simple Lie algebras where the killing forms (\ref{trace}) are non-degenerate, we conclude the non-degeneracy of (\ref{trace}) for $\mathfrak{h}'_\pi$ and $\mathfrak{c}$, which further implies the non-degeneracy of $\mathfrak{h}_\pi$ through (\ref{decomp1}).

This completes the first part of the lemma, and reduce the second part to the case 
\[
\mu_\pi: \BA^N \to \BA^1, \quad (z_1, \dots, z_N) \mapsto \sum_{i=1}^N z_i^2.
\]
In this case, the Milnor fiber is a sphere \cite{Milnor} whose reduced homology computes the vanishing cycle \cite[Proposition 4.2.2]{Dim}.
\end{proof}

The $H_\pi$-invariant function (\ref{eq78}) induces the functions 
\[
\mu_1: [\mathfrak{h}_\pi/H_\pi] \to \BA^1, \quad \mu_2: \mathfrak{h}_\pi\sslash H_\pi \to \BA^1
\]
which form the commutative diagram
\begin{equation}\label{eqn79}
\begin{tikzcd}
\CM_{r,p}(\pi) \arrow[dr, "\mu_1"] \arrow["h_{\pi,p}", d] \\
\CA_p(\pi) \arrow[r, "\mu_2"]
& \BA^1.
\end{tikzcd}
\end{equation}
The pullback of (\ref{eqn79}) along the diagram (\ref{eqn67}) yields the functions
\begin{equation}\label{extra_function}
\mu_{\pi,\CM}: \CM^D_{r,L}(\pi) \to \BA^1,\quad \mu_{\pi,\CA}: \CA^D(\pi) \to \BA^1
\end{equation}
fitting into the commuatative diagram
\begin{equation}\label{eq80}
\begin{tikzcd}
\CM_{r,L}^D(\pi) \arrow[dr, "\mu_{\pi,\CM}"] \arrow["h^D_\pi", d] \\
\CA^D(\pi) \arrow[r, "\mu_{\pi,\CA}"]
& \BA^1.
\end{tikzcd}
\end{equation}

Before stating and proving the main results (Theorem \ref{thm4.4} and its variant Corollary \ref{cor_correction}) of this section, we first note the following standard facts. 

\begin{lem}\label{lem4.5}
 Let $f: V\to \BA^1$ be a regular function.
\begin{enumerate}
    \item[(a)] Assume $V$ admits an action of a finite group $G$ which is fiberwise with respect to $f$. Then the nearby and vanishing cycle functors $\Phi_f, \varphi_f$ are $G$-equivariant.
    \item[(b)] Assume $\CF \in D^b_c(V)$, and assume that $g = \lambda\cdot \mathrm{id} \in \mathrm{End}(\CF)$ is a scaling automorphism of $\CF$ with $\lambda \in \BC^*$. Then applying the nearby or vanishing cycle functor to $g$ yields also an scaling endomorphism:
    \[
    \Phi_f(g) = \lambda\cdot \mathrm{id}: \Phi_f\CF \xrightarrow{\simeq} \Phi_f\CF, \quad  \varphi_f(g) = \lambda\cdot \mathrm{id}: \varphi_f\CF \xrightarrow{\simeq} \varphi_f\CF.
    \]
    \item[(c)] Assume that $g: W \to V$ is smooth with $f' =  f\circ g: W \to \BA^1$ the composition, then we have the smooth pullback formula for vanishing cycles:
    \[
    g^* \circ \varphi_f = \varphi_{f'}\circ g^*: D^b_c(V) \to D_c^b(f'^{-1}(0_{\BA^1})). 
    \]
\end{enumerate}
\end{lem}

Parts (a) and (c) follow directly from the definition of these functors \cite[(1.1.5)]{Illusie}; part (b) holds more generally for any $\BC$-linear exact functor between triangulated categories.

Now we consider the case $s=1$ of Proposition \ref{prop4.1}; let $D$ be the divisor and let $p \in C$ be the point as before.

\begin{thm}\label{thm4.4}
Assume that $D,p$ satisfy (a) and (b) of Proposition \ref{prop4.1}. 
\begin{enumerate}
    \item[(a)] The closed embedding (\ref{eqn76}) can be realized as the critical locus of the function $\mu_{\pi,\CM}: \CM^D_{r,L}(\pi) \to \BA^1$, \emph{i.e.}, we have
    \[
    \CM^{D-p}_{r,L}(\pi) = \mathrm{Crit}\left(\mu_{\pi,\CM}\right)\hookrightarrow \CM^{D}_{r,L}(\pi).
    \]
    \item[(b)] 
    We have a natural isomorphism
    \begin{equation}\label{eqn81}
    \varphi_{\mu_{\pi,\CM}} \underline{\BC} = \underline{\BC}[-r_0]
    \end{equation}
    of $\Gamma$-equivariant objects. Here the first and the second $\underline{\BC}$ stand for the trivial local systems on $\CM^D_{r,L}(\pi)$ and $\CM^{D-p}_{r,L}(\pi)$ respectively with the natural $\Gamma$-structures, and $r_0$ is the codimension of (\ref{eqn76}).
    \item[(c)] For any character $\kappa \in \hat{\Gamma}$, the isomorphism (\ref{eqn81}) induces a natural isomorphism
    \[
    \varphi_{\mu_{\pi,\CA}}\left(\mathrm{Rh^D_\pi}_* \BBC \right)_\kappa = \left(\mathrm{Rh^{D-p}_\pi}_* \BBC \right)_\kappa [-r_0]
    \]
    where $\mu_{\pi,\CA}$ is given in (\ref{eq80}) and $r_0$ is the same as in (b).
\end{enumerate}
\end{thm}

\begin{proof}
By definition, the function $\mu_{\pi,\CM}: \CM^D_{r,L}(\pi) \to \BA^1$ is the composition
\[
\CM^D_{r,L}(\pi) \xrightarrow{\mathrm{ev}_p} \CM_{r,p}(\pi) \xrightarrow{\mu_1} \BA^1
\]
where the first morphism is smooth. Hence we have
\[
\mathrm{Crit}(\mu_{\pi,\CM}) =\mathrm{Crit}\left(\mu_1\circ \mathrm{ev}_p\right) =  \mathrm{ev}_p^{-1}\mathrm{Crit}(\mu_{1}) = \mathrm{ev}_p^{-1}(0_H).
\]
Here the last identity follows from Lemma \ref{lem4.3}. This implies (a) by Lemma \ref{lem4.2}.

Now we prove (b). The smooth pullback of vanishing cycles (Lemma \ref{lem4.5} (c)) yields the canonical equivalence of the functors
\begin{equation}\label{eqn82}
\mathrm{ev}_p^* \circ \varphi_{\mu_1}  = \varphi_{\mu_{\pi,\CM}}\circ \mathrm{ev}_p^*. 
\end{equation}
Since the vanishing cycle complex
\begin{equation}\label{eqn83}
\varphi_{\mu_1} \BBC  \in \mathrm{D}^b_c\left([\mathfrak{h}_\pi/H_\pi]\right)
\end{equation}
is the $H_\pi$-equivariant vanishing cycle complex $\varphi_{\mu_\pi}\BBC$ on $\mathfrak{h}_\pi$, we see from Lemma \ref{lem4.3} that (\ref{eqn83}) is the shifted skyscraper sheaf supported at $0_H$ with the trivial $H_\pi$-action. Hence, applying (\ref{eqn82}) to the trivial local system $\BBC$, we deduce that the vanishing cycle complex $\varphi_{\mu_{\pi,\CM}} \BBC$ is canonically isomorphic to $\BBC$ on $\CM^D_{r,L}(\pi)$ with a shift. Since the functor $\varphi_{\mu_{\pi,\CM}}$ preserves the perverse t-structures, (\ref{eqn81}) is concluded.

For (c), the proper base change of vanishing cycles \cite[Proposition 4.2.11]{Dim} implies the canonical equivalence of the functors
\begin{equation*}\label{eqn84}
{\mathrm{Rh^D_\pi}}_* \circ \varphi_{\mu_{\pi,\CM}} = \varphi_{\mu_{\pi,\CA}} \circ {\mathrm{Rh^D_\pi}}_*.
\end{equation*}
As a consequence, we obtain 
\begin{equation}\label{pushforward}
\varphi_{\mu_{\pi,\CA}}\left(\mathrm{Rh^D_\pi}_* \BBC \right) = \mathrm{Rh^{D-p}_\pi}_* \BBC  [-r_0]
  \end{equation}
by applying the pushforward functor $\mathrm{Rh^D_\pi}_*$ to (\ref{eqn81}) and the fact that $h^D_\pi$ coincides with $h^{D-p}_\pi$ restricting to $\CM^{D-p}_{r,L}$. Noticing that the regular function 
\[
\mu_{\pi,\CM}: \CM^D_{r,L}(\pi) \to \BA^1
\]
passes through the Hitchin base $\CA^D(\pi)$, it admits a fiberwise $\Gamma$-action. By Lemma \ref{lem4.5} (a), the vanishing cycle functor $\varphi_{\mu_{\pi,\CM}}$ is $\Gamma$-equivariant. Hence the isomorphism (\ref{eqn81}) is compatible with the $\Gamma$-equivariant structures on the shifted trivial local systems on $\CM^D_{r,L}(\pi)$ and $\CM^{D-p}_{r,L}(\pi)$. After pushing forward, we see that (\ref{pushforward}) matches for any $\kappa$-isotypic components with respect to the $\Gamma$-action. This completes the proof of (c).
\end{proof}


Next, we consider the actions of the Galois group $G_\pi$. The value of the function $\mu_{\pi,\CA}: \CA^D(\pi) \rightarrow \BA^1$ is constant along each orbit of the Galois group $G_\pi$-action on $\CA^D(\pi)$. Hence $\mu_{\pi,\CA}$ induces a regular function on the $G_\pi$-quotient of $\CA^{D}(\pi)$:
\begin{equation}\label{mu_gamma}
\mu_{\gamma,\CA}: \CA^D_{\gamma} \to \BA^1
\end{equation}
where the element $\gamma \in \Gamma$ corresponds to $\pi: C' \to C$. Alternatively, (\ref{mu_gamma}) is the restriction of the function on the $\mathrm{SL}_n$-Hitchin base
\[
\mu_{\pi=\mathrm{id},\CA}: \CA^D \to \BA^1
\]
to the closed subvariety $i^D_{\gamma}: \CA^D_\gamma \hookrightarrow \CA^D$.

The moduli space $M^D_{r,L}(\pi)$ admits a natural $G_\pi\times \Gamma$-action, which endows the constant sheaf $\BBC$ a natural $G_\pi\times \Gamma$-equivariant structure. The following proposition shows that, if we consider the $G_\pi$-equivariant structures, then (b,c) of Theorem \ref{thm4.4} become more complicated, and do not hold as in the non-equivariant case.

\begin{prop} \label{prop_correction}
We use the same notation as in Theorem \ref{thm4.4}. Let $\BBC$ be the constant sheaf on $M^D_{r,L}(\pi)$ with the standard $G_\pi\times \Gamma$-equivariant structure. There exists a 2-torsion $G_\pi\times \Gamma$ -equivariant local system $\BL$ on $M^{D-p}_{r,L}(\pi)$ (\emph{i.e.} $\BL^{\otimes 2} \simeq \BBC$) of rank 1 satisfying the following properties.
\begin{enumerate}
    \item[(a)] We have
    \begin{equation}\label{correction}
    \varphi_{\mu_{\pi,\CM}} \underline{\BC} = \BL[-r_0]
\end{equation}
as $G_\pi\times \Gamma$-equivariant objects.
\item[(b)] The isomorphism (\ref{correction}) induces
\[\varphi_{\mu_{\pi,\CA}}\left(\mathrm{Rh^D_\pi}_* \BBC \right)_\kappa = \left(\mathrm{Rh^{D-p}_\pi}_* \BL \right)_\kappa [-r_0]
    \]\end{enumerate}
\end{prop}

\begin{proof}
    The proofs are parallel to those for Theorem \ref{thm4.4} (b,c). 
    
    For (a), it suffices to show that, if we apply the vanishing cycle functor to the quadratic function $\mu_\pi: \mathfrak{h}_\pi \to \BA^1$, the constant sheaf $\BBC$ (with the standard $G_\pi\time \Gamma$-structure) is sent to the skyscraper sheaf support at $0 \in \mathfrak{h}_\pi$ whose $G_\pi\times \Gamma$-equivariant structure is at worst 2-torsion. In view of Lemma \ref{lem4.3} and the fact that $\Gamma$ acts trivially on $\CM_{r,p}(\pi)$, the only new part concerns the $G_\pi$-equivariant structure.

    Since $G_\pi$ is a cyclic abelian group of order $m$, we have a decomposition into eigenspaces
    \[
    \mathfrak{h}_\pi= \bigoplus_{i=0}^{m-1}\mathfrak{h}_i,\quad  \quad \xi \cdot v = \xi^iv,~~ \mathrm{for}~~ v\in \mathfrak{h}_i, \quad \xi:= \mathrm{exp}\left(\frac{2\pi\sqrt{-1}}{m} \right).
    \]
The $G_\pi$-invariant quadratic form $\mu_\pi$ admits a decomposition $\mu_\pi = \mu_{\pi,1} + \mu_{\pi,2}+ \mu_{\pi,3}$ where $\mu_{\pi,1}$ is the induced quadractic form on $\bigoplus_{2i\neq 0, m} \left(\mathfrak{h}_i \oplus \mathfrak{h}_{m-i}\right)$, $\mu_{\pi,1}$ is the induced quadractic form on $\mathfrak{h}_0$, and $\mu_{\pi,3}$ is the induced quadractic form on $\mathfrak{h}_{\frac{m}{2}}$. Note that $\mu_{\pi,3}$ is non-trivial only if $m$ is even. By the Thom-Sebastiani theorem, it suffices to show that the vanishing cycle sheaf associated with each quadratic function has 2-torsion $G_\pi$-equivariant structure.

For the quadratic form $\mu_{\pi,1}$, we note that the $G_\pi$-action on $\mathfrak{h}_i\oplus \mathfrak{h}_{m-i}$ can be extended to a $\BG_m$-action; therefore the $G_\pi$-structure on the reduced homology of the Milnor fiber is trivial due to the fact that $\BG_m$ is continuous. For the quadratic form $\mu_{\pi,2}$, since the $G_\pi$-action is trivial on $\mathfrak{h}_0$, it is clear that the $G_\pi$-structure on the vanishing cycle is also trivial. Finally, we assme $m$ is even and treat $\mu_{\pi,3}$; using again the Thom-Sebastiani theorem, we can reduce to the quadratic form $z^2$ on $\BA^1$ with the $G_\pi$-action given by the composition of the quotient map $G_\pi \to \BZ/2\BZ$ and the standard involution on $\BA^1$. The vanishing cycle sheaf is the skyscraper sheaf $0_{\BA^1}$ with a nontrivial $\BZ/2\BZ$-structure. Consequently, the $G_\pi$-structure is 2-torsion. We have completed the proof for (a). The proof for (b) is identical to that for Theorem \ref{thm4.4} (c).
\end{proof}

\begin{rmk}
The proof of Proposition \ref{prop_correction} (a) shows that the $G_\pi$-local system $\BL$ can indeed be nontrivial; the first case it happens is when $n=m=2$ and $r=1$. The $G_\pi$-equivariant structure was overlooked in a previous version of the paper, and we are grateful to Elsa Maneval for helpful correspondence that led us to discover this.
\end{rmk}

For our purpose, we need a ``two-point" $G_\pi$-equivariant version of Theorem \ref{thm4.4}. We consider $D$ and two distinct points $p,q\in C$ satisfying (a) and (b) of Proposition \ref{prop4.1}, and define a new function $\mu'_{\pi,\CM}$ on $\CM^D_{r,L}(\pi)$ as
\[
\mu'_{\pi,\CM}: \CM^D_{r,L}(\pi) \xrightarrow{\mathrm{ev}_{p,q}} \CM_{r,p}(\pi) \times \CM_{r,q}(\pi) \xrightarrow{\mu_1\times \mu_1} \BA^1\times \BA^1 \xrightarrow{+} \BA^1.
\]
This induces a ``two-point" version of (\ref{eq80}):

\begin{equation}\label{eq80_two_point}
\begin{tikzcd}
\CM_{r,L}^D(\pi) \arrow[dr, "\mu'_{\pi,\CM}"] \arrow["h^D_\pi", d] \\
\CA^D(\pi) \arrow[r, "\mu'_{\pi,\CA}"]
& \BA^1.
\end{tikzcd}
\end{equation}
Furthermore, the $G_\pi$-invariant function $\mu'_{\pi,\CA}$ descends to 
\[
\mu'_{\gamma, \CA}: \CA^D_\gamma \to \BA^1.
\]

The following is an immediate consequence of (the proof of) Proposition \ref{prop_correction} (a), Proposition \ref{prop4.1}, and the Thom-Sebastiani theorem applied to \[\CM_{r,p}(\pi) \times \CM_{r,q}(\pi) \xrightarrow{\mu_1\times \mu_1} \BA^1\times \BA^1 \xrightarrow{+} \BA^1.\]
The reason we consider two points $p,q$ at once is due to the appearance of the 2-torsion $G_\pi$-equivariant local system $\BL$ in Proposition \ref{prop_correction}.

\begin{cor}\label{cor_correction} The following statements hold.
\begin{enumerate}
    \item[(a)] The closed embedding $\CM^{D-p-q}_{r,L}(\pi) \hookrightarrow \CM^D_{r,L}(\pi)$ of (\ref{eqn76}) can be realized as the critical locus of the function $\mu'_{\pi,\CM}$, \emph{i.e.} we have
    \[
    \CM^{D-p-q}_{r,L}(\pi) = \mathrm{Crit}\left(\mu'_{\pi,\CM}\right)\hookrightarrow \CM^{D}_{r,L}(\pi).
    \]
    \item[(b)]  We have
    \begin{equation}\label{correction'}    \varphi_{\mu'_{\pi,\CM}}\underline{\BC} = \BBC[-r_0].
    \end{equation}
    Here the first and the second $\underline{\BC}$ stand for the trivial $G_\pi\times \Gamma$-equivariant local systems on $\CM^D_{r,L}(\pi)$ and $\CM^{D-p-q}_{r,L}(\pi)$ respectively, and $r_0$ is the codimension of (\ref{eqn76}).    \item[(c)] The isomorphism (\ref{correction'}) induces
\[\varphi_{\mu'_{\pi,\CA}}\left(\mathrm{Rh^D_\pi}_* \BBC \right)_\kappa = \left(\mathrm{Rh^{D-p-q}_\pi}_* \BBC\right)_\kappa [-r_0]
    \]\end{enumerate} 
\end{cor}

Recall the notation from the diagram (\ref{diagram111}). We have the following corollary.

\begin{cor}\label{cor4.5}
Corollary \ref{cor_correction} (c) induces for any character $\kappa \in \hat{\Gamma}$ a natural isomorphism
\[
\varphi_{\mu'_{\gamma,\CA}} \left(\mathrm{Rh^D_\gamma}_* \BBC \right)_\kappa = \left(\mathrm{Rh^{D-p-q
}_\gamma}_* \BBC \right)_\kappa [-r_0].
\]
Here $r_0$ is the same as in Corollary \ref{cor_correction} (b,c).
\end{cor}
\begin{proof}
We consider the quotient map $q_\CA: \CA^D(\pi) \to \CA^D_\gamma$. By the proper base change \cite[Proposition 4.2.11]{Dim}, we have
\begin{equation}\label{eq88}
\varphi_{\mu'_{\gamma,\CA}} \left({q_\CA}_*\mathrm{Rh^D_\pi}_* \BBC \right) = {q_\CA}_* \varphi_{\mu'_{\pi,\CA}}
 \left(\mathrm{Rh^D_\pi}_* \BBC \right).
\end{equation}
Similar as in the proof of Theorem \ref{thm4.4} (c), the identity (\ref{eq88}) is compatible with the $G_\pi$- and $\Gamma$-actions on both sides. By taking the $G_\pi$-invariant and the $\kappa$-isotypic parts, we obtain that 
\[
\begin{split}
  \varphi_{\mu'_{\gamma,\CA}}  \left(({q_\CA}_*\mathrm{Rh^D_\pi}_* \BBC)^{G_\pi}\right)_\kappa = \left({q_\CA}_* \varphi_{\mu'_{\pi,\CA}}
 \left(\mathrm{Rh^D_\pi}_* \BBC\right)_\kappa \right)^{G_\pi} = \left({q_\CA}_*\mathrm{Rh^{D-p-q}_\pi}_* \BBC \right)_\kappa^{G_\pi}[-r_0]
\end{split}
\]
where we use Theorem \ref{thm4.4} (c) in the last identity. Hence Lemma \ref{lemma1.6} implies that
\[
\varphi_{\mu'_{\gamma,\CA}} \left(\mathrm{Rh^D_\gamma}_* \BBC \right)_\kappa = \left(\mathrm{Rh^{D-p-q}_\gamma}_* \BBC \right)_\kappa [-r_0]. \qedhere
\]
\end{proof}

\begin{rmk}
A different choice of trivialization (\ref{trivialization}) affects the functions (\ref{extra_function}), and therefore the vanishing cycle sheaves, by a scalar $\lambda \in \BC^*$. Since all the isomorphisms in our main theorems are only constructed up to scaling, choices of trivializations at the points $p,q$ do not matter for our purpose.
\end{rmk}

\subsection{Proof of Theorem \ref{thm3.2}}

In this section, we prove Theorem \ref{thm3.2} by constructing the operator 
\begin{equation}\label{C-Corr}
c^D_\kappa: \left(\mathrm{Rh^D_*} \BBC\right)_\kappa \xrightarrow{\simeq} {i^D_\gamma}_*\left(\mathrm{Rh^D_\gamma}_* \BBC\right)_\kappa [-2d^D_\gamma].
\end{equation}
This recovers the $G_\pi$-equivariant operator 
\[
q_\CA^*\left(c^D_\kappa\right): q_\CA^* \left( \mathrm{Rh^D_*} \BBC \right)_{\kappa} \Big{|}_{\CA^D(\pi)^*} \xrightarrow{~~\simeq~~} \left( \mathrm{Rh_\pi^D}_* \BBC \right)_{\kappa}\Big{|}_{\CA^D(\pi)^*}[-2d^D_\gamma].
\]
for Theorem \ref{main0'} as explained in Section \ref{Sec3.1}. 

When $\mathrm{deg}(D)$ is even and greater than $2g-2$, the operator $c^D_\kappa$ is constructed in Theorem \ref{main1} from Ng\^o's endoscopic correspondence, which is canonical up to scaling. The only remaining case is $D= K_C$, where the support theorems fail for the Hitchin fibrations (\emph{c.f.} \cite{Supp}). We are able to construct the operator (\ref{C-Corr}) via Corollary \ref{cor_correction} as we explain in the following.

Let $p,q$ be 2 distinct closed points on the curve $C$. We consider the divisor 
\[
D: =K_C+p+q,  \quad \mathrm{deg}(D) = 2g > 2g-2.
\]
By Theorem \ref{main1}, we already have the isomorphism (\ref{C-Corr}) for this divisor $D$. Applying to (\ref{C-Corr}) the vanishing cycle functor $\varphi_{\mu'_{\CA}}$ associated with the function 
\[
\mu_\CA :=\mu'_{\mathrm{id},\CA}: \CA^{D} \to \BA^1,
\]
we obtain 
\begin{equation}\label{eqn88}
\varphi_{p,q}(c^{D}_\kappa):= \varphi_{\mu_\CA}(c^{D}_\kappa):  \varphi_{\mu_{\CA}}\left(\mathrm{Rh^{D}_*} \BBC\right)_\kappa \xrightarrow{\simeq}  \varphi_{\mu_{\CA}}\left({i^{D}_\gamma}_*\left(\mathrm{Rh^{D}_\gamma}_* \BBC\right)_\kappa \right) [-2d^{D}_\gamma].
\end{equation}
Since (\ref{C-Corr}) is canonical up to scaling, so is (\ref{eqn88}) by Lemma \ref{lem4.5} (b). In fact, the operator (\ref{eqn88}) gives the desired operator $c^{D}_\kappa$. To justify this, we calculate both sides of (\ref{eqn88}).

For the left-hand side, Theorem \ref{thm4.4} (c) applied to the special case $\pi = \mathrm{id}$ yields
\[
\varphi_{\mu_{\CA}}\left(\mathrm{Rh^{D}_*} \BBC\right)_\kappa  = \left(\mathrm{Rh^{K_C}_*} \BBC\right)_\kappa[-r_1], \quad r_1 = \mathrm{codim}_{\CM^{D}_{n,L}}(\CM_{n,L}).
\]

For the right-hand side, we have by the proper base change (\cite[Proposition 4.2.11]{Dim}) and Corollary \ref{cor4.5} that
\[
\begin{split}
    \varphi_{\mu_{\CA}}\left({i^{D}_\gamma}_*\left(\mathrm{Rh^{D}_\gamma}_* \BBC\right)_\kappa \right) & = {i^{D}_\gamma}_*\varphi_{\mu'_{\gamma,\CA}}\left(\mathrm{Rh^{D}_\gamma}_* \BBC\right)_\kappa \\
    &= {i^{K_C}_\gamma}_*\left(\mathrm{Rh^{K_C}_\gamma}_* \BBC \right)_\kappa [-r_2]
\end{split}
\]
where 
\[
r_2 = \mathrm{codim}_{\CM^{D}_{r,L}(\pi)}(\CM_{r,L}(\pi)).
\]
In conclusion, (\ref{eqn88}) gives an isomorphism
\[
\varphi_{p,q}(c^{D}_\kappa):  \left(\mathrm{Rh^D_*} \BBC\right)_\kappa \xrightarrow{\simeq} {i_\gamma}_*\left(\mathrm{Rh^{D}_\gamma}_* \BBC\right)_\kappa [-2d^{D_p}_\gamma+r_1-r_2].
\]
By the dimension formulas in \cite[Section 6.1]{dC_SL}, we have
\[
\begin{split}
r_1 -r_2  & = \left(\mathrm{dim}(\CM^{D}_{n,L}) - \mathrm{dim}(\CM_{n,L})\right) - \left(\mathrm{dim}(\CM^{D}_{r,L}(\pi) - \mathrm{dim}(\CM_{r,L}(\pi)\right)\\ & =2d^{D}_\gamma -  2d_\gamma.
\end{split}
\]
Hence $-2d^{D}_\gamma+r_1-r_2 = -  2d_\gamma$, and the operator
\begin{equation}\label{eq92}
    c^{K_C}_\kappa = \varphi_{p,q}(c^{D}_\kappa)
\end{equation}
induces an isomorphism (\ref{C-Corr}) as desired.

We have completed the construction of (\ref{C-Corr}) which proves Theorems \ref{main0'} and \ref{thm3.2}. The construction of the operator $c^{K_C}_\kappa$ \emph{a priori} depends on the choice of the closed points $p$ and $q$. We finish this section by showing in the following proposition that $c^{K_C}_\kappa$ in Case 2 or 3 is, in fact, independent of the choice of the points.

\begin{prop}
The operators (\ref{eq92}) does not depend on the choice of $p,q\in C$.
\end{prop}

\begin{proof}
Varying the point $p$, we have a family of Hitchin fibrations
\[
h^{K_C+p+q}: \CM^{K_C+p+q}_{n,L} \to \CA^{K_C+p+q}
\]
over a base $T$. The construction of the correspondence (\ref{eqn88}) works relatively over the base which gives a family of operators $c^{K_C+p+q}_\kappa$. By applying the vanishing cycle functor relatively over $T$, we obtain a family of operators $\varphi_{p,q}(c^{K_C+p+q}_\kappa)$ which form a section of the trivial local system
\[
\BBC \otimes \mathrm{Hom}_{\CA^D}(\CF_1, \CF_2)
\]
on $T$. Here 
\[
\CF_1 = \left(\mathrm{Rh^{K_C+p+q}_*} \BBC\right)_\kappa,\quad  \CF_2 = {i_\gamma}_*\left(\mathrm{Rh^{K_C+p+q}_\gamma}_* \BBC\right)_\kappa [-2d^{K_C+p+q}].
\]
are independent of the points $p,q$. Hence $\varphi_{p,q}(c^{K_C+p+q}_\kappa)$ is constant over $T$. 
\end{proof}

\begin{rmk}\label{rmk4.8}
Applying the vanishing cycle functors as above, we obtain that Proposition \ref{Prop2.11} also holds for $D= K_C$.
\end{rmk}

\begin{rmk}\label{odd_deg}
When $\mathrm{deg}(D)$ is odd, applying the vanishing cycle functor and Proposition \ref{prop_correction} (b) yields an isomorphism
\[
\left(\mathrm{Rh^D_*} \BBC\right)_\kappa \xrightarrow{\simeq} {i_\gamma}_*\left(\mathrm{Rh^D_\gamma}_* \CL\right)_\kappa [-2d^D_\gamma] \in D_c^b(\CA^D),
\]
where $\CL$ is a 2-torsion $\Gamma$-equivariant local system of rank 1 induced by $\BL$ on $M^D_{r,L}(\pi)$. We see from the proof of Proposition \ref{prop_correction} that $\CL$ can be non-trivial only when $m$ is even and $r$ is odd (in particular, $n=mr$ is even). 
\end{rmk}

\section{The P=W conjecture and the Hausel-Thaddeus conjecture} \label{Section5}

Throughout Section \ref{Section5}, we assume that the curve $C$ has genus $g \geq 2$. We assume that $D$ is an effective divisor such that $\mathrm{deg}(D)$ is even and greater than $2g-2$, or $D = {K}_C$. For a cyclic Galois cover $\pi: C' \to C$, we denote by $D'$ the divisor $\pi^*D$ on $C'$.

We discuss some applications of Theorems \ref{main0'} and \ref{thm3.2}.

\subsection{Perverse filtrations}\label{Sec5.1}
We briefly recall the definition of perverse filtrations \cite{dCM0, dCHM1}.

Let $f: X \rightarrow Y$ be a proper morphism with $X$ a nonsingular algebraic variety. The perverse $t$-structure on the constructible derived category $D_c^b(Y)$ induces an increasing filtration on the cohomology $H^*(X, \BC)$,
\begin{equation} \label{Perv_Filtration}
    P_0H^\ast(X, \BC) \subset P_1H^\ast(X, \BC) \subset \dots \subset P_kH^\ast(X, \BC) \subset \dots \subset H^\ast(X, \BC),
\end{equation}
called the \emph{perverse filtration} associated with $f$. 

The perverse filtration (\ref{Perv_Filtration}) can be described via the decomposition theorem \cite{BBD}. In fact, applying the decomposition theorem to the map $f: X \to Y$, we obtain an isomorphism
\begin{equation*}
\mathrm{Rf}_*\BBC[\dim(X)-l]\simeq \bigoplus_{i=0}^{2l}\mathcal{P}_i[-i] \in D^b_c(Y)
\end{equation*}
with $\mathcal{P}_i$ a perverse sheaf on $Y$ and $l$ the defect of semi-smallness:
\[
l= \mathrm{dim}\left( X \times_Y X\right) - \mathrm{dim}(X).
\]
The $k$-th piece of the perverse filtration is
\[
P_kH^j(X,\BQ)=\mathrm{Im}\Big\{H^{j-(\dim(X) - l)}(Y, \bigoplus_{i=0}^k\mathcal{P}_i[-i])\to H^j(X,\BQ)\Big\}.
\]

\subsection{The P=W conjecture}
Perverse filtrations appear naturally in studying the topology of Hitchin fibrations. For notational convenience, we let
\[
h: \CM \to \CA
\]
be the Hitchin fibration with $\CM = \CM_{n,L}$ or $\widetilde{\CM}_{n,d}$.\footnote{Here the divisor $D$ is chosen as the canonical divisor $K_C$.} We denote by $\CM^{B}$ the Betti moduli space associated with $\CM$. There is a diffeomorphism $\CM \cong \CM^{B}$ induced by non-abelian Hodge theory \cite{Simp, Si1994II,HT2}, which identifies the cohomology
\begin{equation}\label{NonAbel}
    H^*(\CM, \BC) = H^*(\CM^{B}, \BC).
\end{equation}

A central question concerning the cohomological aspect of the non-abelian Hodge theory is \emph{the P=W conjecture} formulated by de Cataldo--Hausel--Migliorini \cite{dCHM1}, connecting the perverse filtration associated with the Hitchin fibration $h$ to the weight filtration \[
W_{0}H^*(\CM^{B}, \BC) \subset W_{1}H^*(\CM^{B}, \BC) \subset \cdots \subset W_kH^*(\CM^{B}, \BC) \subset \cdots \subset H^*(\CM^{\mathrm{B}}, \BC)
\]
associated with the mixed Hodge structure on $\CM^{\mathrm{B}}$.

\begin{conj}["P=W" \cite{dCHM1}]
Under the non-abelian Hodge correspondence (\ref{NonAbel}), we have
\[
P_kH^i(\CM, \BC) = W_{2k}H^i(\CM^B, \BC).
\]
\end{conj}

For the $\mathrm{GL}_n$-case, the P=W conjecture was proven for $n=2$ in \cite{dCHM1}, and recently, for $g=2$ \cite{dCMS}. Furthermore, \cite{dCMS} reduces the full P=W conjecture to the multiplicativity of the perverse filtration; see \cite[Introduction]{dCMS} for the precise statement. In either situation, the way for attacking the P=W conjecture is to analyze the location of the \emph{tautological classes} in both the perverse and the weight filtrations.

The case of $\mathrm{SL}_n$ is more complicated due to the lack of tautological classes accessing the $\Gamma$-variant cohomology. When $n$ is a prime number, the shapes of the perverse and the weight filtrations on the $\Gamma$-variant parts are of simpler forms, and therefore the P=W conjecture was verified for the $\Gamma$-variant cohomology via direct calculations; see \cite{dCHM1} for $n=2$ and \cite{dCMS2} for any prime number $n$. 

When $n$ is not a prime number, numerical evidence from the Hausel--Thaddeus conjecture suggests that the P=W conjecture for $\mathrm{SL}_n$ should rely on the P=W conjecture for a sequence of moduli spaces of stable $\mathrm{GL}_*$-Higgs bundles on different curves with different ranks. In particular, we expect that the P=W conjecture for $\mathrm{SL}_n$ can be eventually reduced to the P=W conjecture for $\mathrm{GL}_r$.

As a first step towards this direction, we introduce the operator (\ref{q_kappa}) below connecting $H^*(\CM_{n,L}, \BC)_\kappa$ and the cohomology of the moduli space of stable $\mathrm{GL}_r$-Higgs bundles on another curve $C'$, where $r$ and $C'$ are determined by $\kappa \in \hat{\Gamma}$. Then we prove Theorem \ref{thm5.4} on the compatibility of the perverse filtrations.

\subsection{The stable cohomology}\label{5.3}
Let $\pi: C'\to C$ be a cyclic Galois cover of degree $m$. Let $L \in \mathrm{Pic}^d(C)$ be a fixed line bundle with $\mathrm{gcd}(n,d)=1$. 

We recall the moduli spaces $\widetilde{\CM}_{r,d}^{D'}(C')$ and $\CM_{r,L}^{D}(\pi)$ as well as their Hitchin fibrations (\ref{GL_Hit}) and (\ref{eqn_h_pi}) respectively. The group scheme
\[
\widetilde{\CM}^D_{1,0}(C) = \mathrm{Pic}^0(C)\times H^0(C, \CO_C(D))
\]
acts on the moduli space $\widetilde{\CM}_{r,d}^{D'}(C')$ inducing 
\begin{equation}\label{action}
     \widetilde{q}:  \widetilde{\CM}^D_{1,0}(C)  \times  \CM_{r,L}^{D}(\pi)  \to \widetilde{\CM}_{r,d}^{D'}(C'), \quad \left( (\CE_1, \theta_1), (\CE_r, \theta_r) \right)  \mapsto (\pi^*\CE_1 \otimes \CE_r, \pi^*\theta_1+ \theta_r).
\end{equation}
Here $\theta_1 \in H^0(C, \CO_C(D))$ and its pullback gives a section $\pi^*\theta_1 \in H^0(C', \CO_{C'}(D'))$. The finite group $\Gamma$ acts on the left-hand side of (\ref{action}) diagonally,
\[
\CL \cdot \left( (\CE_1, \theta_1), (\CE_r, \theta_r) \right) =  \left((\CE_1 \otimes \CL^{-1}, \theta_1),  (\CE_r \otimes \CL, \theta_r)\right), \quad \CL \in \Gamma.
\]
The morphism (\ref{action}) factors through this $\Gamma$-quotient and the fibers of (\ref{action}) are given by $\Gamma$-orbits. By dimension reasons, the right-hand side of (\ref{action}) coincides with the $\Gamma$-quotient of the left-hand side. We have the following canonical isomorphisms of the cohomology
\begin{equation}\label{97}
\begin{split}
H^*(\widetilde{\CM}_{r,d}^{D'}(C'), \BC) & \xrightarrow{\simeq} H^*(\widetilde{\CM}^D_{1,0}(C)  \times  \CM_{r,L}^{D}(\pi) , \BC)^\Gamma \\ & = \left(H^*(\widetilde{\CM}^D_{1,0}(C), \BC)  \otimes H^*( \CM_{r,L}^{D}(\pi) , \BC) \right)^\Gamma\\
& = H^*(\widetilde{\CM}^D_{1,0}(C), \BC)  \otimes H^*( \CM_{r,L}^{D}(\pi) , \BC)^\Gamma 
\end{split}
\end{equation}
where the first isomorphism is induced by the $\Gamma$-quotient map $\widetilde{q}^*$, the second identity is the K\"unneth decomposition, and the last identity follows from the triviality of the $\Gamma$-action on $H^*(\widetilde{\CM}^D_{1,0}(C), \BC)$.

For any Hitchin-type moduli space $\widetilde{\CM}^D_{n,d}$, $\CM^D_{n,L}$, or $\CM^D_{r,L}(\pi)$, we consider the perverse filtrations on the cohomology defined via the corresponding Hitchin fibration (\ref{GL_Hit}), (\ref{SLn_Hit}), or (\ref{eqn_h_pi}) respectively. The following proposition provides a description of the perverse filtration on the stable part of $H^*( \CM_{r,L}^{D}(\pi) , \BC)$.

\begin{prop}\label{prop5.1}
The quotient map (\ref{action}) induces a canonical isomorphism
\begin{equation}\label{5.1}
H^*(\widetilde{\CM}_{r,d}^{D'}(C'), \BC) = H^*(\widetilde{\CM}^D_{1,0}(C), \BC)  \otimes H^*( \CM_{r,L}^{D}(\pi) , \BC)_{\mathrm{st}}
\end{equation}
satisfying that
\begin{equation}\label{perv1}
P_kH^*(\widetilde{\CM}_{r,d}^{D'}(C'), \BC) = \bigoplus_{i+j = k} H^i(\widetilde{\CM}^D_{1,0}(C), \BC) \otimes P_jH^*( \CM_{r,L}^{D}(\pi) , \BC)_{\mathrm{st}}.
\end{equation}
Here the stable part $(-)_\mathrm{st}$ denotes the $\Gamma$-invariant part of the cohomology.
\end{prop}

\begin{proof}
The first isomorphism is induced by (\ref{97}). It suffices to show the compatibility (\ref{perv1}) of the perverse filtrations.

We notice that the quotient map (\ref{action}) is compatible with the Hitchin fibrations, and we have the commutative diagram
\begin{equation*}
\begin{tikzcd}
 \widetilde{\CM}^D_{1,0}(C)  \times  \CM_{r,L}^{D}(\pi) \arrow[r,"\widetilde{q}"] \arrow[d, "\widetilde{h}^{D}(C) \times h_\pi^D"]
& \widetilde{\CM}_{r,d}^{D'}(C') \arrow[d, "\widetilde{h}^{D'}"] \\
H^0(C, \CO_C(D)) \times \CA^D(\pi) \arrow[r, "\simeq"]
& \widetilde{\CA}^{D'}(C')
\end{tikzcd}
\end{equation*}
where the bottom arrow is a canonical identification. The pullback morphism $\widetilde{q}^*$ for the cohomology is induced sheaf-theoretically by the canonical morphism
\begin{equation}\label{add3}
\BBC \to \widetilde{q}_* \BBC
\end{equation}
where the first and the second $\BBC$ denote the trivial local systems on the target and the source of $\widetilde{q}$ respectively. By applying the perverse truncation functor to the pushforward of (\ref{add3}) along $\widetilde{h}^{D'}$, we obtain that the first map of (\ref{97}) satisfies
\begin{equation}\label{add4}
P_kH^*(\widetilde{\CM}_{r,d}^{D'}(C'), \BC)  \xrightarrow{\simeq} P_kH^*(\widetilde{\CM}^D_{1,0}(C)  \times  \CM_{r,L}^{D}(\pi) , \BC)^\Gamma.
\end{equation}
We conclude (\ref{perv1}) from (\ref{add4}), the K\"unneth decomposition, and the fact that the perverse filtration on \[ 
H^*(\widetilde{\CM}^D_{1,0}(C), \BC) = H^*(\mathrm{Pic}^0(C), \BC)
\]
coincides with the cohomological filtration $H^{*\leq k}$.
\end{proof}

As a consequence of Proposition \ref{prop5.1}, we obtain a canonical operator given by the projection
\begin{equation}\label{100}
    \mathfrak{p_1}: H^i(\widetilde{\CM}_{r,d}^{D'}(C'), \BC) \rightarrow H^i( \CM_{r,L}^{D}(\pi) , \BC)_{\mathrm{st}}, \quad \forall i\geq 0
\end{equation}
sending a class in $H^i(\widetilde{\CM}_{r,d}^{D'}(C'), \BC)$ to its projection to the direct summand component 
\[
H^0(\widetilde{\CM}^D_{1,0}(C), \BC)  \otimes H^i( \CM_{r,L}^{D}(\pi) , \BC)_{\mathrm{st}} =  H^i( \CM_{r,L}^{D}(\pi) , \BC)_{\mathrm{st}}.
\]
with respect to the decomposition (\ref{5.1}). The identity above is induced by the fundamental class   $1\in H^0(\widetilde{\CM}^D_{1,0}(C), \BC)$.

\begin{cor}\label{cor5.3}
The operator (\ref{100}) respects the perverse filtrations:
\[
\mathfrak{p_1}\left( P_k H^i(\widetilde{\CM}_{r,d}^{D'}(C'), \BC)\right) = P_kH^i( \CM_{r,L}^{D}(\pi) , \BC)_{\mathrm{st}}.
\]
\end{cor}

\begin{proof}
Since the fundamental class $1\in H^0(\widetilde{\CM}^D_{1,0}(C), \BC)$ lies in $P_0H^0(\widetilde{\CM}^D_{1,0}(C), \BC)$, the corollary follows from (\ref{perv1}).
\end{proof}

\subsection{Operators}
Let $\pi: C' \to C$ be the cyclic Galois cover given by $\gamma\in \Gamma$, which corresponds to $\kappa \in \hat{\Gamma}$ via (\ref{Weil}). We define the operator 
\begin{equation}\label{q_kappa}
\mathfrak{p}_\kappa: H^{i-2d^D_\gamma}(\widetilde{\CM}_{r,d}^{D'}(C'), \BC) \twoheadrightarrow H^{i}(\CM^D_{n,L}, \BC)_\kappa
\end{equation}
as the following composition:
\begin{equation}
\begin{split}\label{104}
    H^*(\widetilde{\CM}_{r,d}^{D'}(C'),\BC) \xrightarrow{\mathfrak{p}_1} & H^*( \CM_{r,L}^{D}(\pi) , \BC)_{\mathrm{st}}  =   H^*( \CM_{r,L}^{D}(\pi) , \BC)_{\kappa}  \\ \xrightarrow{\mathrm{proj.}}   H^*( \CM_{r,L}^{D}(\pi) , \BC)^{G_\pi}_{\kappa} 
=  & H^*( \CM_{\gamma}^{D} , \BC)_{\kappa} \xrightarrow[\simeq]{\mathrm{Thm}~~3.2} H^{*+2d^D_\gamma}(\CM^D_{n,L}, \BC)_\kappa.
\end{split}
\end{equation}
Here the first morphism is (\ref{100}), the second isomorphism is given by Proposition \ref{prop2.8}, the third morphism is the projection to the $G_\pi$-invariant part, the fourth isomorphism is given by Lemma \ref{lemma1.6}, and the last isomorphism follows from Theorem \ref{thm3.2}. Hence we obtain that (\ref{q_kappa}) is surjective and canonically defined up to scaling.

\begin{thm}\label{thm5.4}
We have 
\[
\mathfrak{p}_\kappa \left( P_k H^i(\widetilde{\CM}_{r,d}^{D'}(C'), \BC)\right) = P_{k+d^D_\gamma}H^{i+2d^D_\gamma}( \CM_{n,L}^{D}, \BC)_{\kappa}.
\]
\end{thm}

\begin{proof}
By Corollary \ref{cor5.3}, the morphism $\mathfrak{p}_1$ preserves the perverse filtrations. All the other morphisms in (\ref{104}) except the last one are deduced from sheaf-theoretic morphisms which clearly preserve the perverse filtrations. Hence we have
\[
P_kH^i(\widetilde{\CM}_{r,d}^{D'}(C'), \BC) \twoheadrightarrow P_k H^i(\CM^D_\gamma, \BC)_\kappa.
\]
The last morphism of (\ref{104}) is given by the sheaf-theoretic isomorphism (\ref{ck^D}). Taking account of the shift, we have
\[
P_k H^i(\CM^D_\gamma, \BC)_\kappa \xrightarrow{\simeq} P_{k+d^D_\gamma}H^{i+2d^D_\gamma}( \CM_{r,L}^{D}(\pi) , \BC)_{\kappa}. \qedhere
\]
\end{proof}

Now we consider the special case $D = K_C$. Passing through the isomorphisms (\ref{NonAbel}) induced by the non-abelian Hodge theory, we obtain an operator for the corresponding Betti moduli spaces
\[
\mathfrak{p}^B_\kappa: H^{i-2d_\gamma}(\widetilde{\CM}_{r,d}^{B}(C'), \BC) \twoheadrightarrow H^{i}(\CM^B_{n,L}, \BC)_\kappa.
\]
Here $\widetilde{\CM}_{r,d}^{B}(C')$ is the Betti moduli space associated with the curve $C'$, the group $\mathrm{GL}_r$, and the degree $d$, and $\CM^B_{n,L}$ stands for the Betti moduli space associated with the curve $C$, the group $\mathrm{SL}_n$, and the line bundle $L$. We refer to \cite{HT2} for more details on these moduli spaces.

\begin{question}\label{Q5.5}
Is it true that
\[
\mathfrak{p}^B_\kappa \left( W_{2k} H^i(\widetilde{\CM}_{r,d}^{B}(C'), \BC)\right) = W_{2k+2d_\gamma}H^{i+2d_\gamma}( \CM_{n,L}^{B}, \BC)_{\kappa}?
\]
\end{question}

If Question \ref{Q5.5} has an affirmative answer, then Theorem \ref{thm5.4} implies that, if the P=W conjecture holds for $\mathrm{GL}_r$ for any $r$ dividing $n$, then the P=W conjecture holds for $\mathrm{SL}_n$. However, the construction of the operator $\mathfrak{p}_\kappa$ relies heavily on the topology of Hitchin fibrations, which is mysterious on the Betti side. A better understanding of the operator $\mathfrak{p}_\kappa$ may be needed.

\subsection{The Hausel--Thaddeus conjecture}\label{Section5.5}
We explain in the last section that Theorem \ref{thm3.2} implies Theorem \ref{thm0.5}. Here for Higgs bundles, we again work with any effective divisor $D$ with $\mathrm{deg}(D)>2g-2$ or $D= K_C$.

\begin{proof}[Proof of Theorem \ref{thm0.5}]
We first note that for two line bundles $L_1$ and $L_2$ with $L_1 = L_2 \otimes N^{\otimes n}$, there is a natural identification of the moduli spaces 
\begin{equation}\label{eqn127}
    \CM^D_{n, L_1} \xrightarrow{\simeq} \CM^D_{n, L_2}, \quad (\CE, \theta) \mapsto (\CE\otimes N, \theta)
\end{equation}
compatible with the Hitchin fibrations.
Hence we obtain
\begin{equation}\label{111}
\begin{split}
[P_kH^i(\CM^D_{n,L}, \BC)_\kappa] = [P_kH^i(\CM^D_{n,L'^{\otimes de}}, \BC)_\kappa] &=   [P_{k-d^D_\gamma}H^{i-2d^D_\gamma}((\CM^D_{n,L'^{\otimes de}})_\gamma, \BC)_{\kappa}] \\
&= [P_{k-d^D_\gamma}H^{i-2d^D_\gamma}((\CM^D_{n,L'})_\gamma, \BC)_{de\kappa}] 
\end{split}
\end{equation}
in the Grothendieck group $K_0(\mathrm{Vect})$ of $\BC$-vector spaces. Here the first identity is induced by (\ref{eqn127}) since $\mathrm{deg}(L) = \mathrm{deg}(L'^{\otimes de})~~\textup{mod}~~n$, the second identity follows from (\ref{ck^D}), and the third identity is given by Remark \ref{rmk4.8} and Proposition \ref{Prop2.11}.

This proves the Betti number version of the refined Hausel--Thaddeus conjecture (\ref{eqnthm0.5}). To get the enhanced version concerning Hodge structures, we follow \cite[Section 2.1]{dCRS} to work with the category of mixed Hodge modules \cite{Saito} which refines the category of perverse sheaves. Identical arguments show that Theorems \ref{main0'} and \ref{thm3.2} actually hold in the derived category of mixed Hodge modules, which gives the enhanced version  
of (\ref{111}) in $K_0(\mathrm{HS})$. This completes the proof of (\ref{eqnthm0.5}).

Finally, we note that (\ref{eqnthm0.5}) implies (\ref{eqnthm_new}). This follows from taking the summation over all $\gamma \in \Gamma$ and the natural identification of the fixed loci
\[
(\CM^D_{n,L'})_\gamma =  (\CM^D_{n,L'})_{q\gamma}
\]
for any $q \in \BZ$ coprime to $n$.
\end{proof}



\begin{thebibliography}{99}

\bibitem{AB} M. Atiyah, R. Bott, {\em The Yang-Mills equations over Riemann surfaces,}
Philos. Trans. Roy. Soc. London Ser. A 308 (1983), no. 1505, 523--615.









\bibitem{Beau} A. Beauville, {\em Sur la cohomologie de certains espaces de modules de fibr\'es vectoriels,} Geometry and analysis (Bombay, 1992), 37--40, Tata Inst. Fund. Res.,
Bombay, 1995.



\bibitem{BBD} A. A. Be\u{\i}linson, J. Bernstein, and P. Deligne, {\em Faisceaux pervers,} Analysis and topology on singular spaces, I (Luminy, 1981), 5--171, Ast\'erisque, 100, Soc. Math. France, Paris, 1982.

\bibitem{BR} I. Biswas and S. Ramanan, {\em An infinitesimal study of the moduli of
Hitchin pairs,} J. London Math. Soc. (2) 49 (1994), 219--231.





\bibitem{dCHM1} M. A. de Cataldo, T. Hausel, and L. Migliorini, {\em Topology of Hitchin systems and Hodge theory of character varieties: the case $A_1$,} Ann. of Math. (2) 175 (2012), no.~3, 1329--1407.



\bibitem{dC_SL} M.A. de Cataldo, {\em A support theorem for the Hitchin fibration: the case of $\mathrm{SL}_n$,} Compos.Math. 153 1316--1347.

\bibitem{Supp} M.A. de Cataldo, J. Heinloth, and L. Migliorini, {\em A support theorem for the Hitchin fibration:
the case of $GL_n$ and $K_C$,} arXiv 1906.09582.



\bibitem{dCMS} M. A. de Cataldo, D. Maulik, and J. Shen, {\em Hitchin fibrations, abelian surfaces, and the $P=W$ conjecture,} arXiv 1909.11885.

\bibitem{dCMS2} M.A. de Cataldo, D. Maulik, and J. Shen, {\em On the P=W conjecture
for $\mathrm{SL}_n$,} arXiv:2002.03336.





\bibitem{dCM0} M. A. de Cataldo and L. Migliorini, {\em The Hodge theory of algebraic maps,} Ann. Sci. \'Ecole Norm. Sup. (4) 38 (2005), no. 5, 693--750.


\bibitem{dCRS} M. A. de Cataldo, A. Rapagnetta, and G. Sacc\`a, {\em The Hodge numbers of O‘Grady 10 via Ng\^o strings,} arXiv:1905.03217.










\bibitem{CL} P.-H. Chaudouard, G. Laumon, {\em Un th\'{e}or\`{e}me du support pour la fibration de Hitchin,} Ann. Inst. Fourier, Grenoble, Tome 66, no 2 (2016), p. 711--727.





\bibitem{Dim} A. Dimca, {\em Sheaves in topology,} Springer--Verlag, Berlin, 2004.





\bibitem{EK} R. Earl, F. Kirwan, {\em Complete sets of relations in the cohomology rings of moduli spaces of holomorphic bundles and parabolic bundles over a Riemann surface,} Proc. London Math. Soc. (3) 89 (2004), no. 3, 570--622. 

\bibitem{Unramified} E. Franco, P.B. Gothen, A. Oliveira, A. Pe\'on-Netao, {\em Unramified covers and branes on the Hitchin system, } Adv. Math. 377, 22(2021).









\bibitem{GK} P.B. Gothen, A.D. King, {\em Homological algebra of twisted quiver bundles,} J. London Math.
Soc. 71 (2005), 85--99.


\bibitem{GWZ} M. Groechenig, D. Wyss, and P. Ziegler, {\em Mirror symmetry for moduli spaces of
Higgs bundles via p-adic integration,} Invent. Math. 221. 505--596(2020)

\bibitem{GWZ2} M. Groechenig, D. Wyss and P. Ziegler,
{\em Geometric stabilisation via $p$-adic integration,} J. Amer. Math. Soc. 33 (2020), no. 3, 807--873. 








\bibitem{HN} G. Harder, M. S. Narasimhan, {\em On the cohomology groups of moduli spaces of vector
bundles on curves,} Math. Ann. 212 (1974/75) 215--248.


\bibitem{HP} T. Hausel, C. Pauly, {\em Prym varieties of spectral covers,} Geom. Topol. 16 (2012), no. 3, 1609--1638. 




\bibitem{HT} T. Hausel, M. Thaddeus, {\em Mirror symmetry, Langlands duality, and the Hitchin
system,} Invent. Math. 153 (2003), no. 1, 197--229. 




\bibitem{HT2} T. Hausel, M. Thaddeus, {\em Relations in the cohomology ring of the moduli space of rank 2 Higgs bundles,} J. Amer. Math. Soc. 16 (2003), no. 2, 303--327.




\bibitem{Survey} T. Hausel, {\em Global topology of the Hitchin system,} Handbook of moduli. Vol. II, 29--69, Adv. Lect. Math. (ALM), 25, Int. Press, Somerville, MA, 2013.



\bibitem{Hit} N. J. Hitchin, {\em The self-duality equations on a Riemann surface,} Proc. London Math. Soc. (3) 55 (1987), no. 1, 59--126.

\bibitem{Hit1} N.J. Hitchin, {\em Stable bundles and integrable systems,} Duke Math. J. 54 (1987) 91--114.





\bibitem{Illusie} L. Illusie, {\em Vanishing cycles over general bases, after P. Deligne, O. Gabber, G. Laumon and F. Orgogozo.} Preprint. http://www.math.u-psud.fr/~illusie/vanishing1b.pdf. (2006).





\bibitem{Kir} F.C. Kirwan, {\em The cohomology rings of moduli spaces of bundles over Riemann surfaces,} J. Amer. Math. Soc. 5 (1992) 853--906.









\bibitem{LauNgo} G. Laumon and B.C. Ng\^{o}, {\em Le lemme fondamental pour les groupes unitaires,} Ann. of Math. 168(2008), 477--573.

\bibitem{LW} F. Loeser and D. Wyss, {\em Motivic integration on the Hitchin fibration, }preprint, arXiv:1912.11638v2.









\bibitem{Milnor} J. W. Milnor, {\em Singular points of complex hypersurfaces,} Annals of Mathematics Studies, No. 61, 1968.

\bibitem{Markman} E. Markman, {\em Generators of the cohomology ring of moduli spaces of sheaves on symplectic surfaces,} J. Reine Angew. Math. 544 (2002), 61--82. 
















\bibitem{Nit} N. Nitsure, {\em Moduli space of semistable pairs on a curve,} Proc. London Math. Soc. (3), 62(2):275--300, 1991.



\bibitem{Ngo0} B.C. Ng\^{o}, {\em Fibration de Hitchin et endoscopie,} Invent. Math. 164 (2006), no. 2, 399--453.

\bibitem{Ngo} B.C. Ng\^{o}, {\em Le lemme fondamental pour les alg\`{e}bres de Lie, } Publ. Math. IHES. 111 (2010) 1--169.






 





\bibitem{Saito} M. Saito, { \em Mixed Hodge modules,} Publ. Res. Inst. Math. Sci. 26 (1990), no. 2, 221--333. 













\bibitem{Simp} C. T. Simpson, {\em Higgs bundles and local systems,} Inst. Hautes \'Etudes Sci. Publ. Math. No. 75 (1992), 5--95.

\bibitem{Si1994II} C. T. Simpson, {\em Moduli of representations of the fundamental group of smooth projective varieties II,"} Publ. Math. IHES.  No. 80 (1994), 5--79.







\bibitem{Yun12} Z. Yun, {\em Global Springer theory,} Adv. Math. 228 (2011) 266--328.

\bibitem{Yun3} Z. Yun, {\em Towards a global springer theory III: endoscopy and Langlands duality,} preprint, arXiv:0904.3372.



\end{thebibliography}
\end{document}